\documentclass[12pt]{article}
\usepackage{amssymb}
\usepackage{amsmath}
\usepackage{color}
\usepackage{graphics}
\usepackage{epsfig}
\usepackage{hyperref}
\usepackage{amsthm}

\newtheorem{claim}{\bf \t}[part]

\newtheorem{Corollary}{Corollary}[part]
\newtheorem{Definition}{Definition}[part]

\newtheorem{Lemma}{Lemma}[part]
\newtheorem{Proposition}{Proposition}[part]
\newtheorem{Remark}{Remark}[part]
\newtheorem{Theorem}{Theorem}[part]

\numberwithin{Assumption}{section} \numberwithin{Corollary}{section}
\numberwithin{Definition}{section} \numberwithin{equation}{section}
\numberwithin{Example}{section} \numberwithin{Lemma}{section}
\numberwithin{Proposition}{section} \numberwithin{Remark}{section}
\numberwithin{Theorem}{section}

\def\t{\rho}

\def\text#1{{\rm #1}}
\begin{document}
	\date{}
	\title{\Large \bf Global well-posedness of 3D inhomogenous incompressible Navier-Stokes equations with density-dependent viscosity}
	\author{\small \textbf{Dongjuan Niu},$^{a,1}$	
		\thanks{The research is partially supported by the key research project of National Natural Science Foundation (Grant No. 11931010). E-mail: djniu@cnu.edu.cn.}\quad
		\textbf{Lu Wang},$^{a,2}$
		\thanks{E-mail: wlu1130@163.com.}}
	
	\maketitle
	\small $^a$ School of Mathematical Sciences, Capital Normal University, Beijing
	100048, P. R. China\\[2mm]

	{\bf Abstract:}  The issue of global well-posedness for the 3D inhomogenous incompressible Navier-Stokes equations was first addressed by Kazhikov \cite{1974K} in 1974. In this manuscript, we obtain its global well-posedness for the system with density-dependent viscosity under the smallness assumption of initial velocity in the critical space $\dot{B}^{-1+\frac 3p}_{p,1}$ with $p\in ]1,9/2]$. To the best of our knowledge, this is the first result about the global well-posedness for which one does not assume any smallness condition on the density when the initial density is far away from vacuum.

    \vskip0.3cm
	
	{\bf Key Words:} Inhomogeneous Navier-Stokes system, Global well-posedness, Time-decay estimates.
	
	{\bf AMS Subject Classification:} 35Q30, 76D03.
	
	\section{Introduction } \setcounter{equation}{0}
	\setcounter{Assumption}{0} \setcounter{Theorem}{0}
	\setcounter{Proposition}{0} \setcounter{Corollary}{0}
	\setcounter{Lemma}{0}
	
	In this paper, we are interested in  the global well-posedness for the following 3D inhomogeneous incompressible Navier-Stokes system:
	\begin{equation}\label{a1}
	\left\{
	\begin{array}{l}
	\partial_{t}\rho+\operatorname{div}(\rho u)=0,\quad (t,x)\in \mathbb{R}^{+}\times\mathbb{R}^{3},                \\
	\partial_{t}(\rho u)+\operatorname{div}(\rho u\otimes u)-\operatorname{div}(2\mu(\rho)\mathbb{D}u)+\nabla\pi=0, \\
	\operatorname{div} u=0,                                                                                         \\
	(\rho,u)|_{t=0}=(\rho_0,u_0),
	\end{array}
	\right.
	\end{equation}
	where $\rho$ and $u=(u_{1}(t,x),u_{2}(t,x),u_{3}(t,x))$ stand for the density and velocity field of the fluid respectively. $\mathbb{D}u=\frac{1}{2}(\partial_{i}u_{j}+\partial_j u_i) $ is the deformation tensor. $\pi=\pi(t,x)$ is a scalar pressure function, and in general, the viscosity coefficient $\mu(\rho)$ is a smooth, positive function on $[0,\infty)$. 
	
	Such a system is usually used to describe a fluid that is incompressible but has a non-constant density due to the complex structure of the flow caused by a mixture or contamination. More information related to the background of (\ref{a1}) can be found in \cite{1996L} and so on.

    There are a lot of literatures devoted to the global well-posedness of \eqref{a1}. Precisely, the global existence of weak solution with finite energy was first investigated by Kazhikov \cite{1974K} in the case of constant viscosity.  Lions \cite{1996L} dealt with the similar result including vaccum for system \eqref{a1} with density-dependent  viscosity. However, even for the two-dimensional case, the uniqueness and regularity of such weak solutions remain open. Later, Desjardins \cite{1997D} proved the global weak solution with higher regularity for the two-dimensional system, provided that the viscosity $\mu(\rho)$ is a small perturbation of a positive constant in $L^{\infty}$-norm. Under the similar assumption, Abidi-Zhang \cite{2015az} generalized this result to strong solutions in $\mathbb{R}^2$. Gui-Zhang \cite{2009GZ} proved the global well-posedness of (\ref{a1}) with initial density satisfying that $\|\rho_0-1\|_{H^{s+1}}$ is sufficiently small for some $s>2$.  Very recently, He-Li-L${\rm \ddot{u}}$ \cite{2021HLL}  obtained the uniquely global solution  of three-dimensional system when the initial density contains vacuum under the small assumption that
	\begin{equation}\label{aa8}
	\|u_0\|_{\dot{H}^{\beta}}\leq \varepsilon_{0}\quad{\rm with}\quad\beta\in ]1/2,1].
	\end{equation}
	More results can be referred to \cite{2004CK,2014HW,2015HW,2019LS,2015Z} and so on.

	As well known,  similar to the classical incompressible Navier-Stokes system, i.e., $\rho=constant$ in $(\ref{a1})$,   the inhomogenous Navier-Stokes system also has the property of scaling invariance. That is, if $(\rho, u)$ solves (\ref{a1}) with initial datum $(\rho_{0},u_{0})$, then for $\forall~\lambda>0$,
	\begin{equation}\label{aa5}
	(\rho,u)_{\lambda}(t,x)=(\rho(\lambda^{2}t,\lambda x),\lambda u(\lambda^{2}t,\lambda x))
	\end{equation}
	is a solution of (\ref{a1}) with initial data~$(\rho_{0}(\lambda x),\lambda u_{0}(\lambda x))$. 
	A functional space for data $(\rho_{0},u_{0})$ or for solution $(\rho,u)$ is said to be scaling-invariant of the equation if its norm is invariant under transformation (\ref{aa5}). In particular, $\dot{B}^{\frac{3}{p}}_{p,r}(\mathbb{R}^{3})\times \dot{B}^{\frac{3}{p}-1}_{p,r}(\mathbb{R}^{3})$ is a  scaling-invariant space for \eqref{a1} according to (\ref{aa5}).   In addition,  as indicated in \cite{1996L}, system \eqref{a1}  with the initial density  $\rho_0$ being far away from vacuum has the following property:
	\begin{equation}\label{aa7}
	 \operatorname{meas}\left\{x \in \mathbb{R}^N \mid \alpha \leq \rho(t, x) \leq \beta\right\} \text {~is~independent~of}~t \geq 0,
	\end{equation}
	for any $0\leq\alpha\leq\beta.$ 
	
	Based on the above two properties of system \eqref{a1}, a lot of effort was spent on the global well-posedness of 3D inhomogenous Navier-Stokes equations with constant viscosity in the critical spaces when the initial density has the positive lower bound. Danchin \cite{2003D} established the global well-posedness of (\ref{a1}) under the smallness assumptions on $\|\rho_0-1\|_{\dot{B}^{\frac{N}{p}}_{p, 1}}$ and $\|u_0\|_{\dot{B}^{\frac{N}{p}-1}_{p,1}}$ for any $p\in[1,N],$ where $N=2,3.$ Abidi-Gui-Zhang \cite{2012AGZ,2013AGZ} removed the size restriction of the initial density and proved the global well-posedness of (\ref{a1}) only provided that $u_{0}$ is small in $\dot{B}^{-1+\frac{3}{p}}_{p,1}(\mathbb{R}^3)$ with $p\in [3, 4].$
	Zhang \cite{2020z} proved the global existence of weak solutions to (\ref{a1}) when $u_0$ is small in $\dot{B}^{\frac{1}{2}}_{2,1}$ and $\rho_0\in L^\infty$ without any additional density regularity. Danchin-Wang \cite{2022a} generalized this result under the condition that $\rho_0$ is close to a positive constant and $u_{0}$ is small in $\dot{B}^{-1+\frac{3}{p}}_{p,1}(\mathbb{R}^3)$ with $p\in]1,3[$.
	For the case that $\rho_0$ is allowed to vanish, Craig-Huang-Wang \cite{2013CHW} established the global existence and uniqueness of strong solutions to (\ref{a1}) provided that $u_0$ is small in $\dot{H}^{\frac{1}{2}}(\mathbb{R}^3).$ More recently, Abidi-Gui-Zhang proved the global existence of 2D Navier-Stokes equations with constant viscosity in \cite{2023AGZ}, assuming that $u_{0}\in L^2\cap\dot{B}^{-1+\frac{2}{p}}_{p,1}(\mathbb{R}^2)$ with $p\in [2,+\infty[$. See \cite{2011AGZ,2004D,2012DM,2017j} and so on for more results  on the well-posedness of the Navier-Stokes system with constant viscosity in critical spaces. 
	
	For the case with density-dependent viscosity, it becomes complicated. Until now, there are only a few results on this topic. For two-dimensional system, Huang-Paicu-Zhang \cite{2013HPZ} proved that there exists a global and unique solution 
	$(\rho,u)$ satisfying that $\rho-1 \in \mathcal{C}_b([0, \infty) ; \dot{B}_{q, 1}^{\frac{2}{q}}\left(\mathbb{R}^2\right))$ and $u \in \mathcal{C}_b([0, \infty) ; \dot{B}_{p, 1}^{-1+\frac{2}{p}}\left(\mathbb{R}^2\right)) \cap L^1(\mathbb{R}^{+} ; \dot{B}_{p, 1}^{1+\frac{2}{p}}\left(\mathbb{R}^2\right))$ for $1<q\leq p<4$ provided that 
		\begin{equation}\label{b1}
	\eta \stackrel{\text { def }}{=}\left\|\rho_0-1\right\|_{\dot{B}_{q, 1}^{\frac{2}{q}}} \exp \left\{C_0\left(1+\mu^2(1)\right) \exp \left(\frac{C_0}{\mu^2(1)}\left\|u_0\right\|_{\dot{B}_{p, 1}^{-1+\frac{2}{p}}}^2\right)\right\} \leq \frac{c_0 \mu(1)}{1+\mu(1)},	
	\end{equation}
	for some sufficiently small $c_0.$ For the $N$-dimensional system, assuming that $p\in(1,2N)$ with $N=2,3$ and $ 0<\underline{\mu}<\mu(\rho),$ Abidi \cite{2007A} proved that  inhomogeneous incompressible Navier-Stokes system has a uniquely global solution under the smallness condition that
	\begin{equation}\label{b2}
	\|\rho_0-1\|_{\dot{B}^{\frac{N}{p}}_{p,1}}+\|u_0\|_{\dot{B}^{\frac{N}{p}-1}_{p,1}}\leq \varepsilon_0.
	\end{equation}
When the initial density is strictly away from vacuum, Abidi-Zhang \cite{2015AZ} 
obtained the global well-posedness  of solution to (\ref{a1}) under 
\begin{equation}\label{aa1}
	\|\mu(\rho_0)-\mu(1)\|_{L^\infty}\leq \varepsilon_{0} \quad {\rm and}\quad \|u_0\|_{L^2}\|\nabla u_0\|_{L^2}\leq \varepsilon_1,
\end{equation}
where $\varepsilon_0, \varepsilon_1$ are sufficiently small constants. In particular, we mention that the small assumption of $\|\mu(\rho_0)-\mu(1)\|_{L^{\infty}}$ in \eqref{aa1} is slightly weaker than that of $\|\rho_0-1\|_{\dot{B}^{\frac{N}{p}}_{p,1}}$ in \eqref{b2}. However, all the above global well-posedness results for (\ref{a1}) with density-dependent viscosity in critical Besov spaces have been obtained on the basis of smallness restriction on the intial density. Indeed, accounting for the kind of assumption on the fluctuation of intial density, the momentum equation $\eqref{a1}_2$ can be regarded as the inhomogenous Navier-Stokes system with constant viscosity in terms of some small source term, i.e.,
	\begin{align}
	\partial_{t}(\rho u)+\operatorname{div}(\rho u\otimes u)-\Delta u +\nabla\pi=\operatorname{div}(2(\mu(\rho)-1)\mathbb{D}u).
	\end{align}
Because of above-mentioned views, it is possible to obtain the gobal well-posedness for \eqref{a1} in the same spirit of \cite{2013AGZ}.
	
	In this manuscript, we are concerned with the global well-posedness of system \eqref{a1} in critical spaces under the case that the initial density is far away from vaccum. Compared with the previous results, we are devoted to removing the small assumption on the fluctuation of initial  density. In other words, we will prove the global existence and uniqueness of solutions to system \eqref{a1} with variable viscosity $\mu(\rho)$ under the small assumption of intial velocity fields $u_0$ in the critical space $\dot{B}^{-1+\frac 3p}_{p,1}$.  Obviously, without the small restriction on the fluctuation of initial density, it is inevitable to encounter a lot of new challenges. 

	Before we state our main result, we introduce our main difficulties and strategies here. In fact, now we are not able to borrow the idea in deriving the global well-posedness of 3D inhomogenous Navier-Stokes equations with constant viscosity as mentioned earlier. Our first strategy is to fullly utilize the regularity of Stokes equatons with variable coefficients (see Lemma \ref{Lemma-5.1} for details) instead of the classical Stokes equations. Specifically, we rewrite $(\ref{a1})_{2}$ as
	\begin{equation}\label{aa13}
	-\operatorname{div}(2\mu(\rho)\mathbb{D}u)+\nabla\pi=-\rho u_t-\rho (u\cdot \nabla) u,
	\end{equation} 	
which implies the regularity of the velocity fields is coupled with the uniform bound on the $L^{\infty}(0,\infty;L^{r})$-norm of $\nabla\mu(\rho)$ (see \eqref{d24} below) due to the appearance of the variable viscosity.
At the same time, the regularity theory of transport equation tells us that the uniform bound on the $L^{\infty}(0,\infty;L^{r})$-norm of $\nabla\mu(\rho)$ strongly depends on  the $L^{1}(0,\infty;L^\infty)$-norm of $\nabla u$, since  $\mu(\rho)$ satisfies the transport equation
	\begin{align}
	\partial_t \mu(\rho)+u\cdot \nabla \mu(\rho)=0.
	\end{align}
At this stage, according to the bootstrapping argument, we have to prove that the $L^{1}(0,\infty;L^\infty)$-norm of $\nabla u$ is small enough to close the above estimates instead of the uniform bound on the $L^{1}(0,\infty;L^\infty)$-norm of $\nabla u$ in the previous work, which bring us  the second essential difficulty. Our novel ingredient here is the decomposition of the velocity fields, i.e., we split $u$ into $v$ and $w$ (resp. $\nabla\pi$ into $\nabla\pi_v$ and $\nabla\pi_w$), where $(v,\nabla\pi_v)$ satisfies the classical 3D Navier-Stokes equations
\begin{equation}\label{a11}
\left\{\begin{array}{l}
\partial_t v+v \cdot \nabla v-\Delta v+\nabla \pi_v=0, \\
\operatorname{div} v=0,                                \\
\left. v\right|_{t=t_0}=u\left(t_0\right),
\end{array}\right.
\end{equation}
and $(\rho,w,\nabla\pi_w)$ satisfies
\begin{equation}\label{a12}
\left\{\begin{array}{l}
\partial_t \rho+\operatorname{div}(\rho(v+w))=0,                                                                         \\
\rho \partial_t w+\rho(v+w) \cdot \nabla w-\operatorname{div}(2\mu(\rho)\mathbb{D}w)+\nabla \pi_w                        \\
\quad=(1-\rho)\left(\partial_t v+v \cdot \nabla v\right)-\rho w \cdot \nabla v+\operatorname{div}(2(\mu(\rho)-1)\mathbb{D}v), \\
\operatorname{div} w=0,                                                                                                  \\
\left.\rho\right|_{t=t_0}=\rho\left(t_0\right),\left.\quad w\right|_{t=t_0}=0.
\end{array}\right.
\end{equation}
As shown in \eqref{a11} and \eqref{a12}, we solve the two systems from the time $t_0$ instead of $0$. Specifically, according to the local well-posedness theory of \cite{2022q}, there exists a local solution $(\rho, u, \nabla\pi)$ of the system \eqref{a1} and some $t_0\in (0,1)$ such that
$$
\|u(t_0)\|_{\dot{B}^{-1+\frac{3}{p}}_{p,1}\cap\dot{B}^{1+\frac{3}{p}}_{p,1}}\leq C \|u_0\|_{ \dot{B}^{-1+\frac{3}{p}}_{p,1}}.
$$
Obviously, $u(t_0)$ possesses the more regularity than $u_0$ due to the smooth effect of the paraboic system.
For system \eqref{a11}, we easily know from \cite{2013AGZ} that for $s\in[-1+\frac{3}{p},1+\frac{3}{p}],$
\begin{equation}\label{bbb1}
\|v\|_{\widetilde{L}^{\infty}([t_0,\infty);\dot{B}^{s}_{p,1})}+\| v\|_{L^1([t_0,\infty);\dot{B}^{s+2}_{p,1})}\leq C\|u_0\|_{\dot{B}_{p,1}^{-1+\frac{3}{p}}}.
\end{equation}
As for system \eqref{a12}, it is more complicated than the original system. Taking advantage of the zero initial velocity and the higher regularity of $v$ implies that we have
\begin{equation}\label{bbbbb1}
\begin{aligned}
\|\nabla w\|_{L^{\infty}([t_0,\infty);L^2)}
\leq C\|v\|_{L^2([t_0,\infty);\dot{B}^{s+1}_{p,1})}\leq C \|u_0\|_{ \dot{B}^{-1+\frac{3}{p}}_{p,1}},
\end{aligned}
\end{equation}
which contributes to the $L^{\infty}([t_0,\infty);L^2)$-norm of $\nabla u.$ Based on this, we assumed that the initial velocity $u_0$ belongs to a negative Sobolev space $H^{-2\delta}$ in order to obtain the time-decay estimates of the velocity fields $u$ with different temporal and spatial norms. In fact, we emphasize that it is impossible to obtain the desired time-decay estimates of the velocity fields $u$ without the help of the decomposition of the velocity fields. See section 4 for more details. 

\par In the following, we concentrate on verifying the smallness of the $L^{1}(t_0,\infty;L^\infty)$-norm of $\nabla u$. In particular, by means of the interpolation theorem and time-decay estimates of $u$, it is sufficient to prove instead the smallness of the gradient of the velocity fields with respect to the $L^2(t_0,\infty;L^6)$-norm, which plays the key role in obtaining the global well-posedness of the solution for \eqref{a1} with variable viscosity. Roughly speaking, we expect to prove that
\begin{equation}\label{bbbb2}
\|\nabla u\|_{L^2([t_0,\infty);L^6)}\leq \|\nabla^2 w\|_{L^2([t_0,\infty);L^2)}+\|\nabla v\|_{L^2([t_0,\infty);L^6)}\leq C \|u_0\|_{ \dot{B}^{-1+\frac{3}{p}}_{p,1}}.
\end{equation}
In fact, according to \eqref{bbb1} and the embedding theorem, $\dot{B}^{\frac 3p+\frac 12}_{p,1}\hookrightarrow \dot{B}^{1}_{6,1}$, implies that
\begin{equation}\label{bbb2}
\|\nabla v\|_{L^2([t_0,\infty);L^6)}\leq C\|v\|_{\widetilde{L}^2([t_0,\infty);\dot{B}^{\frac 3p+\frac12}_{p,1})}\leq C\|u_0\|_{\dot{B}_{p,1}^{-1+\frac{3}{p}}}.
\end{equation}
The estimate
$$
\|\nabla^2 w\|_{L^2([t_0,\infty);L^2)}\leq C \|u_0\|_{ \dot{B}^{-1+\frac{3}{p}}_{p,1}},
$$
comes from the second-order estimates of $w$ on the basis of \eqref{bbbbb1}.  More details can be found in section 4. Besides, we mention that the propogation of  the negative Sobolev space also brings some  additional constraits such as the maximum value of $p$ is $\frac92$ instead of $6$ compared with  the local well-posedness work of \cite{2022q}. See Proposition \ref{Proposition-4.3} for more information.



	Throughout this paper, we shall always assume that $\rho_0$ having a positive lower bound and
	\begin{equation} \label{a6}
	0<\underline{\mu}\leq\mu(\rho_0)\leq\overline{\mu}, \quad \mu(\cdot)\in\mathrm{W}^{3,\infty}(\mathbb{R}^{+})\quad{\rm  and }\quad\mu(1)=1.
	\end{equation}
	\par The main result of this paper states as follows:
	
	\begin{Theorem}\label{Theorem-1} Let $q\in[1,2]$ and $p\in]1,9/2]$ with $\max\{\frac{1}{p},1-\frac{2}{p}\}\leq\frac{1}{q}\leq\frac{1}{p}+\frac{1}{3}.$ Assume that the initial data $(\rho_{0},u_{0})$ satisfying 
		$$\rho_{0}-1\in B^{\frac{3}{q}}_{q,1}(\mathbb{R}^{3}),~\nabla\mu(\rho_0)\in L^{r}(\mathbb{R}^{3}),~u_{0}\in \dot{B}^{-1+\frac{3}{p}}_{p,1}\cap\dot{H}^{-2\delta}(\mathbb{R}^{3}),$$
		where $r\in ]3,9[$ and $\delta\in]1/2,3/4[.$ Then there exists a small positive constant $\varepsilon$ depending on $\|\rho_{0}-1\|_{B^{\frac{3}{q}}_{q,1}},$ $\|\nabla\mu(\rho_0)\|_{L^{r}}$ and $\|u_0\|_{\dot{H}^{-2\delta}}$, which satisfies
		\begin{equation}\label{a7}
		\|u_{0}\|_{\dot{B}^{-1+\frac{3}{p}}_{p,1}}\leq\varepsilon,
		\end{equation}
	then the Cauchy problem $(\ref{a1})$ admits a unique and global strong solution $(\rho,u,\nabla\pi)$ with
	\begin{equation}
	\rho-1 \in \mathcal{C}_{b}\bigg(\mathbb{R}^{+} ;B^{\frac{3}{q}}_{q,1}\bigg)~and~
	u \in \mathcal{C}_{b}\bigg(\mathbb{R}^{+}; \dot{B}^{-1+\frac{3}{p}}_{p,1}\bigg)\cap L^{1}_{loc}\bigg(\mathbb{R}^{+}; \dot{B}^{1+\frac{3}{p}}_{p,1}\bigg).
	\end{equation}
	Moreover, 
	 there exists a positive time $t_0$ such that 
	\begin{equation}\label{aa18}
	\begin{aligned}
	&\sup_{t\in[t_{0},T]}\int^{T}_{t_{0}}t^{2\delta}|u|^2dx+\int^{T}_{t_{0}}\int_{\mathbb{R}^{3}}t^{2\delta_{-}}|\nabla u|^2dxdt+\int^{T}_{t_{0}}\int_{\mathbb{R}^{3}}t|u_{t}|^{2}dxdt\\
	&\quad+\sup_{t\in[t_{0},T]}\int_{\mathbb{R}^{3}}t|\nabla u|^{2} dx+\sup_{t\in[t_{0},T]}\int_{\mathbb{R}^{3}}t^2|u_{t}|^{2}dx+\int^{T}_{t_{0}}\int_{\mathbb{R}^{3}}t^{2}|\nabla u_{t}|^{2}dxdt\leq \widetilde{C},
	\end{aligned}
	\end{equation}
	where $\widetilde{C}$ depends only on $\underline{\mu}$, $\overline{\mu}$, $p$, $q$, $\|\nabla\mu(\rho_0)\|_{L^{r}}$ and $\|u_0\|_{\dot{H}^{-2\delta}}.$ 
\end{Theorem}
	\begin{Remark}
		In fact, in Theorem \ref{Theorem-1} we show more detailed information about this unique and global solution. In particular, as $3<p\leq\frac{9}{2},$ Theorem \ref{Theorem-1} implies the global well-posedness of (\ref{a1}) with data $(\rho_{0}-1,u_0^{\epsilon})$ for $\varepsilon$ sufficiently small, where $\rho_0-1$ satisfies the assumptions of Theorem \ref{Theorem-1} and
		\begin{equation}
		u_0^{\varepsilon}=\sin(\frac{x_1}{\varepsilon})(0,-\partial_3 \phi,\partial_2 \phi), \quad {\rm for~any}~\phi\in\mathcal{S}(\mathbb{R}^3).
		\end{equation}
	\end{Remark}
\begin{Remark}
	When the initial density has a positive lower bound, (\ref{a1}) is analogous to a quasilinear parabolic system, where we do not have the exponential decay properties stated as He-Li-L${\rm \ddot{u}}$ \cite{2021HLL}. Therefore the assumption $u_0$ belonging to $\dot{H}^{-2\delta}$ of theorem \ref{Theorem-1} is necessary, which inherits the algebraic time-decay estimate (\ref{aa18}).
\end{Remark}
\begin{Remark}
In our recent work, we prove the similar result when $p=2$ in \cite{NW2023NS}.
\end{Remark}

	\par Let us complete this section with the notations we are going to use in this context.
	
	\par \textbf{Notations} We always denote $a_{-}$ to be any number strictly less than $a$. Let $A,$ $B$ be two operators, we denote $[A, B]=AB-BA,$ the commutator between $A$ and $B.$ For $a\lesssim b$, we mean that there is a uniform constant $C$, which may be different on different lines, such that $a\leq Cb$. We shall denote by $(a|b)$ the $L^2(\mathbb{R}^3)$ inner product of $a$ and $b.$ For $X$ a Banach space, we denote $\|(f,g)\|_{X}\stackrel{\Delta}{=}\|f\|_{X}+\|g\|_{X}.$ Finally, $(c_{j})_{j\in\mathbb{Z}}$ (resp. $(d_{j})_{j\in\mathbb{Z}}$) will be a generic element of $\ell^{2}(\mathbb{Z})$ (resp. $\ell^{1}(\mathbb{Z})$) so that $\sum_{j\in\mathbb{Z}}c^{2}_{j}=1$ (resp. $\sum_{j\in\mathbb{Z}}d_{j}=1$).

    \par This paper is organized as follows. In  section 2, we recall some basic Littlewood-Paley theory, as well as some necessary lemmas. We obtain the local-in-time existence  of system \eqref{a1} in section 3. With these estimates, we shall prove the global well-posedness of solution to system (\ref{a1}) in section 4.
	
\section{Littlewood-Paley theory and preliminary estimates}
In this section, we introduce some notations and conventions, and recall some standard lemmas of Besov space, which will be used throughout this paper, as well as some necessary linear estimates for density-dependent Stokes equations.

\subsection{Basic facts on Littlewood-Paley theory}
Since the proof of Theorem \ref{Theorem-1} requires a dyadic decomposition of the Fourier variables, we shall recall the Littlewood-Paley decomposition, such as the commutator's estimates. See \cite{2011BCD} for more details.

\begin{Definition}\label{Definition-1} (\cite{2011BCD}) Let $(p,r)\in[1,\infty]^{2},~s\in\mathbb{R}$ and $u\in\mathcal{S}'(\mathbb{R}^{3}),$ we set
	\begin{equation}\label{a13}
	\|u\|_{B^{s}_{p,r}}\stackrel{\mathrm{ def }}{=}(2^{qs}\|\Delta_{q}u\|_{L^{p}})_{l^{r}},\quad \|u\|_{\dot {B}^{s}_{p,r}}\stackrel{\mathrm{ def }}{=}(2^{qs}\|\dot{\Delta}_{q}u\|_{L^{p}})_{l^{r}}.
	\end{equation}
\end{Definition}

\begin{Definition}\label{Definition-2} (\cite{2011BCD}) Let $(p,\lambda,r)\in[1,\infty]^{3},$ $s\in\mathbb{R},$ $T\in(0,+\infty],$ and $u\in\mathcal{S}'(\mathbb{R}^{3}),$ we set
	\begin{equation}\label{a14}
	\|u\|_{\widetilde{L}^{\lambda}_{T}(B^{s}_{p,r})}\stackrel{\mathrm{ def }}{=}(2^{qs}\|\Delta_{q}u\|_{L^{\lambda}_{T}(L^{p})})_{l^{r}},\  \|u\|_{\widetilde{L}^{\lambda}_{T}(\dot{B}^{s}_{p,r})}\stackrel{\mathrm{ def }}{=}(2^{qs}\|\dot{\Delta}_{q}u\|_{L^{\lambda}_{T}(L^{p})})_{l^{r}}.
	\end{equation}
\end{Definition}

\begin{Lemma}\label{lemma-2.1} $(\cite{2011BCD})$ Let $\mathcal{C}$ be the annulus $\{\xi\in\mathbb{R}^{N}:3/4\leq|\xi|\leq 8/3\}.$ There exist radial functions $\chi$ and $\varphi$, valued in the interval $[0,1],$ belonging respectively to $\mathcal{D}(\mathcal{C})$, and such that
	\begin{equation}\label{a15}
	\begin{aligned}
	& \forall~\xi\in\mathbb{R}^{N},~\chi(\xi)+ \sum_{j\geq 0} {\varphi(2^{-j\xi})}=1,\\
	& \forall~\xi\in\mathbb{R}^{N},~ \sum_{j\in\mathbb{Z}}{ \varphi(2^{-j\xi})}=1,\\
	& |j-j'|\geq 2\Longrightarrow {\rm Supp}\varphi(2^{-j\xi\cdot})\cap{\rm Supp}\varphi(2^{-j'\xi\cdot})=\varnothing,\\
	& j\geq 1\Longrightarrow {\rm Supp}\chi\cap{\rm Supp}\varphi(2^{-j'\xi\cdot})=\varnothing.
	\end{aligned}
	\end{equation}
	the set $\widetilde{\mathcal{C}}=B(0,2/3)+\mathcal{C}$~is an annulus, and we have
	\begin{equation}\label{a16}
	|j-j'|\geq 5\Longrightarrow 2^{j'}\widetilde{\mathcal{C}}\cap 2^{j}\mathcal{C}=\varnothing.
	\end{equation}
\end{Lemma}

\begin{Lemma}\label{Lemma-2.2} $(\cite{2011BCD})$ 
	Let $\mathcal{C}$ be an annulus and $\mathcal{B}$ a ball. There exists a constant $C>0$ such that for any nonnegative integer $k,$ any couple  $(p,q)$ in $[1,\infty]^{2}$ with $q\geq p\geq 1$ and any $u$ of $L^{p}$ satisfying
	\begin{enumerate}
		\item If ${\rm Supp} \ \hat{u}\subset\lambda\mathcal{B},$ we have
		\begin{equation}\label{a17}
		\sup_{|\alpha|=k}\|\partial^{\alpha}u\|_{L^{q}}\leq C^{k+1}\lambda^{k+N(\frac{1}{p}-\frac{1}{q})}\|u\|_{L^{p}}.\end{equation}
		\item If ${\rm Supp} \ \hat{u}\subset\lambda\mathcal{C},$ we have
		\begin{equation}\label{a18} C^{-k-1}\lambda^{k}\|u\|_{L^{p}}\leq \|D^{k}u\|_{L^{p}}\leq C^{k+1}\lambda^{k}\|u\|_{L^{p}}.
		\end{equation}
	\end{enumerate} 
\end{Lemma}

\par In the rest of the paper, we shall frequently use homogeneous Bony's decomposition:
\begin{equation}\label{a19}
uv=T_{u} v+T'_{v}u=T_{u} v+T_{v} u+R(u,v),
\end{equation}
where
\begin{equation}\label{a20}
\begin{aligned}
& T_{u}v= \sum_{q\in\mathbb{Z}}\dot{S}_{q-1}u\dot{\Delta}_{q}v,\quad T'_{v}u=\sum_{q\in\mathbb{Z}}\dot{S}_{q+2}v\dot{\Delta}_{q}u,\\
& R(u,v)=\sum_{q\in\mathbb{Z}}\dot{\Delta}_{q}u\widetilde{\dot{\Delta}}_{q}v,\quad \widetilde{\dot{\Delta}}_{q}v=\sum_{|q'-q|\leq1}\dot{\Delta}_{q'}v.
\end{aligned}
\end{equation}
Similarly, we can obtain inhomogeneous Bony's decomposition \cite{2011BCD}.

\begin{Lemma}\label{Lemma-2.3} $(\cite{2011BCD})$
	Let $s \in \mathbb{R},$ $1 \leq r \leq \infty$, and $1 \leq p \leq p_1 \leq \infty$. Let $v$ be a vector field over $\mathbb{R}^d$. Assume that
	$$
	s>-d \min \left\{\frac{1}{p_1}, \frac{1}{p^{\prime}}\right\} \quad \text { or } \quad s>-1-d \min \left\{\frac{1}{p_1}, \frac{1}{p^{\prime}}\right\} \quad \text { if } \quad \operatorname{div} v=0,	$$
	where $p^{\prime}$ is the conjugate exponent of $p$. There exists a constant $C$, depending continuously on $p,$ $p_1,$ $s$ and $d$, such that
	\begin{equation}
	\left\|\left(2^{j s}\left\|[v \cdot \nabla; \Delta_j] f\right\|_{L^p}\right)_j\right\|_{\ell^r} \lesssim \|\nabla v\|_{B_{p_1, \infty}^{\frac{d}{p_1}} \cap L^{\infty}}\|f\|_{B_{p, r}^s} \quad\text { if }\quad s<1+\frac{d}{p_1}.
	\end{equation}
	In the limit case $s=-\min (\frac{d}{p_1}, \frac{d}{p^{\prime}})$ $($or $s=-1-\min (\frac{d}{p_1}, \frac{d}{p^{\prime}})$, if $\operatorname{div} v=0$$)$, we have
	\begin{equation}
	\sup _{j \geq-1} 2^{j s}\left\|[v \cdot \nabla; \Delta_j] f\right\|_{L^p} \lesssim \|\nabla v\|_{B_{p_1, 1}^{\frac{d}{p_1}}}\|f\|_{B_{p, \infty}^s}.
	\end{equation}
\end{Lemma}

\par Then we shall derive some conclusions about the transport equation, which will be frequently used throughout the succeeding sections. The proof process can be referred to \cite{2011BCD}.

\begin{Lemma}\label{Lemma-2.4} $(\cite{2011BCD})$
	Let $(p, q) \in[1, \infty]^2,$ $-1-\frac{3}{p}<s<1+3 \min \{\frac{1}{p}, \frac{1}{q}\}$ if $\frac{1}{p}+\frac{1}{q} \leq 1,$ or $-1-\frac{3}{q^{\prime}}<s<1+3 \min \{\frac{1}{p}, \frac{1}{q}\}$ if $\frac{1}{p}+\frac{1}{q} \geq 1,$ and $r \in[1, \infty]$ $($resp. $s=1+3 \min \left\{\frac{1}{p}, \frac{1}{q}\right\}$ with $r=1$$),$ where $q^{\prime}$ is the conjugate exponent of $q$. Let $v$ be a divergence free vector field with $\nabla v \in L^1 (0, T; \dot{B}_{p, r}^{\frac{3}{p}} \cap L^{\infty}(\mathbb{R}^3))$ $($resp. $v \in L^1 (0, T; \dot{B}_{p,1}^{1+\frac{3}{p}} (\mathbb{R}^3))$$).$ Given $a_0 \in \dot{B}_{q, r}^s$ and $f \in L^1 (0, T;$ $\dot{B}_{q, r}^s (\mathbb{R}^3)),$ the following transport equation
	\begin{equation}
	\left\{\begin{array}{l}
	\partial_t a+u \cdot \nabla a=f \\
	\left.a\right|_{t=0}=a_0
	\end{array}\right.
	\end{equation}
	has a unique solution $a \in \mathcal{C}([0, T] ; \dot{B}_{q, r}^s(\mathbb{R}^3)).$ Moreover, there holds for all $t\in[0,T]$
	\begin{equation}\label{a-equation}
	\|a\|_{\widetilde{L}_t^{\infty}(\dot{B}_{q, r}^s)} \leq\|a_0\|_{\dot{B}_{q, r}^s}+C\bigg\{\int_0^t\|a(\tau)\|_{\dot{B}_{q, r}^s} V^{\prime}(\tau) d \tau+\|f\|_{L_t^1(\dot{B}_{q, r}^s)}\bigg\},
   \end{equation}
	where $V(t) \stackrel{\text { def }}{=} \int_0^t\|\nabla v\|_{\dot{B}_{p, r}^{\frac{3}{p}} \cap L^{\infty}} d \tau$ $($resp. $V(t) \stackrel{\text { def }}{=} \int_0^t \|v\|_{\dot{B}_{p,1}^{1+\frac{3}{p}}} d \tau$$).$ Similar inequality holds for the inhomogeneous Besov spaces.
\end{Lemma}

\subsection{Regularity results}
\par  The purpose of this subsection is to give the \emph{a priori} estimates for smooth solutions of the linearised equations. In order to derive the higher order \emph{a priori} estimates, we first recall the regularity results for the density-dependent Stokes equations as follows. 
\begin{equation}\label{d23}
\left\{\begin{array}{l}
-\operatorname{div} (2\mu(\rho)\mathbb{D}U)+\nabla\Pi=F,\ \ \ \ \ \ \ \ \ \ \  (t,x)\in\mathbb{R}_{+}\times\mathbb{R}^{3}, \\
\operatorname{div} U=0, \ \ \ \ \ \ \ \ \ \ \ \ \ \ \ \ \ \ \  \ \ \ \ \ \ \ \ \ \ \ \ \ \ \ (t,x)\in\mathbb{R}_{+}\times\mathbb{R}^{3}.
\end{array}\right.
\end{equation}	

\begin{Lemma}\label{Lemma-5.1} $(\cite{NW2023NS})$
	For positive constants $\underline{\mu}, \bar{\mu},$ and $r\in(3,\infty),$ in addition, assume that $\mu(\rho)$ satisfies 
	\begin{equation}\label{d22}
	\nabla\mu(\rho)\in L^{r},\quad 0<\underline{\mu}\leq\mu(\rho)\leq\bar{\mu}<\infty.
	\end{equation}
	Then, if $F\in L^{m}$ with $m\in[\frac{2r}{r+2},r],$ there exists some positive $C$ depending only on $\underline{\mu},\bar{\mu},r$ and $m$ such that the unique weak solution $(U,\Pi)\in H^{1}\times L^{2}$ to the Cauchy problem (\ref{d23}) satisfies
	\begin{equation}\label{d24}
	\|\nabla^{2} U\|_{L^{2}}\leq C \|F\|_{L^{2}}+C\|\nabla \mu(\rho)\|^{\frac{r}{r-3}}_{L^{r}}\|\nabla U\|_{L^{2}},
	\end{equation}
	and
	\begin{equation}\label{d25}
	\|\nabla^{2} U\|_{L^{m}}\leq C\|F\|_{L^{m}}+C\|\nabla\mu(\rho)\|^{\frac{r(5m-6)}{2m(r-3)}}_{L^{r}}\Bigg\{\|\nabla U\|_{L^{2}}+\|(-\Delta)^{-1}\operatorname{div}F\|_{L^{2}}\Bigg\}.
	\end{equation}
\end{Lemma}
\par More information can be found in \cite{NW2023NS}, we omit it here for simplicity. Then we introduce auxiliary corollaries, which are a consequence of the
	$\dot{W}^{2,2}$ estimates in Lemma \ref{Lemma-5.1}. Assuming that 
	\begin{equation}\label{mu}
	\sup_{t\in[0,T]}\|\nabla\mu(\rho)\|_{L^{r}\cap L^{3}}\leq 4\|\nabla\mu(\rho_0)\|_{L^{r}\cap L^3}
	\end{equation}
	is valid, then in the light of the above inequality, the following corollaries naturally follow, which will be used repeatedly in the proof of the global existence of the solution.
	
	\begin{Corollary}\label{Corollary-5.1}
	Suppose $(\rho,u,\pi)$ is a unique local strong solution of (\ref{a1}) on $[0,T]$ and satisfies (\ref{mu}), then it holds that
	\begin{equation}\label{dd36}
	\|\nabla^{2} u\|_{L^{2}}\leq C \|\rho u_t\|_{L^{2}}+C\|\nabla u\|_{L^{2}}+C\|\nabla u\|^{3}_{L^{2}},
	\end{equation}
	where $C$ depends only on $\underline{\mu}$, $\overline{\mu}$, $r$ and $\|\nabla\mu(\rho_0)\|_{L^{r}}$.
	\end{Corollary}
    \begin{proof}
    The momentum equations $(\ref{a1})_2$ can be rewritten as follows,
    \begin{equation}
    -\operatorname{div}(2\mu(\rho)\mathbb{D}u)+\nabla\pi=F
    \end{equation}
    with $F\stackrel{\text { def }}{=}-\rho u_t-\rho u\cdot\nabla u.$ Due to Lemma \ref{Lemma-5.1}, we achieve
    \begin{equation}
    \begin{aligned}
    \|\nabla^{2} u\|_{L^{2}}
    &\lesssim \|\rho u_t\|_{L^{2}}+\|\nabla u\|_{L^{2}}+\|u\|_{L^6}\|\nabla u\|_{L^3}\\
    &\lesssim \|\rho u_t\|_{L^{2}}+\|\nabla u\|_{L^{2}}+\|\nabla u\|^{\frac{3}{2}}_{L^2}\|\nabla^2 u\|^{\frac{1}{2}}_{L^2}.
    \end{aligned}
    \end{equation}
    Thus, by virtue of Young's inequality, we get (\ref{dd36}).	
    \end{proof}
\begin{Corollary}\label{c-2.1}
	Let $(v,\pi_v)$ be a unique local strong solution of (\ref{a11}) on $[0,T]$ with $v\in L^\infty_{T}(\dot{B}^{s_1}_{p,1})$ $s_1\in[\frac{3}{p},1+\frac{3}{p}]$ and $\Delta v\in L^\infty_{T}(L^\infty).$ Let $(\rho,w,\pi_w)$ is a unique local strong solution of (\ref{a12}) on $[0,T]$ and satisfies (\ref{mu}), then it holds that
	\begin{equation}\label{d36}
	\|\nabla^{2} w\|_{L^{2}}\leq C\|\rho w_t\|_{L^{2}}+C\|\nabla w\|_{L^{2}}+C\|\nabla w\|^{3}_{L^{2}}+C\|\Delta v\|_{L^\infty}+C\|v\|_{\dot{B}^{s_1}_{p,1}},
	\end{equation} 
	where $C$ depends only on $\underline{\mu}$, $\overline{\mu}$, $r$ and $\|\nabla\mu(\rho_0)\|_{L^{r}\cap L^3}$. 
\end{Corollary}  
\begin{proof} 
	\par We rewrite $(\ref{a12})_2$ as follows,
	\begin{equation}\label{d37}
	\begin{aligned}
	-\operatorname{div}(\mu(\rho)\mathbb{D} w)+\nabla \pi=F
	\end{aligned}
	\end{equation}
	with
	$$F\stackrel{\text { def }}{=}-\rho\partial_{t} w-\rho(v+w)\cdot\nabla w-\rho w\cdot\nabla v+(1-\rho)(\partial_{t}v+v\cdot\nabla v)+ \operatorname{div}\left(2(\mu(\rho)-1)\mathbb{D}v\right),$$
	from which and using Lemma \ref{Lemma-5.1}, we get that
	\begin{equation}
	\begin{aligned}
	\|\nabla^{2}w\|_{L^{2}}&\lesssim \|\nabla w\|_{L^{2}}+\|\rho\partial_{t} w\|_{L^{2}}+\|\rho(v+w)\cdot\nabla w\|_{L^{2}}+\|\rho w\cdot\nabla v\|_{L^{2}}\\ &\quad+\|(1-\rho)(\partial_{t}v+v\cdot\nabla v)\|_{L^{2}}+\|\operatorname{div}\left(2(\mu(\rho)-1)\mathbb{D}v\right)\|_{L^{2}}.
	\end{aligned}
	\end{equation}
	Thanks to the Gagliardo-Nirenberg inequality, one has
	\begin{equation}\label{d38}
	\begin{aligned}
	\|\nabla^{2}w\|_{L^{2}}
	&\lesssim \|\nabla w\|_{L^{2}}+\|\rho\partial_{t} w\|_{L^{2}}+\|\nabla w\|^{\frac{3}{2}}_{L^2}\|\nabla^{2} w\|^{\frac{1}{2}}_{L^{2}}\\
	&\quad+\|v\|_{L^\infty}\|\nabla w\|_{L^{2}}+\|1-\rho_0\|_{L^2}\left\{\|\partial_{t} v\|_{L^{\infty}}+\|v\|_{L^\infty}\|\nabla v\|_{L^{\infty}}\right\}\\
	&\quad+\|\mu(\rho_0)-1\|_{L^2}\|\Delta v\|_{L^\infty}+\| \nabla\mu(\rho)\|_{L^3}\|\nabla v\|_{L^{6}}.  
	\end{aligned}
	\end{equation}
	By Young's inequality and embedded inequalities, for $s_1\in[\frac{3}{p},1+\frac{3}{p}],$ we can deduce 
	$$ \|\nabla^{2} w\|_{L^{2}}\lesssim \|\rho\partial_{t} w\|_{L^{2}}+\|\nabla w\|_{L^{2}}+\|\nabla w\|^{3}_{L^{2}}+\|\Delta v\|_{L^\infty}+\|v\|_{\dot{B}^{s_1}_{p,1}},$$
	where we have used the following inequality, 
	$$\begin{aligned}
	\|\partial_{t} v\|_{L^{\infty}}\lesssim &\|v\cdot\nabla v\|_{L^{\infty}}+\|\Delta v\|_{L^{\infty}}+\|\nabla\pi_{v}\|_{L^{\infty}}
	\lesssim \|v\cdot\nabla v\|_{L^{\infty}}+\|\Delta v\|_{L^{\infty}}\\
	\lesssim & \|v\|_{L^{\infty}}\|\nabla v\|_{L^{\infty}}+\|\Delta v\|_{L^{\infty}}\lesssim \|v\|_{\dot{B}^{s_1}_{p,1}}+\|\Delta v\|_{L^{\infty}}.
	\end{aligned}$$
	\end{proof}

\par Then we shall derive some conclusions about the estimate of the pressure function and the proof process can be referred to \cite{2017B,2010D}. We omit it here for simplicity.

\begin{Lemma}\label{pi} $(\cite{2017B, 2010D})$
    Let $\frac{6}{5}<p<6$ and $1 \leq q<\infty$ satisfy $\frac{1}{2}-\frac{1}{q}\leq\frac{1}{p}\leq\frac{1}{2}+\frac{1}{q}$, $a \in \dot{B}_{q,1}^{\frac{3}{q}}\left(\mathbb{R}^3\right)$ with
	$$
	1+a \geq \underline{a}>0.
	$$
	Let $f \in \dot{B}_{p,2}^{\frac{3}{p}-\frac{3}{2}}(\mathbb{R}^3)$ and $\nabla \pi \in \dot{B}_{p,2}^{\frac{3}{p}-\frac{3}{2}}(\mathbb{R}^3)$ solve the equation
	\begin{equation}
	-\operatorname{div}((1+a) \nabla \pi)=\operatorname{div} f.
	\end{equation}
	Then we have
	\begin{equation}\label{p2}
	\|\nabla \pi\|_{\dot{B}_{p, 2}^{\frac{3}{p}-\frac{3}{2}}} \lesssim(1+\|a\|_{\dot{B}_{q,1}^{\frac{3}{q}}})\|\mathbb{Q} f\|_{\dot{B}_{p,2}^{\frac{3}{p}-\frac{3}{2}}},
	\end{equation}
	where $\mathbb{Q}=I d-\mathbb{P}$ and $\mathbb{P}$ is the Leray projector over divergence-free vector fields.
\end{Lemma}

\section{Local well-posedness of (\ref{a1})}
In this section, we will state the local well-posedness of system \eqref{a1}. Compared with the original system, we intend to study the alternatively equivalent system. Specifically, since the density $\rho_0$ is away from zero, we denote by $a\stackrel{\text{def}}{=}\rho^{-1}-1$ and $\widetilde{\mu}(a)\stackrel{\text{def}}{=}\mu(\rho)$, then the system (\ref{a1}) can be equivalently reformulated as
\begin{equation} \label{a2}
\left\{\begin{array}{l}
\partial_{t}a+u\cdot\nabla a=0,\quad (t,x)\in \mathbb{R}^{+}\times\mathbb{R}^{3},\\
\partial_{t}u+u\cdot\nabla u-{\rm  div }(2(1+b)\mathbb{D} u)+(1+a)\nabla\pi=-\nabla\lambda~\mathbb{D}u,\\
\operatorname{div} u=0,\\
(a,u)|_{t=0}=(a_{0},u_{0}),
\end{array}\right.
\end{equation}
where
\begin{equation} \label{a3}
b \stackrel{\text { def }}{=}(1+a) \tilde{\mu} (a)-1 \quad \text { and } \quad \lambda \stackrel{\text { def }}{=}2 \int_0^{a} \tilde{\mu}(s)ds.
\end{equation} 
For simplicity, we first start with the local well-posedness theory of Theorem 3.1 of \cite{2022q} in the following.
\begin{Lemma}\label{Lemma-2} $(\cite{2022q})$
	Let $1<q\leq p<6$ with $\frac{1}{2}-\frac{1}{p}<\frac{1}{q}\leq \frac{1}{p}+\frac{1}{3}.$ Assume that $a_0 \in B_{q, 1}^{\frac{3}{q}}\left(\mathbb{R}^3\right),$ $u_0 \in \dot{B}_{p, 1}^{\frac{3}{p}-1}\left(\mathbb{R}^3\right)$ with $\operatorname{div} u_0=0$. There exits a small constant $\varepsilon_{0}$ depending on $\|a_0\|_{B_{q, 1}^{\frac{3}{q}}}$ so that if
	\begin{equation}\label{aa12}
	\left\|u_0\right\|_{\dot{B}_{p, 1}^{\frac{3}{p}-1}} \leq \varepsilon_{0},
	\end{equation}
	then there is a time $T^*\geq1,$ such that (\ref{a2}) has a unique local solution $a$ $\in$ $\mathcal{C}_b([0, T^*];$ $B_{q, 1}^{\frac{3}{q}}(\mathbb{R}^3))$, $ u\in\mathcal{C}_b([0, T^*] ; \dot{B}_{p, 1}^{\frac{3}{p}-1}(\mathbb{R}^3))$ $\cap $ $L^1([0, T^*] ;$ $ \dot{B}_{p, 1}^{\frac{3}{p}+1}(\mathbb{R}^3))$
	and $u_t,$ $\nabla\pi$ $\in $ $L^1([0, T^*] ;$ $ \dot{B}_{p,1}^{\frac{3}{p}-1}(\mathbb{R}^3)).$ 
\end{Lemma}

\begin{Remark}
	In comparison with the local solution in Lemma \ref{Lemma-2}, we further restrict the range of $p,q$ when showing that $T^*$ can be extended to $+\infty$ for (\ref{a1}). The reason for this is that we need to prove the propagation of the regularity of $u$ in $\dot{H}^{-2\delta}(\mathbb{R}^{3})$ with $\delta\in]1/2,3/4[$, which guarantees the decay estimates of the velocity field $u$ at infinity. 
\end{Remark}

\begin{Proposition}\label{Proposition-4.3} 	
	Let $p$, $q$, $\delta$, $u_0$ be given by Theorem \ref{Theorem-1}. Let $a,$ $b,$ $\lambda\in \widetilde{L}^{\infty}_{T}(B^{\frac{3}{q}}_{q,1})$ be functions of (\ref{a2}) satisfying $1+a\geq \underline{a}>0$ and $1+b\geq \underline{b}>0$. Let $(u,\nabla\pi)$ be a sufficiently smooth solution of the system (\ref{a2}) such that $u$ $\in$ $\mathcal{C}([0, T] ; \dot{B}_{p, 1}^{\frac{3}{p}-1}(\mathbb{R}^3))$ $\cap $ $L^1([0, T];$ $\dot{B}_{p, 1}^{\frac{3}{p}+1}(\mathbb{R}^3)),$ then there exists a sufficiently large integer $k\in\mathbb{Z}$ and for any $t\in[0,T]$ one has  
	\begin{equation}\label{c56}
	\begin{aligned}
	&\|u\|_{\widetilde{L}^{\infty}_{t}(\dot{H}^{-2\delta})}+\|u\|_{\widetilde{L}^{1}_{t}(\dot{H}^{2(1-\delta)})}+\|\nabla\pi\|_{\widetilde{L}^{1}_{t}(\dot{H}^{-2\delta})}\\\leq& C\bigg\{\|u_{0}\|_{\dot{H}^{-2\delta}}+t^{\frac{1}{4}}2^{(\frac{3}{2}-2\delta)k}(1+\|u\|_{L^{\infty}_{t}(\dot{B}^{-1+\frac{3}{p}}_{p,1})})\|u\|_{L^{\infty}_{t}(\dot{B}^{-1+\frac{3}{p}}_{p,1})}^{\frac{1}{4}} \|u\|_{L^{1}_{t}(\dot{B}^{1+\frac{3}{p}}_{p,1})}^{\frac{3}{4}}\bigg\}\\
	&\quad\times\exp\bigg\{C\|u\|_{L^{1}_{t}(\dot{B}^{1+\frac{3}{p}}_{p,1})}+Ct2^{(2+\frac{6}{q})k}\|(b,\lambda)\|^{2}_{L^{q}}+C\|(a,b,\lambda)\|_{\widetilde{L}^{\infty}_{t}(B_{q,1}^{\frac{3}{q}})}\bigg\}.
	\end{aligned}
	\end{equation}  
\end{Proposition}

\begin{proof}
	We first deduce from  $1+a\geq \underline{b}$ and $1+b\geq\underline{b}$ that
	$$1+\dot{S}_{k}a=1+a+(\dot{S}_{k}a-a)\geq \frac{\underline{a}}{2}~{\rm and}~ 1+\dot{S}_{k}b=1+b+(\dot{S}_{k}b-b)\geq \frac{\underline{b}}{2}.$$
	Correspondingly, we rewrite the $u$ equation of $(\ref{a2})_2$ as 
	\begin{equation}\label{c57}
	\begin{aligned}
	\partial_{t}u+u\cdot\nabla u-\operatorname{div}(2(1+\dot{S}_{k}b)\mathbb{D}u)+\nabla\pi
	=E_{k}-\dot{S}_{k}a\nabla\pi-\nabla\dot{S}_{k}\lambda~\mathbb{D}u
	\end{aligned}
	\end{equation}
	with
	$$\begin{aligned}
	E_{k}\stackrel{\text { def }}{=}&\operatorname{div}(2(b-\dot{S}_{k}b)\mathbb{D}u)-\nabla(\lambda-\dot{S}_{k}\lambda)\mathbb{D}u-(a-\dot{S}_{k}a)\nabla\pi.
	\end{aligned}$$
	\vskip0.5cm
	\par Firstly, we estimate that $\|u\|_{\widetilde{L}^{\infty}_{t}(\dot{H}^{-2\delta})}$ and $\|u\|_{\widetilde{L}^{1}_{t}(\dot{H}^{2-2\delta})}$. 
	
	Let $\mathbb{P}=I+\nabla(-\Delta)^{-1}\operatorname{div}$~be the Leray projection operator. Applying $\dot{\Delta}_{j}\mathbb{P}$ to $(\ref{c57})$, it  gives 
	\begin{equation}\label{c58}
	\begin{aligned}
	\partial_{t}\dot{\Delta}_{j}u+\dot{\Delta}_{j}\mathbb{P}(u\cdot\nabla u)-\dot{\Delta}_{j}\mathbb{P}\{\operatorname{div}(2(1+\dot{S}_{k}b)\mathbb{D}u)\}
	=\dot{\Delta}_{j}\mathbb{P}(E_{k}-\dot{S}_{k}a\nabla\pi-\nabla\dot{S}_{k}\lambda~\mathbb{D}u).
	\end{aligned}
	\end{equation}
	Multiplying the above equation by $\dot{\Delta}_{j}u$ and then integrating the resulting equations on $x\in\mathbb{R}^{3}$, it leads to
	\begin{equation}\label{c59}
	\begin{aligned}
	\frac{d}{dt}\|\dot{\Delta}_{j}u\|_{L^{2}}+2^{2j}\|\dot{\Delta}_{j}u\|_{L^{2}}\lesssim&\|[u\cdot\nabla;\dot{\Delta}_{j}\mathbb{P}]u\|_{L^{2}}+2^{j}\|[\dot{S}_{k}b;\dot{\Delta}_{j}]\mathbb{D} u\|_{L^{2}} \\
	&+\|\dot{\Delta}_{j}E_{k}\|_{L^{2}}+\|\dot{\Delta}_{j}(\dot{S}_{k}a\nabla\pi)\|_{L^{2}}+\|\dot{\Delta}_{j}(\nabla\dot{S}_{k}\lambda~\mathbb{D}u)\|_{L^{2}}.
	\end{aligned}
	\end{equation}
	\par In what follows, we shall deal with the right-hand side of~(\ref{c59}).~Firstly applying homogeneous Bony's decomposition yields 
	$$[u\cdot\nabla;\dot{\Delta}_{j}\mathbb{P}]u=[T_{u}\cdot\nabla;\dot{\Delta}_{j}\mathbb{P}]u+T^{'}_{\nabla\dot{\Delta}_{j} u}u-\dot{\Delta}_{j}\mathbb{P}(T_{\nabla u}u)-\dot{\Delta}_{j}\mathbb{P}\operatorname{div}\mathcal{R}(u,u).$$
	It follows again from the above estimate,~which implies that
	$$
	\begin{aligned}
	\|[T_{u}\cdot\nabla;\dot{\Delta}_{j}\mathbb{P}]u\|_{L^{2}}
	&\lesssim\sum_{|j-k|\leq 4}2^{-j}\|\nabla\dot{S}_{k-1}u\|_{L^{\infty}}\|\dot{\Delta}_{k}\nabla u\|_{L^{2}}\\
	& \lesssim\sum_{|j-k|\leq 4} 2^{k-j} \|\nabla u\|_{L^{\infty}}\|\dot{\Delta}_{k}u\|_{L^{2}}\\
	& \lesssim 2^{2j\delta}\sum_{|j-k|\leq 4} 2^{(k-j)(1+2\delta)} \|\nabla u\|_{L^{\infty}}2^{-2k\delta}\|\dot{\Delta}_{k} u\|_{L^{2}}\\
	& \lesssim c_{j}2^{2j\delta}\|\nabla u\|_{L^{\infty}}\|u\|_{\dot {H}^{-2\delta}}.
	\end{aligned}
	$$
	The same estimate holds for $T^{'}_{\nabla\dot{\Delta}_{j} u}u.$ Note that
	$$
	\begin{aligned}
	\|\dot{\Delta}_{j}\mathbb{P}(T_{\nabla u} u)\|_{L^{2}}&\lesssim\sum_{|k-j|\leq 4}\|\dot{S}_{k-1}\nabla u\|_{L^{2}}\|\dot{\Delta}_{k}u\|_{L^{\infty}}\\&
	\lesssim 2^{2j\delta}\sum_{|k-j|\leq 4}2^{2(k-j)\delta}2^{-2k\delta}\|\dot{S}_{k-1} u\|_{L^{2}}\|\dot{\Delta}_{k}\nabla u\|_{L^{\infty}}\\&
	\lesssim c_{j}2^{2j\delta}\|u\|_{\dot {H}^{-2\delta}}\|\nabla u\|_{L^{\infty}}.
	\end{aligned}
	$$
	For $\delta<\frac{3}{4},$ we have
	$$
	\begin{aligned}
	\|\dot{\Delta}_{j}\mathbb{P}\operatorname{div}\mathcal{R}(u,u)\|_{L^{2}}&\lesssim2^{\frac{3j}{2}}\|\dot{\Delta}_{j}\mathbb{P}\mathcal{R}(u,u)\|_{L^{\frac{3}{2}}}\\&
	\lesssim2^{\frac{3j}{2}}\sum_{k\geq j-3}\|\dot{\Delta}_{k} u\|_{L^{2}}\|\dot{\Delta}_{k} u\|_{L^{6}}\\&
	\lesssim 2^{2j\delta}\sum_{k\geq j-3}2^{(\frac{3}{2}-2\delta)(j-k)} 2^{-2k\delta}\|\dot{\Delta}_{k} u\|_{L^{2}}2^{\frac{k}{2}}\|\dot{\Delta}_{k}\nabla u\|_{L^{6}}\\&
	\lesssim c_{j}2^{2j\delta}\|u\|_{\dot {H}^{-2\delta}}\| \nabla u\|_{\dot{B}^{\frac{1}{2}}_{6,1}}.
	\end{aligned}
	$$
	From above, we obtain
	\begin{equation}\label{c60}
	\|[u\cdot\nabla;\dot{\Delta}_{j}\mathbb{P}]u\|_{L^{2}}\lesssim c_{j}2^{2j\delta}\|\nabla u\|_{\dot{B}^{\frac{3}{p}}_{p,1}}\|u\|_{\dot {H}^{-2\delta}}.
	\end{equation}
	As for the estimate of $[\dot{S}_{k}b;\dot{\Delta}_{j}]\mathbb{D}u,$ we get, by virtue of Lemma \ref{Lemma-2.3}, that
	\begin{equation}\label{c61}
	\|[\dot{S}_{k}b;\dot{\Delta}_{j}]\mathbb{D}u\|_{L^{2}}\lesssim c_{j}2^{j(2\delta-1)}\|\nabla\dot{S}_{k}b\|_{\dot{B}^{\frac{3}{q}}_{q,\infty}\cap L^{\infty}}\|u\|_{\dot{H}^{1-2\delta}}.
	\end{equation}
	Similarly, for $\delta<\frac{3}{4}$, we obtain that
	\begin{equation}\label{c62}
	\begin{aligned}
	\|\dot{\Delta}_{j} E_{k}\|_{L^{2}}&\lesssim c_{j}2^{2j\delta} \|(a-\dot{S}_{k}a,~b-\dot{S}_{k}b,~\lambda-\dot{S}_{k}\lambda)\|_{\dot{B}^{\frac{3}{q}}_{q,1}}\left\{\|u\|_{\dot{H}^{2(1-\delta)}}+\|\nabla\pi\|_{\dot{H}^{-2\delta}}\right\}.
	\end{aligned}
	\end{equation}
	Along the same line, one has
	\begin{equation}\label{c63}
	\begin{aligned}
	\|\dot{\Delta}_{j}(\dot{S}_{k}a\nabla\pi) \|_{L^{2}}
	\lesssim&\|\dot{\Delta}_{j}(T_{\nabla\pi}\dot{S}_{k}a) \|_{L^{2}}+\|\dot{\Delta}_{j}(T_{\dot{S}_{k}a}\nabla\pi) \|_{L^{2}}+\|\dot{\Delta}_{j}\mathcal{R}(\nabla\pi,\dot{S}_{k}a) \|_{L^{2}} \\
	\lesssim&c_j 2^{2j\delta}\|\dot{S}_{k}a\|_{\dot{B}^{-2\delta}_{\infty,2}}\|\nabla\pi\|_{L^{2}}+2^{\frac{3}{2}j}\|\dot{\Delta}_{j}\mathcal{R}(\nabla\pi,\dot{S}_{k}a) \|_{L^{1}}\\
	\lesssim&c_j 2^{2j\delta}\|\dot{S}_{k}a\|_{\dot{H}^{\frac{3}{2}-2\delta}}\|\nabla\pi\|_{L^{2}}.
	\end{aligned}
	\end{equation}
	Similarly,
	\begin{equation}\label{c663}
	\begin{aligned}
	\|\dot{\Delta}_{j}(\nabla\dot{S}_{k}\lambda~\mathbb{D}u)\|_{L^{2}}
	\lesssim&\|\dot{\Delta}_{j}(T_{\mathbb{D}u}\nabla\dot{S}_{k}\lambda) \|_{L^{2}}+\|\dot{\Delta}_{j}(T_{\nabla\dot{S}_{k}\lambda}\mathbb{D}u)\|_{L^{2}}+\|\dot{\Delta}_{j}\mathcal{R}(\mathbb{D}u,\nabla\dot{S}_{k}\lambda) \|_{L^{2}}\\
	\lesssim& c_j 2^{2j\delta}\|\nabla u\|_{\dot{H}^{-2\delta}}\|\nabla\dot{S}_{k}\lambda\|_{L^{\infty}}+2^{\frac{3}{2}j}\|\dot{\Delta}_{j}\mathcal{R}(\mathbb{D}u,\nabla\dot{S}_{k}\lambda) \|_{L^{1}}\\
	\lesssim&  c_j 2^{2j\delta}\|\nabla\dot{S}_{k}\lambda\|_{\dot{B}^{\frac{3}{q}}_{q,\infty}\cap L^{\infty}}\|u\|_{\dot{H}^{1-2\delta}}.
	\end{aligned}
	\end{equation}
	According to Lemma \ref{lemma-2.1} and Lemma \ref{Lemma-2.2},~we may get
	\begin{equation}\label{c664}
	\begin{aligned}
	&\|(\nabla\dot{S}_{k}b,~\nabla\dot{S}_{k}\lambda)\|_{\dot{B}^{\frac{3}{q}}_{q,\infty}\cap L^{\infty}}\lesssim 2^{(1+\frac{3}{q})k}\|(\dot{S}_{k}b,~\dot{S}_{k}\lambda)\|_{L^{q}}.
	\end{aligned}
	\end{equation}
	\par Substituting (\ref{c60})-(\ref{c663}) into (\ref{c59}) and using the interpolation inequality
	$\|u\|_{\dot{H}^{1-2\delta}}\lesssim\|u\|^{\frac{1}{2}}_{\dot{H}^{-2\delta}}$ $\|u\|^{\frac{1}{2}}_{\dot{H}^{2(1-\delta)}}$ and Young's inequality, we obtain
	\begin{equation}\label{c64}
	\begin{aligned}
	&\|u\|_{\widetilde{L}^{\infty}_{t}(\dot{H}^{-2\delta})}+\|u\|_{\widetilde{L}^{1}_{t}(\dot{H}^{2(1-\delta)})}\\
	\lesssim& \|u_{0}\|_{\dot{H}^{-2\delta}}+\int^{t}_{0}\bigg(\|\nabla u\|_{\dot{B}^{\frac{3}{p}}_{p,1}}+2^{(2+\frac{6}{q})k}\|(\dot{S}_{k}b,\dot{S}_{k}\lambda)\|^{2}_{\widetilde{L}^{\infty}_{t}(L^{q})}\bigg)\|u\|_{\widetilde{L}^{\infty}_{t}(\dot{H}^{-2\delta})}d\tau\\
	&+\|(a-\dot{S}_{k}a,b-\dot{S}_{k}b,\lambda-\dot{S}_k\lambda)\|_{\widetilde{L}^{\infty}_{t}(\dot{B}^{\frac{3}{q}}_{q,1})}\bigg\{\|u\|_{\widetilde{L}^{1}_{t}(\dot{H}^{2(1-\delta)})}+\|\nabla\pi\|_{\widetilde{L}^{1}_{t}(\dot{H}^{-2\delta})}\bigg\}\\&+2^{(\frac{3}{2}-2\delta)k}\|\dot{S}_{k}a\|_{L^{\infty}_{t}(L^{2})}\|\nabla\pi\|_{L^{1}_{t}(L^{2})}.
	\end{aligned}
	\end{equation}
	\vskip0.5cm
	Secondly, we notice that the right-hand terms $\|\nabla\pi\|_{\widetilde{L}^{1}_{t}(\dot{H}^{-2\delta})}$ and $\|\nabla\pi\|_{L^{1}_{t}(L^{2})}$ of the inequality (\ref{c64}) are both unknown, we next estimate that $\|\nabla\pi\|_{\widetilde{L}^{1}_{t}(\dot{H}^{-2\delta})}.$
	
	In order to estimate the pressure function $\pi,$ we get from $(\ref{a2})_{2}$ that
	\begin{equation}\label{c65}
	\begin{aligned}
	\operatorname{div}((1+\dot{S}_{k}a)\nabla \pi)=&-\operatorname{div}(u\cdot\nabla u)+\operatorname{div}\operatorname{div}(2\dot{S}_{k}b\mathbb{D}u)+\operatorname{div}E_{k}-\operatorname{div}(\nabla\dot{S}_{k}\lambda~\mathbb{D}u),
	\end{aligned}
	\end{equation}
	which implies
	$$
	\begin{aligned}
	\operatorname{div}((1+\dot{S}_k a) \nabla \dot{\Delta}_j \pi)=&-\operatorname{div} \dot{\Delta}_j(u \cdot \nabla u)+\operatorname{div} \dot{\Delta}_j E_{k}\\
	&+\operatorname{div} \dot{\Delta}_j(2\nabla \dot{S}_k b ~\mathbb{D}u+\dot{S}_k b \Delta u)-\operatorname{div}\dot{\Delta}_j (\nabla\dot{S}_{k}\lambda~\mathbb{D}u)\\&-\operatorname{div}[\dot{\Delta}_j, \dot{S}_k a] \nabla \pi.
	\end{aligned}
	$$
	Taking $L^{2}$ inner product of the above equation with $\dot{\Delta}_j \pi$ and using the Definition \ref{Definition-2}, we find
	\begin{equation}\label{c66}
	\begin{aligned}
	\|\nabla\pi\|_{\widetilde{L}^{1}_{t}(\dot{H}^{-2\delta})}\lesssim&	\|\operatorname{div}(u\cdot\nabla u)\|_{\widetilde{L}^{1}_{t}(\dot{H}^{-1-2\delta})}+\|E_{k}\|_{\widetilde{L}^{1}_{t}(\dot{H}^{-2\delta})}\\&+\|\nabla \dot{S}_k b ~\mathbb{D}u\|_{\widetilde{L}^{1}_{t}(\dot{H}^{-2\delta})}+\|\Delta u\cdot\nabla\dot{S}_k b\|_{\widetilde{L}^{1}_{t}(\dot{H}^{-1-2\delta})}\\&+\|\nabla\dot{S}_{k}\lambda~\mathbb{D}u\|_{\widetilde{L}^{1}_{t}(\dot{H}^{-2\delta})}+\|(2^{-2j\delta }\|[\dot{S}_k a ;\dot{\Delta}_j] \nabla \pi \|_{L^1_{t}(L^2)})_{j\in\mathbb{Z}}\|_{l^2}.
	\end{aligned}
	\end{equation}
	We now estimate term by term in (\ref{c66}). Applying Lemma \ref{Lemma-2.2} , it yields
	\begin{equation}\label{cc20}
	\|\nabla u\cdot\nabla u\|_{\widetilde{L}^{1}_{t}(\dot{H}^{-1-2\delta})}\lesssim\int^{t}_{0}\|\nabla u\|_{\dot{B}^{\frac{3}{p}}_{p,1}}\|u\|_{\widetilde{L}^{\infty}_{t}(\dot{H}^{-2\delta})}d\tau,
	\end{equation}
	and for $\delta\in(\frac{1}{2},\frac{3}{4}),$ one has
	\begin{equation}
	\begin{aligned}
	\|E_{k}\|_{\widetilde{L}^{1}_{t}(\dot{H}^{-2\delta})}\lesssim& \|(a-\dot{S}_{k}a,~b-\dot{S}_{k}b,~\lambda-\dot{S}_{k}\lambda)\|_{\widetilde{L}^{\infty}_{t}(\dot{B}^{\frac{3}{q}}_{q,1})}\bigg\{\|u\|_{\widetilde{L}^{1}_{t}(\dot{H}^{2(1-\delta)})}+\|\nabla\pi\|_{\widetilde{L}^{1}_{t}(\dot{H}^{-2\delta})}\bigg\}.
	\end{aligned}
	\end{equation}
	Similarly,
	\begin{equation}
	\begin{aligned}
	&\|\nabla \dot{S}_k b~\mathbb{D}u\|_{\widetilde{L}^{1}_{t}(\dot{H}^{-2\delta})}
	\lesssim \|\nabla\dot{S}_{k}b\|_{{L}^{2}_{t}(\dot{B}^{\frac{3}{q}}_{q,\infty}\cap L^{\infty})}\|u\|_{L^{2}_{t}(\dot {H}^{1-2\delta})},
	\end{aligned}
	\end{equation}
	and
	\begin{equation}\label{c67}
	\|\Delta u\cdot\nabla\dot{S}_k b\|_{\widetilde{L}^{1}_{t}(\dot{H}^{-1-2\delta})}\lesssim \|\nabla\dot{S}_{k}b\|_{L^{2}_{t}(\dot{B}^{\frac{3}{q}}_{q,\infty}\cap L^{\infty})}\|u\|_{L^{2}_{t}(\dot {H}^{1-2\delta})}.
	\end{equation}
	The same estimate holds for $\nabla\dot{S}_{k}\lambda~\mathbb{D}u,$ i.e.,
	\begin{equation}
	\|\nabla\dot{S}_{k}\lambda~\mathbb{D}u\|_{\widetilde{L}^{1}_{t}(\dot{H}^{-2\delta})}\lesssim \|\nabla\dot{S}_{k}\lambda\|_{L^{2}_{t}(\dot{B}^{\frac{3}{q}}_{q,\infty}\cap L^{\infty})}\|u\|_{L^{2}_{t}(\dot {H}^{1-2\delta})}.
	\end{equation}
	Whereas applying Bony’s decomposition once again, it leads to
	$$[\dot{S}_k a ;\dot{\Delta}_j] \nabla \pi=[T_{\dot{S}_k a};\dot{\Delta}_j]\nabla \pi+T^{'}_{\dot{\Delta}_j \nabla \pi}\dot{S}_k a-\dot{\Delta}_j T_{\nabla \pi} \dot{S}_k a-\dot{\Delta}_j\mathcal{R}( \dot{S}_k a,\nabla \pi).$$
	Applying Lemma \ref{Lemma-2.2}, we yield
	$$
	\begin{aligned}
	\|[T_{\dot{S}_k a};\dot{\Delta}_j]\nabla \pi\|_{L^{2}}
	&\lesssim\sum_{|j-l|\leq 4}2^{-j}\|\nabla\dot{S}_{l-1}\dot{S}_{k}a\|_{L^{\infty}}\|\dot{\Delta}_{l}\nabla \pi\|_{L^{2}}\\
	&\lesssim c_{j}2^{2j\delta}\|\dot{S}_{k}a\|_{\dot{H}^{\frac{3}{2}-2\delta}}\|\nabla \pi\|_{L^{2}}.
	\end{aligned}
	$$
	Similarly, we have
	$$
	\begin{aligned}
	\|T^{'}_{\dot{\Delta}_j \nabla \pi}\dot{S}_k a\|_{L^{2}}&\lesssim\sum_{l\geq j-2}\|\dot{S}_{l+2}\dot{\Delta}_{j}\nabla \pi\|_{L^{\infty}}\|\dot{\Delta}_{l}\dot{S}_{k}a\|_{L^{2}}\\
	&\lesssim2^{2j\delta}\|\dot{\Delta}_{j}\nabla \pi\|_{L^{2}} \sum_{l\geq j-2}2^{(j-l)(\frac{3}{2}-2\delta)}2^{l(\frac{3}{2}-2\delta)}\|\dot{\Delta}_{l}\dot{S}_{k}a\|_{L^{2}}\\
	&\lesssim c_{j}2^{2j\delta}\|\nabla \pi\|_{L^{2}}\|\dot{S}_{k}a\|_{\dot{H}^{\frac{3}{2}-2\delta}}.
	\end{aligned}
	$$
	The same estimate holds for~$\dot{\Delta}_j T_{\nabla \pi} \dot{S}_k a$. Notice that
	$$
	\begin{aligned}
	\|\dot{\Delta}_j\mathcal{R}( \dot{S}_k a,\nabla \pi)\|_{L^{2}}&\lesssim2^{\frac{3}{2}j}\|\dot{\Delta}_j\mathcal{R}( \dot{S}_k a,\nabla \pi)\|_{L^{1}}\\
	&\lesssim2^{2j\delta}\sum_{l\geq j-3}2^{(j-l)(\frac{3}{2}-2\delta)}\|\dot{\Delta}_{l}\nabla \pi\|_{L^{2}}2^{l(\frac{3}{2}-2\delta)}\|\dot{\Delta}_{l}\dot{S}_{k}a\|_{L^{2}}\\
	&\lesssim c_{j}2^{2j\delta}\|\nabla \pi\|_{L^{2}}\|\dot{S}_{k}a\|_{\dot{H}^{\frac{3}{2}-2\delta}}.
	\end{aligned}
	$$
	Therefore, we can deduce that 
	\begin{equation}\label{c68}
	\|[\dot{S}_k a ;\dot{\Delta}_j] \nabla \pi\|_{L^{1}_{t}(L^{2})}\lesssim c_{j}2^{2j\delta}\|\dot{S}_{k}a\|_{\widetilde{L}^{\infty}_{t}(\dot{H}^{\frac{3}{2}-2\delta})}\|\nabla \pi\|_{L^{1}_{t}(L^{2})}.
	\end{equation}
	\par Substituting (\ref{cc20})-(\ref{c68}) into (\ref{c66}) and using the interpolation inequality
	$\|u\|_{\dot{H}^{1-2\delta}}\lesssim\|u\|^{\frac{1}{2}}_{\dot{H}^{-2\delta}}\|u\|^{\frac{1}{2}}_{\dot{H}^{2(1-\delta)}}$ and Young's inequality, we write
	\begin{equation}\label{c69}
	\begin{aligned}
	\|\nabla\pi\|_{\widetilde{L}^{1}_{t}(\dot{H}^{-2\delta})}\leq&C\int^{t}_{0}\bigg(\|u\|_{\dot{B}^{1+\frac{3}{p}}_{p,1}}+2^{(2+\frac{6}{q})k}\|(\dot{S}_{k}b,\dot{S}_{k}\lambda)\|^{2}_{L^{q}}\bigg)\|u\|_{\widetilde{L}^{\infty}_{t}(\dot {H}^{-2\delta})}d\tau\\
	&+C\|(a-\dot{S}_{k}a,b-\dot{S}_{k}b,\lambda-\dot{S}_k \lambda)\|_{\widetilde{L}^{\infty}_{t}(\dot{B}^{\frac{3}{q}}_{q,1})}\bigg\{\|u\|_{\widetilde{L}^{1}_{t}(\dot{H}^{2(1-\delta)})}+\|\nabla\pi\|_{\widetilde{L}^{1}_{t}(\dot{H}^{-2\delta})}\bigg\}\\
	&+C\|\dot{S}_{k}a\|_{\widetilde{L}^{\infty}_{t}(\dot{H}^{\frac{3}{2}-2\delta})}\|\nabla \pi\|_{L^{1}_{t}(L^{2})}+\frac{1}{8}\|u\|_{\widetilde{L}^{1}_{t}(\dot{H}^{2-2\delta})},
	\end{aligned}
	\end{equation}
	which along with (\ref{c64}) ensures that
	\begin{equation}
	\begin{aligned}
	&\|u\|_{\widetilde{L}^{\infty}_{t}(\dot{H}^{-2\delta})}+\|u\|_{\widetilde{L}^{1}_{t}(\dot{H}^{2(1-\delta)})}+\|\nabla\pi\|_{\widetilde{L}^{1}_{t}(\dot{H}^{-2\delta})}\\
	\leq&C\|u_{0}\|_{\dot{H}^{-2\delta}}+C\int^{t}_{0}\bigg(\|u\|_{\dot{B}^{1+\frac{3}{p}}_{p,1}}+2^{(2+\frac{6}{q})k}\|(\dot{S}_{k}b,\dot{S}_{k}\lambda)\|^{2}_{L^{q}}\bigg)\|u\|_{\widetilde{L}^{\infty}_{t}(\dot{H}^{-2\delta})}d\tau\\
	&+C\|(a-\dot{S}_{k}a,b-\dot{S}_{k}b,\lambda-\dot{S}_k \lambda)\|_{\widetilde{L}^{\infty}_{t}(\dot{B}^{\frac{3}{q}}_{q,1})}\bigg\{\|u\|_{\widetilde{L}^{1}_{t}(\dot{H}^{2(1-\delta)})}+\|\nabla\pi\|_{\widetilde{L}^{1}_{t}(\dot{H}^{-2\delta})}\bigg\}\\
	&+C\|\dot{S}_{k}a\|_{\widetilde{L}^{\infty}_{t}(\dot{H}^{\frac{3}{2}-2\delta})}\|\nabla \pi\|_{L^{1}_{t}(L^{2})}.
	\end{aligned}
	\end{equation}
    Then we have used $a,$ $b,$ $\lambda\in \widetilde{L}^{\infty}_{T}(B^{\frac{3}{q}}_{q,1})$ and chosen $k$ sufficiently large so that
	\begin{equation}
	C\|(a-\dot{S}_{k}a,b-\dot{S}_{k}b,\lambda-\dot{S}_k \lambda)\|_{\widetilde{L}^{\infty}_{t}(\dot{B}^{\frac{3}{q}}_{q,1})}\leq\frac{1}{8},
	\end{equation}
	from which, we duduce that
	\begin{equation}\label{c70}
	\begin{aligned}
	&\|u\|_{\widetilde{L}^{\infty}_{t}(\dot{H}^{-2\delta})}+\|u\|_{\widetilde{L}^{1}_{t}(\dot{H}^{2(1-\delta)})}+\|\nabla\pi\|_{\widetilde{L}^{1}_{t}(\dot{H}^{-2\delta})}\\
	\leq&C\|u_{0}\|_{\dot{H}^{-2\delta}}+C\int^{t}_{0}\bigg(\|u\|_{\dot{B}^{1+\frac{3}{p}}_{p,1}}+2^{(2+\frac{6}{q})k}\|(\dot{S}_{k}b,\dot{S}_{k}\lambda)\|^{2}_{L^{q}}\bigg)\|u\|_{\widetilde{L}^{\infty}_{t}(\dot{H}^{-2\delta})}d\tau\\
	&+C\|\dot{S}_{k}a\|_{\widetilde{L}^{\infty}_{t}(\dot{H}^{\frac{3}{2}-2\delta})}\|\nabla \pi\|_{L^{1}_{t}(L^{2})}.
	\end{aligned}
	\end{equation}
	\vskip 0.5cm
     \par  Finally, we remark that the only unknown term in the right-hand side of the equation (\ref{c70}) is the pressure term $\|\nabla\pi\|_{L^{1}_{t}(L^{2})}.$ In the following, we will estimating $\|\nabla\pi\|_{L^{1}_{t}(L^{2})}$ and break down the proof of (\ref{c56}) with $p\in]1,9/2]$ into the following cases:
	\vskip 0.5cm
	\noindent\textbf{Case 1} The estimate of (\ref{c56}) with $p\in ]1,2[$.
	
	\par We first get by taking div to $(\ref{a2})_{2}$ that
	\begin{equation}
	\begin{aligned}\label{bb28}
	\operatorname{div}\left\{\left(1+a\right)\nabla\pi\right\}=&-\operatorname{div}\left(u\cdot\nabla u\right)+\operatorname{div}\operatorname{div} \left(2b\mathbb{D}u\right)
	-\operatorname{div}\left(\nabla\lambda~\mathbb{D}u\right),
	\end{aligned}
	\end{equation}
	from which and $0<\underline{a}\leq 1+a$, we deduce by a standard energy estimate that
	\begin{equation}\label{bb29}
	\begin{aligned}
	\|\nabla\pi\|_{L^{1}_{t}(L^{2})}\lesssim\|u\cdot\nabla u\|_{L^{1}_{t}(L^{2})}+&\|\operatorname{div} (2b\mathbb{D}u)\|_{L^{1}_{t}(L^2)}+\|\nabla\lambda~\mathbb{D}u\|_{L^{1}_{t}(L^2)}.
	\end{aligned}
	\end{equation}
	For the case of $1<p<2,$ we have $\dot{B}^{\frac{3}{p}-\frac{3}{2}}_{p,2}\hookrightarrow L^2,$ it is enough for estimate every term on the
	right hand side of (\ref{bb29}) in the space of $\dot{B}^{\frac{3}{p}-\frac{3}{2}}_{p,2}$. As a result, it comes out
	$$
	\begin{aligned}
	\|u\cdot\nabla u\|_{\dot{B}^{\frac{3}{p}-\frac{3}{2}}_{p,2}}\lesssim&\|T_{u}\nabla u\|_{\dot{B}^{\frac{3}{p}-\frac{3}{2}}_{p,2}}+\|T_{\nabla u} u\|_{\dot{B}^{\frac{3}{p}-\frac{3}{2}}_{p,2}}+\|\operatorname{div}\mathcal{R}(u,u)\|_{\dot{B}^{\frac{3}{p}-\frac{3}{2}}_{p,2}}\\
	\lesssim& \|u\|_{\dot{B}^{\frac{3}{p}}_{p,1}}\|\nabla u\|_{\dot{B}^{-\frac{3}{2}}_{\infty,1}}+\|\mathcal{R}(u,u)\|_{\dot{B}^{\frac{3}{p}-\frac{1}{2}}_{p,2}}\\
	\lesssim & \|u\|_{\dot{B}^{\frac{3}{p}}_{p,1}}\|u\|_{\dot{B}^{\frac{3}{p}-\frac{1}{2}}_{p,1}},
	\end{aligned}
	$$
	then we get, by applying interpolation inequalities, that
	\begin{equation}\label{bb31}
	\|u\cdot\nabla u\|_{\dot{B}^{\frac{3}{p}-\frac{3}{2}}_{p,2}}\lesssim \|u\|^{\frac{5}{4}}_{\dot{B}^{-1+\frac{3}{p}}_{p,1}}\|u\|^{\frac{3}{4}}_{\dot{B}^{1+\frac{3}{p}}_{p,1}}.
	\end{equation}
	It follows from Lemma \ref{Lemma-2.2} and $q\leq p$ that
	\begin{equation}
	\begin{aligned}
	\|\operatorname{div} (2b\mathbb{D}u)\|_{\dot{B}^{\frac{3}{p}-\frac{3}{2}}_{p,2}}\lesssim\|b\mathbb{D}u\|_{\dot{B}^{\frac{3}{p}-\frac{1}{2}}_{p,2}}&\lesssim\|T_{b}\mathbb{D}u\|_{\dot{B}^{\frac{3}{p}-\frac{1}{2}}_{p,2}}+\|T_{\mathbb{D}u} b\|_{\dot{B}^{\frac{3}{p}-\frac{1}{2}}_{p,2}}+\|\mathcal{R}(b,\mathbb{D}u)\|_{\dot{B}^{\frac{3}{p}-\frac{1}{2}}_{p,2}}\\
	&\lesssim \|b\|_{L^\infty}\|\nabla u\|_{\dot{B}^{\frac{3}{p}-\frac{1}{2}}_{p,1}}+\|\nabla u\|_{\dot{B}^{-\frac{1}{2}}_{\infty,1}}\|b\|_{\dot{B}^{\frac{3}{p}}_{p,1}}\\
	&\lesssim \|b\|_{\dot{B}^{\frac{3}{q}}_{q,1}}\|u\|_{\dot{B}^{\frac{3}{p}+\frac{1}{2}}_{p,1}}.
	\end{aligned}
	\end{equation}
	The same estimate holds for~$\nabla\lambda~\mathbb{D}u$. Notice that
	\begin{equation}\label{bb32}
	\begin{aligned}
	\|u\|_{\dot{B}^{\frac{3}{p}+\frac{1}{2}}_{p,1}}\lesssim \|u\|^{\frac{1}{4}}_{\dot{B}^{-1+\frac{3}{p}}_{p,1}}\|u\|^{\frac{3}{4}}_{\dot{B}^{1+\frac{3}{p}}_{p,1}}.
	\end{aligned}
	\end{equation}	
	Inserting the above estimates into (\ref{bb29}) and using (\ref{bb32}) gives rise to
	\begin{equation}\label{bb30}
	\begin{aligned}
	\|\nabla\pi\|_{L^{1}_{t}(L^{2})}\lesssim&\int^{t}_{0}\|u\|^{\frac{5}{4}}_{\dot{B}^{-1+\frac{3}{p}}_{p,1}}\|u\|^{\frac{3}{4}}_{\dot{B}^{1+\frac{3}{p}}_{p,1}}d\tau\\&+\|(b,\lambda)\|_{L^{\infty}_{t}(\dot{B}_{q,1}^{\frac{3}{q}})}\int^{t}_{0}\|u\|_{\dot{B}_{p, 1}^{-1+\frac{3}{p}}}^{\frac{1}{4}}\|u\|_{\dot{B}_{p, 1}^{1+\frac{3}{p}}}^{\frac{3}{4}} d\tau,
	\end{aligned}
	\end{equation}
		which combining (\ref{c70}), for some integer $k$, we achieve
	\begin{equation}\label{cc35}
	\begin{aligned}
	&\|u\|_{\widetilde{L}^{\infty}_{t}(\dot{H}^{-2\delta})}+\|u\|_{\widetilde{L}^{1}_{t}(\dot{H}^{2(1-\delta)})}+\|\nabla\pi\|_{\widetilde{L}^{1}_{t}(\dot{H}^{-2\delta})}\\
	\lesssim& \|u_{0}\|_{\dot{H}^{-2\delta}}+\int^{t}_{0}\bigg(\|u\|_{\dot{B}^{1+\frac{3}{p}}_{p,1}}+2^{(2+\frac{6}{q})k}\|(b,\lambda)\|^{2}_{L^{q}}\bigg)\|u\|_{\widetilde{L}^{\infty}_{t}(\dot{H}^{-2\delta})}d\tau\\&+2^{(\frac{3}{q}-2\delta)k}\|a\|_{L^{\infty}_{t}(L^{q})}\int^{t}_{0}\|u\|^{\frac{5}{4}}_{\dot{B}^{-1+\frac{3}{p}}_{p,1}}\|u\|^{\frac{3}{4}}_{\dot{B}^{1+\frac{3}{p}}_{p,1}}d\tau\\&+2^{(\frac{3}{q}-2\delta)k}\|a\|_{L^{\infty}_{t}(L^{q})}\|(b,\lambda)\|_{L^{\infty}_{t}(\dot{B}_{q,1}^{\frac{3}{q}})}\int^{t}_{0}\|u\|_{\dot{B}_{p, 1}^{-1+\frac{3}{p}}}^{\frac{1}{4}}\|u\|_{\dot{B}_{p, 1}^{1+\frac{3}{p}}}^{\frac{3}{4}} d\tau.
	\end{aligned}
	\end{equation}
	Applying Gronwall's lemma to (\ref{cc35}), we complete the inequality of  (\ref{c56}).
	\vskip 0.5cm

	\noindent\textbf{Case 2.} The estimate of (\ref{c56}) with $p\in[2,4]$. 
	
	\par We can easily deduce from (\ref{bb28}) that
	\begin{equation}\label{b25}
	\begin{aligned}
	\|\nabla\pi\|_{L^{1}_{t}(L^{2})}\lesssim&\|u\cdot\nabla u\|_{L^{1}_{t}(L^{2})}+\|\Delta u\cdot\nabla b\|_{L^{1}_{t}(\dot{H}^{-1})}\\&+\|\operatorname{div}(\nabla b~\mathbb{D}u)\|_{L^{1}_{t}(\dot{H}^{-1})}+\|\operatorname{div}\left(\nabla\lambda~\mathbb{D}u\right)\|_{L^{1}_{t}(\dot{H}^{-1})}.
	\end{aligned}
	\end{equation}
	Then it follows from interpolation inequalities in Besov spaces and for $p\in[3,4]$ that
	\begin{equation}\label{b26}
	\begin{aligned}
	\|u\cdot\nabla u\|_{L^{2}}\lesssim &\|u\|_{L^{p}}\|\nabla u\|_{L^{\frac{2p}{p-2}}}\lesssim \|u\|_{L^p}\|u\|_{\dot{B}^{\frac{6}{p}-\frac{1}{2}}_{p,1}}
	\lesssim\|u\|^{\frac{5}{4}}_{\dot{B}^{-1+\frac{3}{p}}_{p,1}}\|u\|^{\frac{3}{4}}_{\dot{B}^{1+\frac{3}{p}}_{p,1}},
	\end{aligned}
	\end{equation}
	if $p\in[2,3],$ we have
	\begin{equation}\label{bb26}
	\begin{aligned}
	\|u\cdot\nabla u\|_{L^{2}}\lesssim &\|u\|_{L^{2p}}\|\nabla u\|_{L^{\frac{4p}{2p-2}}}\lesssim \|u\|_{\dot{B}^{\frac{3}{2p}}_{p,1}}\|u\|_{\dot{B}^{-\frac{1}{2}+\frac{9}{2p}}_{p,1}}
	\lesssim\|u\|^{\frac{5}{4}}_{\dot{B}^{-1+\frac{3}{p}}_{p,1}}\|u\|^{\frac{3}{4}}_{\dot{B}^{1+\frac{3}{p}}_{p,1}}.
	\end{aligned}
	\end{equation}
	Similarly, we can deduce 
	\begin{equation}\label{b27}
	\begin{aligned}
	\|\Delta u\cdot\nabla b\|_{\dot{H}^{-1}}\lesssim&\|T_{\nabla b}\Delta u\|_{\dot{H}^{-1}}+\|T_{\Delta u}\nabla b\|_{\dot{H}^{-1}}+\|\operatorname{div}\mathcal{R}(b,\Delta u)\|_{\dot{H}^{-1}}\\
	\lesssim&\|\nabla b\|_{L^{3}}\|\Delta u\|_{\dot{B}_{6, 2}^{-1}}+\|\mathcal{R}(b,\Delta u)\|_{\dot{B}^{\frac{1}{2}}_{\frac{3}{2},2}} \\
	\lesssim&\|b\|_{\dot{B}_{q,1}^{\frac{3}{q}}}\|u\|_{\dot{B}_{p, 1}^{-1+\frac{3}{p}}}^{\frac{1}{4}}\|u\|_{\dot{B}_{p, 1}^{1+\frac{3}{p}}}^{\frac{3}{4}},
	\end{aligned}
	\end{equation}	
	and
	\begin{equation}\label{b28}
	\begin{aligned}
	\|\operatorname{div}(\nabla b~\mathbb{D}u)\|_{\dot{H}^{-1}}\lesssim\|\nabla u\|_{L^{6}}\|\nabla b\|_{L^{3}} \lesssim\|b\|_{\dot{B}_{q,1}^{\frac{3}{q}}}\|u\|_{\dot{B}_{p, 1}^{-1+\frac{3}{p}}}^{\frac{1}{4}}\|u\|_{\dot{B}_{p, 1}^{1+\frac{3}{p}}}^{\frac{3}{4}}.
	\end{aligned}
	\end{equation}	
	Along the same way, we obtain
	\begin{equation}\label{b29}
	\begin{aligned}
	\|\operatorname{div}\left(\nabla\lambda~\mathbb{D}u\right)\|_{\dot{H}^{-1}}
	\lesssim \|\lambda\|_{\dot{B}_{q,1}^{\frac{3}{q}}}\|u\|_{\dot{B}_{p, 1}^{-1+\frac{3}{p}}}^{\frac{1}{4}}\|u\|_{\dot{B}_{p, 1}^{1+\frac{3}{p}}}^{\frac{3}{4}}.
	\end{aligned}
	\end{equation}
	\par Substituting (\ref{b26}), (\ref{b27}), (\ref{b28}) and (\ref{b29}) into (\ref{b25}) and the fact that
	\begin{equation}\label{b30}
	\begin{aligned}
	\|\nabla\pi\|_{L^{1}_{t}(L^{2})}\lesssim&\int^{t}_{0}\|u\|^{\frac{5}{4}}_{\dot{B}^{-1+\frac{3}{p}}_{p,1}}\|u\|^{\frac{3}{4}}_{\dot{B}^{1+\frac{3}{p}}_{p,1}}d\tau\\&+\|(b,\lambda)\|_{L^{\infty}_{t}(\dot{B}_{q,1}^{\frac{3}{q}})}\int^{t}_{0}\|u\|_{\dot{B}_{p, 1}^{-1+\frac{3}{p}}}^{\frac{1}{4}}\|u\|_{\dot{B}_{p, 1}^{1+\frac{3}{p}}}^{\frac{3}{4}} d\tau,
	\end{aligned}
	\end{equation}
	which combining (\ref{c70}), we achieve
	\begin{equation}\label{cc43}
	\begin{aligned}
	&\|u\|_{\widetilde{L}^{\infty}_{t}(\dot{H}^{-2\delta})}+\|u\|_{\widetilde{L}^{1}_{t}(\dot{H}^{2(1-\delta)})}+\|\nabla\pi\|_{\widetilde{L}^{1}_{t}(\dot{H}^{-2\delta})}\\
	\lesssim& \|u_{0}\|_{\dot{H}^{-2\delta}}+\int^{t}_{0}\bigg(\|u\|_{\dot{B}^{1+\frac{3}{p}}_{p,1}}+2^{(2+\frac{6}{q})k}\|(b,\lambda)\|^{2}_{L^{q}}\bigg)\|u\|_{\widetilde{L}^{\infty}_{t}(\dot{H}^{-2\delta})}d\tau\\&+2^{(\frac{3}{q}-2\delta)k}\|a\|_{L^{\infty}_{t}(L^{q})}\int^{t}_{0}\|u\|^{\frac{5}{4}}_{\dot{B}^{-1+\frac{3}{p}}_{p,1}}\|u\|^{\frac{3}{4}}_{\dot{B}^{1+\frac{3}{p}}_{p,1}}d\tau\\&+2^{(\frac{3}{q}-2\delta)k}\|a\|_{L^{\infty}_{t}(L^{q})}\|(b,\lambda)\|_{L^{\infty}_{t}(\dot{B}_{q,1}^{\frac{3}{q}})}\int^{t}_{0}\|u\|_{\dot{B}_{p, 1}^{-1+\frac{3}{p}}}^{\frac{1}{4}}\|u\|_{\dot{B}_{p, 1}^{1+\frac{3}{p}}}^{\frac{3}{4}} d\tau.
	\end{aligned}
	\end{equation}
	Applying Gronwall's lemma to (\ref{cc43}), it implies (\ref{c56}).
	\vskip 0.5cm
	\noindent\textbf{Case 3.} The estimate of (\ref{c56}) with $p\in ]4,9/2]$. 
	
    \par In order to relax the condition on $p$ in (\ref{c56}) to $]1,9/2]$, we need to re-estimate (\ref{c63}) and (\ref{c68}) with the help of the following Lemma:
	
	\begin{Lemma}\label{pi2}
		Let $1\leq q\leq2<\frac{p}{2}\leq\infty$ with $1-\frac{2}{p}\leq\frac{1}{q}$ and $\delta\in]1/2,3/4[,$ one has
		\begin{equation}\label{cc37}
		\|\dot{\Delta}_j (\dot{S}_{k}a\nabla\pi)\|_{L^{2}}
		\lesssim c_{j}2^{2j\delta}\|\dot{S}_{k}a\|_{\dot{B}^{\frac{3}{q}-2\delta}_{q,\infty}}\|\nabla\pi\|_{\dot{B}^{\frac{6}{p}-\frac{3}{2}}_{\frac{p}{2},2}},
		\end{equation}
		and
		\begin{equation}\label{cc38}
		\|[\dot{S}_k a ;\dot{\Delta}_j] \nabla \pi\|_{L^{2}}\lesssim c_{j}2^{2j\delta}\|\dot{S}_{k}a\|_{\dot{B}^{\frac{3}{q}-2\delta}_{q,\infty}}\|\nabla \pi\|_{\dot{B}^{\frac{6}{p}-\frac{3}{2}}_{\frac{p}{2},2}}.
		\end{equation}
	\end{Lemma}
	\begin{proof}
		By applying Bony's decomposition and a standard process, we write
		$$
		\begin{aligned}
		\|\dot{S}_{k}a\nabla\pi\|_{\dot{H}^{-2\delta}}
		\lesssim&\|T_{\nabla\pi}\dot{S}_{k}a \|_{\dot{H}^{-2\delta}}+\|T_{\dot{S}_{k}a}\nabla\pi \|_{\dot{H}^{-2\delta}}+\|\mathcal{R}(\nabla\pi,\dot{S}_{k}a) \|_{\dot{H}^{-2\delta}} \\
		\lesssim&\|\dot{S}_{k}a\|_{\dot{B}^{\frac{3}{2}-\frac{6}{p}-2\delta}_{\frac{2p}{p-4},\infty}}\|\nabla\pi\|_{\dot{B}^{\frac{6}{p}-\frac{3}{2}}_{\frac{p}{2},2}}+\|\mathcal{R}(\nabla\pi,\dot{S}_{k}a) \|_{\dot{B}^{\frac{3}{2}-2\delta}_{1,2}}\\
		\lesssim&\|\dot{S}_{k}a\|_{\dot{B}^{\frac{3}{q}-2\delta}_{q,\infty}}\|\nabla\pi\|_{\dot{B}^{\frac{6}{p}-\frac{3}{2}}_{\frac{p}{2},2}}.
		\end{aligned}
		$$
		The same estimate holds for $[\dot{S}_k a ;\dot{\Delta}_j] \nabla \pi$, we omitted here.
	\end{proof}
	
	\par Inserting the estimates of Lemma \ref{pi2} into (\ref{c59}) and (\ref{c66}) and using a similar derivation of (\ref{c70}), we deduce that
	\begin{equation}\label{b39}
	\begin{aligned}
	&\|u\|_{\widetilde{L}^{\infty}_{t}(\dot{H}^{-2\delta})}+\|u\|_{\widetilde{L}^{1}_{t}(\dot{H}^{2(1-\delta)})}+\|\nabla\pi\|_{\widetilde{L}^{1}_{t}(\dot{H}^{-2\delta})}\\
	\lesssim&\|u_{0}\|_{\dot{H}^{-2\delta}}+\int^{t}_{0}\bigg(\|u\|_{\dot{B}^{1+\frac{3}{p}}_{p,1}}+2^{(2+\frac{6}{q})k}\|(\dot{S}_{k}b,\dot{S}_{k}\lambda)\|^{2}_{L^{q}}\bigg)\|u\|_{\widetilde{L}^{\infty}_{t}(\dot{H}^{-2\delta})}d\tau\\
	&+\|\dot{S}_{k}a\|_{\widetilde{L}^{\infty}_{t}(\dot{B}^{\frac{3}{q}-2\delta}_{q,\infty})}\|\nabla\pi\|_{\widetilde{L}^{1}_t(\dot{B}^{\frac{6}{p}-\frac{3}{2}}_{\frac{p}{2},2})}.
	\end{aligned}
	\end{equation}
	As in the previous steps, taking div to $(\ref{a2})_{2}$ that
	$$
	\begin{aligned}
	\operatorname{div}\left\{\left(1+a\right)\nabla\pi\right\}=&-\operatorname{div}\left(u\cdot\nabla u\right)+\operatorname{div}\operatorname{div} \left(2b\mathbb{D}u\right)
	-\operatorname{div}\left(\nabla\lambda~\mathbb{D}u\right).
	\end{aligned}
	$$
   From which and Lemma \ref{pi}, one has
   \begin{equation}\label{b41}
   \begin{aligned}
   \|\nabla\pi\|_{\widetilde{L}^{1}_t(\dot{B}^{\frac{6}{p}-\frac{3}{2}}_{\frac{p}{2},2})}\lesssim&(1+\|a\|_{\widetilde{L}^{\infty}_{t}(\dot{B}_{q,1}^{\frac{3}{q}})})\bigg\{\|u\cdot\nabla u\|_{\widetilde{L}^{1}_t(\dot{B}^{\frac{6}{p}-\frac{3}{2}}_{\frac{p}{2},2})}+\|\Delta u\cdot\nabla b\|_{\widetilde{L}^{1}_{t}(\dot{H}^{-1})}\\&+\|\operatorname{div}(\nabla b~\mathbb{D}u)\|_{\widetilde{L}^{1}_{t}(\dot{H}^{-1})}+\|\operatorname{div}\left(\nabla\lambda~\mathbb{D}u\right)\|_{\widetilde{L}^{1}_{t}(\dot{H}^{-1})}\bigg\},
   \end{aligned}
   \end{equation}
   where we use the embedding inequality $\dot{H}^{-1}\hookrightarrow\dot{B}^{\frac{6}{p}-\frac{5}{2}}_{\frac{p}{2},2}$ for $p>4$ in the inequality. We remark that the advantage of (\ref{b41}) is that the pressure term does not appear on the right-hand side of (\ref{b41}). Applying Lemma \ref{lemma-2.1} and Lemma \ref{Lemma-2.2} yields for $p\in]4,\frac{9}{2}],$
   $$
   \begin{aligned}
   \|u\cdot\nabla u\|_{\dot{B}^{\frac{6}{p}-\frac{3}{2}}_{\frac{p}{2},2}}&\lesssim \|T_{u}\nabla u\|_{\dot{B}^{\frac{6}{p}-\frac{3}{2}}_{\frac{p}{2},2}}+\|T_{\nabla u} u\|_{\dot{B}^{\frac{6}{p}-\frac{3}{2}}_{\frac{p}{2},2}}+\|\operatorname{div}\mathbb{R}(u,u)\|_{\dot{B}^{\frac{6}{p}-\frac{3}{2}}_{\frac{p}{2},2}}\\
   &\lesssim \|u\|_{\dot{B}^{\frac{6}{p}-\frac{3}{2}}_{p,2}}\|\nabla u\|_{\dot{B}^{0}_{p,\infty}}+\|\nabla u\|_{\dot{B}^{\frac{6}{p}-\frac{5}{2}}_{p,2}}\|u\|_{\dot{B}^{1}_{p,\infty}}+\|\mathbb{R}(u,u)\|_{\dot{B}^{\frac{6}{p}-\frac{1}{2}}_{\frac{p}{2},2}}\\
   &\lesssim \|u\|_{\dot{B}^{\frac{6}{p}-\frac{3}{2}}_{p,2}}\|u\|_{\dot{B}^{1}_{p,\infty}}.
   \end{aligned}
   $$
   Then it follows from interpolation inequalities in Besov spaces that
   \begin{equation}\label{b42}
   \|u\cdot\nabla u\|_{\dot{B}^{\frac{6}{p}-\frac{3}{2}}_{\frac{p}{2},2}}\lesssim\|u\|_{\dot{B}^{\frac{6}{p}-\frac{3}{2}}_{p,2}}\|u\|_{\dot{B}^{1}_{p,\infty}}\lesssim\|u\|^{\frac{5}{4}}_{\dot{B}^{-1+\frac{3}{p}}_{p,1}}\|u\|^{\frac{3}{4}}_{\dot{B}^{1+\frac{3}{p}}_{p,1}}.
   \end{equation}
   For the remaining terms have been given by (\ref{b27}) (\ref{b28}) and (\ref{b29}) respectively. 
   
   By inserting the above estimates into (\ref{b41}), we get
   \begin{equation}\label{cc42}
   \begin{aligned}
   \|\nabla\pi\|_{\widetilde{L}^{1}_{t}(\dot{B}^{\frac{6}{p}-\frac{3}{2}}_{\frac{p}{2},2})}\lesssim&(1+\|a\|_{\widetilde{L}^{\infty}_{t}(\dot{B}_{q,1}^{\frac{3}{q}})})\Big\{\int^{t}_{0}\|u\|^{\frac{5}{4}}_{\dot{B}^{-1+\frac{3}{p}}_{p,1}}\|u\|^{\frac{3}{4}}_{\dot{B}^{1+\frac{3}{p}}_{p,1}}d\tau\\&+\|(b,\lambda)\|_{\widetilde{L}^{\infty}_{t}(\dot{B}_{q,1}^{\frac{3}{q}})}\int^{t}_{0}\|u\|_{\dot{B}_{p, 1}^{-1+\frac{3}{p}}}^{\frac{1}{4}}\|u\|_{\dot{B}_{p, 1}^{1+\frac{3}{p}}}^{\frac{3}{4}} d\tau\Big\},
   \end{aligned}
   \end{equation}
which along with (\ref{b39}) ensures that 
\begin{equation}\label{ddd1}
\begin{aligned}
&\|u\|_{\widetilde{L}^{\infty}_{t}(\dot{H}^{-2\delta})}+\|u\|_{\widetilde{L}^{1}_{t}(\dot{H}^{2(1-\delta)})}+\|\nabla\pi\|_{\widetilde{L}^{1}_{t}(\dot{H}^{-2\delta})}\\
\lesssim&\|u_{0}\|_{\dot{H}^{-2\delta}}+\int^{t}_{0}\bigg(\|u\|_{\dot{B}^{1+\frac{3}{p}}_{p,1}}+C2^{(2+\frac{6}{q})k}\|(\dot{S}_{k}b,\dot{S}_{k}\lambda)\|^{2}_{L^{q}}\bigg)\|u\|_{\widetilde{L}^{\infty}_{t}(\dot{H}^{-2\delta})}d\tau\\
&+C2^{(\frac{3}{q}-2\delta)k}\|a\|_{\widetilde{L}^{\infty}_{t}(L^q)}(1+\|a\|_{\widetilde{L}^{\infty}_{t}(\dot{B}_{q,1}^{\frac{3}{q}})})\bigg\{\int^{t}_{0}\|u\|^{\frac{5}{4}}_{\dot{B}^{-1+\frac{3}{p}}_{p,1}}\|u\|^{\frac{3}{4}}_{\dot{B}^{1+\frac{3}{p}}_{p,1}}d\tau\\&+\|(b,\lambda)\|_{L^{\infty}_{t}(\dot{B}_{q,1}^{\frac{3}{q}})}\int^{t}_{0}\|u\|_{\dot{B}_{p, 1}^{-1+\frac{3}{p}}}^{\frac{1}{4}}\|u\|_{\dot{B}_{p, 1}^{1+\frac{3}{p}}}^{\frac{3}{4}} d\tau\bigg\}.
\end{aligned}
\end{equation}
Finally, applying Gronwall's lemma to \eqref{ddd1}, it implies the desired inequality (\ref{c56}).
\end{proof}

 With Lemma \ref{Lemma-2}, we can prove the propagation of regularities of $u_0\in \dot{H}^{-2\delta}$ for $u$ on $t\in]0,1[$, even higher regularity estimates can be obtained with the following results:

\begin{Proposition}\label{Corollary-3.1}
 Under the assumptions of Theorem \ref{Theorem-1}, there exists a time $T\geq1,$ such that we can find some $t_1\in]0,T[$ and there holds
	\begin{equation}\label{c772}
	\begin{aligned}
	&\|u\|_{\widetilde{L}^{\infty}([0,T];\dot{H}^{-2\delta})}+\|u\|_{\widetilde{L}^{1}([0,T];\dot{H}^{2(1-\delta)})}\leq C\{1+\|u_0\|_{\dot{H}^{-2\delta}}\},
	\end{aligned}
	\end{equation}
	\begin{equation}\label{c72}
	\begin{aligned}
	&\|u\|_{L^{\infty}([t_1,T];L^2)}+\|u\|_{\widetilde{L}^{1}([t_1,T];\dot{H}^{2})}\leq C\{1+\|u_0\|_{\dot{H}^{-2\delta}}/t_1^{\delta}\},
	\end{aligned}
	\end{equation}
	where $C$ depends only on $t$, $\underline{\mu}$, $\overline{\mu}$, $p$, $q$, $\|\rho_0-1\|_{{B}_{q,1}^{\frac{3}{q}}}$ and $\|u_0\|_{\dot{B}^{-1+\frac{3}{p}}_{p,1}}.$
\end{Proposition}

\begin{proof}
To obtain that (\ref{c772}) holds, we need to resort to the result of the Lemma \ref{Lemma-2} and check that it satisfies the conditions of Proposition \ref{Proposition-4.3}, i.e., $b,$ $\lambda\in \widetilde{L}^{\infty}_{T}(B^{\frac{3}{q}}_{q,1}).$ By virtue of the definition of $b(a)$ and $\lambda(a)$ from (\ref{a3}), we have $b(0)=\lambda(0)=0.$ Thus, resorting to the Paralinearization theorem of \cite{2011BCD}, we can easily obtain
	\begin{align}
	\|(a,b,\lambda)\|_{\widetilde{L}^{\infty}_{1}(B^{\frac{3}{q}}_{q,1})}\lesssim\|a\|_{\widetilde{L}^{\infty}_{1}(B^{\frac{3}{q}}_{q,1})}\lesssim\|a_0\|_{B^{\frac{3}{q}}_{q,1}}\lesssim \|\rho_0-1\|_{B^{\frac{3}{q}}_{q,1}},
	\end{align}
	from which we can conclude that combining Proposition \ref{Proposition-4.3} gives (\ref{c772}).
	
	\par Next, in a similar way to the process of Proposition \ref{Proposition-4.3}, we make a higher order estimate of $u$ in the following two steps: 
	\vskip 0.5cm
	\noindent\textbf{Step 1.} The estimate of (\ref{c72}) with $p\in ]1,4]$.
	
	 By a similar proof of case 1 and case 2 in Proposition \ref{Proposition-4.3}, we can find some $0<\tau<t_1<T$ such that $u(\tau)\in L^2$ and 
	\begin{equation}\label{b45}
	\begin{aligned}
	&\|u\|_{L^{\infty}([\tau,T];L^2)}+\|u\|_{\widetilde{L}^{1}([\tau,T];\dot{H}^{2})}+\|\nabla\pi\|_{L^{1}([\tau,T];L^2)}\\
	\lesssim&\{1+\|u(\tau)\|_{L^{2}}\}
	\lesssim\{1+\|u\|_{L^{1}([0,t_1];L^{2})}/t_1\}
	\lesssim\{1+\|u_0\|_{\dot{H}^{-2\delta}}/t_1^{\delta}\}.
	\end{aligned}
	\end{equation}
	
	\noindent\textbf{Step 2.} The estimate of (\ref{c72}) with $p\in]4,\frac{9}{2}].$

	We rewrite the $u$ equation of $(\ref{a2})_2$ as 
	\begin{equation}\label{c73}
	\begin{aligned}
	&\partial_{t}u+u\cdot\nabla u-\operatorname{div}(2(1+\dot{S}_{k}b)\mathbb{D}u)+\nabla((1+a)\pi)=E_{k}+\pi\nabla a-\nabla\dot{S}_{k}\lambda~\mathbb{D}u
	\end{aligned}
	\end{equation}
	with
	$$\begin{aligned}
	E_{k}=&\operatorname{div}(2(b-\dot{S}_{k}b)\mathbb{D}u)-\nabla(\lambda-\dot{S}_{k}\lambda)\mathbb{D}u.
	\end{aligned}$$
    Applying $\dot{\Delta}_{j}\mathbb{P}$ to $(\ref{c73}),$
	multiplying the equation by $\dot{\Delta}_{j}u$ and then integrating the resulting equations on $x\in\mathbb{R}^{3}$, we deduce
	\begin{equation}\label{c74}
	\begin{aligned}
	\frac{d}{dt}\|\dot{\Delta}_{j}u\|_{L^{2}}+2^{2j}\|\dot{\Delta}_{j}u\|_{L^{2}}&\lesssim\|[u\cdot\nabla;\dot{\Delta}_{j}\mathbb{P}]u\|_{L^{2}}+2^{j}\|[\dot{S}_{k}b;\dot{\Delta}_{j}]\mathbb{D} u\|_{L^{2}} \\
	&+\|\dot{\Delta}_{j}\mathbb{P}E_{k}\|_{L^{2}}+\|[\pi;\dot{\Delta}_{j}\mathbb{P}]\nabla a\|_{L^2}+\|\dot{\Delta}_{j}(\nabla\dot{S}_{k}\lambda~\mathbb{D}u)\|_{L^{2}}.
	\end{aligned}
	\end{equation}
	\par In what follows, we shall deal with the right-hand side of (\ref{c74}). First, by virtue of Lemma \ref{Lemma-2.2} and Lemma \ref{Lemma-2.3}, we get
	\begin{equation}\label{c75}
	\|[u\cdot\nabla;\dot{\Delta}_{j}\mathbb{P}]u\|_{L^{1}_{t}(L^{2})}\lesssim c_{j}\int^{t}_{0}\|\nabla u\|_{\dot{B}^{\frac{3}{p}}_{p,1}}\|u\|_{L^\infty_{t}(L^2)}d\tau,
	\end{equation}
	\begin{equation}\label{c76}
	\|[\dot{S}_{k}b;\dot{\Delta}_{j}]\mathbb{D}u\|_{L^{1}_{t}(L^{2})}\lesssim c_{j}2^{-j}\|\nabla\dot{S}_{k}b\|_{L^{2}_{t}(\dot{B}^{\frac{3}{q}}_{q,\infty}\cap L^{\infty})}\|u\|_{L^{2}_{t}(\dot{H}^{1})},
	\end{equation}
	\begin{equation}\label{c77}
	\begin{aligned}
	\|\dot{\Delta}_{j} E_{k}\|_{L^{1}_{t}(L^{2})}\lesssim c_{j} \|(b-\dot{S}_{k}b,~\lambda-\dot{S}_{k}\lambda)\|_{\widetilde{L}^{\infty}_{t}(\dot{B}^{\frac{3}{q}}_{q,1})}\|u\|_{\widetilde{L}^{1}_{t}(\dot{H}^{2})},
	\end{aligned}
	\end{equation}
	and
	\begin{equation}\label{c78}
	\begin{aligned}
	\|\dot{\Delta}_{j}(\nabla\dot{S}_{k}\lambda~\mathbb{D}u)\|_{L^{1}_{t}(L^{2})}
	\lesssim&  c_j \|\nabla\dot{S}_{k}\lambda\|_{L^{2}_{t}(\dot{B}^{\frac{3}{q}}_{q,\infty}\cap L^{\infty})}\|u\|_{L^{2}_{t}(\dot{H}^{1})}.
	\end{aligned}
	\end{equation}
	About the estimate of $[\pi;\dot{\Delta}_{j}\mathbb{P}]\nabla a$, we get, by applying Bony’s decomposition, that
	$$[\pi ;\dot{\Delta}_j\mathbb{P}] \nabla a=[T_{\pi};\dot{\Delta}_j\mathbb{P}]\nabla a-\dot{\Delta}_j\mathbb{P}(T_{\nabla a} \pi) -\dot{\Delta}_j\mathbb{P}\mathcal{R}(\nabla a,\pi).$$
	Applying Lemma \ref{Lemma-2.2}, we yield
	$$
	\begin{aligned}
	\|[T_{\pi};\dot{\Delta}_j\mathbb{P}]\nabla a\|_{L^{2}}
	&\lesssim\sum_{|j-k|\leq 4}2^{-j}\|\dot{S}_{k-1}\nabla \pi\|_{L^{\frac{p}{2}}}\|\dot{\Delta}_{k}\nabla a\|_{L^{\frac{2p}{p-4}}}\\
	&\lesssim\sum_{|j-k|\leq 4}2^{k-j}2^{(\frac{6}{p}-\frac{3}{2})k}\|\dot{S}_{k-1}\nabla \pi\|_{L^{\frac{p}{2}}}2^{(\frac{1}{2}-\frac{6}{p})k}\|\dot{\Delta}_{k}\nabla a\|_{L^{\frac{2p}{p-4}}}\\
	&\lesssim c_{j}\|\nabla\pi\|_{\dot{B}^{\frac{6}{p}-\frac{3}{2}}_{\frac{p}{2},2}}\|a\|_{\dot{B}^{\frac{3}{q}}_{q,1}},
	\end{aligned}
	$$
	where we use the embedding inequality $\dot{B}^{\frac{3}{q}}_{q,1}\hookrightarrow\dot{B}^{\frac{3}{2}-\frac{6}{p}}_{\frac{2p}{p-4},1}$ for $p>4$. Similarly, we have
	$$
	\begin{aligned}
	\|\dot{\Delta}_j\mathbb{P}(T_{\nabla a} \pi)\|_{L^{2}}
	&\lesssim\sum_{|k-j|\leq 4}2^{(\frac{1}{2}-\frac{6}{p})k}\|\dot{S}_{k+2}\nabla a\|_{L^{\frac{2p}{p-4}}}2^{(\frac{6}{p}-\frac{1}{2})k}\|\dot{\Delta}_{k}\pi\|_{L^{\frac{p}{2}}}\\
	&\lesssim c_{j}\|a\|_{\dot{B}^{\frac{3}{q}}_{q,1}}\|\nabla\pi\|_{\dot{B}^{\frac{6}{p}-\frac{3}{2}}_{\frac{p}{2},2}}.
	\end{aligned}
	$$
	Using the same techniques to $\dot{\Delta}_j\mathcal{R}( \dot{S}_k a,\nabla \pi)$, we get for $1-\frac{2}{p}\leq\frac{1}{q}$ that
	$$
	\begin{aligned}
	\|\dot{\Delta}_j\mathbb{P}\mathcal{R}( \nabla a, \pi)\|_{L^{2}}&\lesssim2^{\frac{3}{2}j}\|\dot{\Delta}_j\mathcal{R}( \nabla a, \pi)\|_{L^{1}}\\
	&\lesssim\sum_{k\geq j-3}2^{\frac{3}{2}(j-k)}2^{(2-\frac{6}{p})k}\|\dot{\Delta}_{k}\nabla a\|_{L^{\frac{p}{p-2}}}2^{(\frac{6}{p}-\frac{1}{2})k}\|\dot{\Delta}_{k} \pi\|_{L^{\frac{p}{2}}}\\
	&\lesssim c_{j}\|a\|_{\dot{B}^{\frac{3}{q}}_{q,1}}\|\nabla\pi\|_{\dot{B}^{\frac{6}{p}-\frac{3}{2}}_{\frac{p}{2},2}}.
	\end{aligned}
	$$
	Therefore, we can deduce that 
	\begin{equation}\label{c79}
	\|[\pi ;\dot{\Delta}_j\mathbb{P}] \nabla a\|_{L^{1}_{t}(L^{2})}\lesssim c_{j}\|a\|_{\widetilde{L}^{\infty}_{t}(\dot{B}^{\frac{3}{q}}_{q,1})}\|\nabla \pi\|_{\widetilde{L}^{1}_{t}(\dot{B}^{\frac{6}{p}-\frac{3}{2}}_{\frac{p}{2},2})}.
	\end{equation}
	 Inserting (\ref{cc42}) into (\ref{c79}), we know
	\begin{equation}\label{c80}
	\begin{aligned}
	\|[\pi ;\dot{\Delta}_j\mathbb{P}] \nabla a\|_{L^{1}_{t}(L^{2})}&\lesssim c_{j}(1+\|a\|^2_{\widetilde{L}^{\infty}_{t}(\dot{B}_{q,1}^{\frac{3}{q}})})\Big\{\int^{t}_{0}\|u\|^{\frac{5}{4}}_{\dot{B}^{-1+\frac{3}{p}}_{p,1}}\|u\|^{\frac{3}{4}}_{\dot{B}^{1+\frac{3}{p}}_{p,1}}d\tau\\&+\|(b,\lambda)\|_{\widetilde{L}^{\infty}_{t}(\dot{B}_{q,1}^{\frac{3}{q}})}\int^{t}_{0}\|u\|_{\dot{B}_{p, 1}^{-1+\frac{3}{p}}}^{\frac{1}{4}}\|u\|_{\dot{B}_{p, 1}^{1+\frac{3}{p}}}^{\frac{3}{4}} d\tau\Big\}\lesssim c_{j}.
	\end{aligned}
	\end{equation}
	Substituting (\ref{c75}), (\ref{c76}), (\ref{c77}), (\ref{c78}) and (\ref{c80}) into (\ref{c74}) together with the  Gronwall's inequality and Proposition \ref{Proposition-4.3}, we can find some $0<\tau<t_1<T$ such that 
	\begin{equation}\label{c81}
	\begin{aligned}
	&\|u\|_{L^{\infty}([\tau,T];L^2)}+\|u\|_{\widetilde{L}^{1}([\tau,T];\dot{H}^{2})}+\|\nabla\pi\|_{\widetilde{L}^{1}([\tau,T];\dot{B}^{\frac{6}{p}-\frac{3}{2}}_{\frac{p}{2},2})}
	\lesssim\{1+\|u_0\|_{\dot{H}^{-2\delta}}/t_1^{\delta}\}.
	\end{aligned}
	\end{equation}
	This completes the proof of Proposition \ref{Corollary-3.1}.
\end{proof}
   \vskip 0.5cm
   From the results obtained in Proposition \ref{Corollary-3.1}, it is easy to obtain the following propagation of the regularities of $u_0\in \dot{H}^{-2\delta}$ for $u$ on $[0,1].$
   
   \begin{Remark}\label{c3.1}
   	Under the assumptions of Theorem \ref{Theorem-1}, there exists $t_0\in]0,1[$ such that $u(t_0)\in \dot{H}^{-2\delta}\cap H^2\cap \dot{B}_{p, 1}^{-1+\frac{3}{p}}\cap\dot{B}_{p, 1}^{1+\frac{3}{p}}\left(\mathbb{R}^3\right).$ Moreover, there holds 
   	\begin{equation}\label{cc62}
   	\begin{aligned}
   	\|u(t_{0})\|_{\dot{B}_{p, 1}^{-1+\frac{3}{p}}\cap\dot{B}_{p, 1}^{1+\frac{3}{p}}}\leq C\|u_{0}\|_{\dot{B}_{p, 1}^{-1+\frac{3}{p}}}\quad{\rm and}\quad\|u(t_0)\|_{\dot{H}^{-2\delta}\cap H^{2}}\leq C\{1+\|u_{0}\|_{\dot{H}^{-2\delta}}\},
   	\end{aligned}
   	\end{equation}
   	where $C$ depends only on $\underline{\mu}$, $\overline{\mu}$, $p$, $q$, $\|\rho_0-1\|_{{B}_{q,1}^{\frac{3}{q}}}$ and $\|u_0\|_{\dot{B}^{-1+\frac{3}{p}}_{p,1}}.$
   \end{Remark}

	\section{Global well-posedness of (\ref{a1})}
    The goal of this section is to prove the global well-posedness part of Theorem \ref{Theorem-1} provided that $\|u_{0}\|_{\dot{B}^{-1+\frac{3}{p}}_{p,1}}$ is sufficiently small. Showing that the \emph{a priori} $L^{1}([t_0,\infty);L^\infty(\mathbb{R}^{3}))$ estimate for $\nabla u$ is sufficiently small is a crucial part of proving Theorem \ref{Theorem-1}. The key ingredient here is to decompose the velocity fields $u$ into $v$ and $w$, where $v$ satisfies the 3D classical Navier-Stokes equations
	\begin{equation}\label{dd1}
	\left\{\begin{array}{l}
	\partial_t v+v \cdot \nabla v-\Delta v+\nabla \pi_v=0, \\
	\operatorname{div} v=0, \\
	\left. v\right|_{t=t_0}=u\left(t_0\right).
	\end{array}\right.
	\end{equation}
	The perturbation $w:=u-v$ satisfies
	\begin{equation}\label{dd2}
	\left\{\begin{array}{l}
	\partial_t \rho+\operatorname{div}(\rho(v+w))=0, \\
	\rho \partial_t w+\rho(v+w) \cdot \nabla w-\operatorname{div}(2\mu(\rho)\mathbb{D}w)+\nabla \pi_w\\
	=(1-\rho)\left(\partial_t v+v \cdot \nabla v\right)-\rho w \cdot \nabla v+\operatorname{div}(2(\mu(\rho)-1)\mathbb{D}v), \\
	\operatorname{div} w=0, \\
	\left.\rho\right|_{t=t_0}=\rho\left(t_0\right),\left.\quad w\right|_{t=t_0}=0.
	\end{array}\right.
	\end{equation}
    Since there is no smallness when making energy estimates for $u$ directly, there is naturally no way to get smallness for the $L^{1}([t_0,\infty);L^\infty(\mathbb{R}^{3}))$ estimate of $\nabla u$. So we use the above decomposition when making the $\dot{H}^1$ estimate of $u$. The advantage here is that while there is no smallness for the $L^{\infty}([t_0,\infty);H^1(\mathbb{R}^{3}))$ estimate of $u$, the $L^{\infty}([t_0,\infty);H^1(\mathbb{R}^{3}))$ estimate of $w$ is sufficiently small. More details can be found in Lemma \ref{Lemma-5.3}.
	\begin{Remark}
		We shall provide more information for this global solution $(\rho,u,\nabla\pi)$ in the process of proof to Theorem \ref{Theorem-1}. Indeed let $t_0$ be determined by (\ref{cc62}), we shall prove (\ref{a1}) has a unique global solution $(\rho,u,\nabla\pi)$ on $[t_0,+\infty[$ with $u=v+w$ and $v$ solving (\ref{dd1}), $w$ solving (\ref{dd2}). The detailed information of $v$ is presented in
		Proposition \ref{Proposition-5.1}, and that of $w$ is in Lemma \ref{Lemma-5.3}.
	\end{Remark}

	\noindent\textbf{Strategy of the proof of Theorem \ref{Theorem-1}:} Let $\rho_0,u_0,p,q,\delta$ be given by Theorem \ref{Theorem-1}. Thanks to Lemma \ref{Lemma-2} and Remark \ref{c3.1}, we conclude that system (\ref{a1}) has a unique local solution $(\rho, u)$ satisfying $\rho-1\in\mathcal{C}([0,T^{*});B_{q,1}^{\frac{3}{q}}(\mathbb{R}^{3}))$ and $u\in \mathcal{C}([0,T^{*});\dot{B}_{p, 1}^{-1+\frac{3}{p}}(\mathbb{R}^{3}))$ $\cap$ $ L^{1}([0,T^{*});$ $\dot{B}_{p, 1}^{1+\frac{3}{p}}(\mathbb{R}^{3}))$ for some $T^{*}\geq 1$, and we can find some $t_{0}\in (0,1)$ such that
	\begin{equation}\label{d1}
	\|u(t_{0})\|_{\dot{B}_{p, 1}^{-1+\frac{3}{p}}\cap\dot{B}_{p, 1}^{1+\frac{3}{p}}}\leq C\|u_{0}\|_{\dot{B}_{p, 1}^{-1+\frac{3}{p}}}\quad{\rm and}\quad\|u(t_0)\|_{\dot{H}^{-2\delta}\cap H^{2}}\leq C\{1+\|u_{0}\|_{\dot{H}^{-2\delta}}\}.
	\end{equation}
	Our aim of what follows is to prove that $T^{*}=+\infty.$

	\par Noticing from (\ref{d1}) that $u(t_{0})\in\dot{B}_{p, 1}^{-1+\frac{3}{p}}\cap\dot{B}_{p, 1}^{1+\frac{3}{p}}$ is very small provided that $\|u_{0}\|_{\dot{B}_{p, 1}^{-1+\frac{3}{p}}}$ is sufficiently small. The detailed information of $v$ is presented in the following, where we omitted the proof.
	\begin{Proposition}\label{Proposition-5.1} $(\cite{2013AGZ})$ 
		Let $(v,\pi_{v})$ be a unqiue global solution of $(\ref{dd1})$ which satisfies $(\ref{a7})$. Then for $s\in[-1+\frac{3}{p},1+\frac{3}{p}],$ there holds
		\begin{equation}\label{d3}
		\|v\|_{\widetilde{L}^{\infty}([t_0,\infty);\dot{B}^{s}_{p,1})}+\| (\Delta v,\nabla\pi_{v})\|_{L^1([t_0,\infty);\dot{B}^{s}_{p,1})}\leq C\|u(t_0)\|_{\dot{B}_{p,1}^{s}}\leq C\|u_0\|_{\dot{B}_{p,1}^{-1+\frac{3}{p}}},
		\end{equation}
		\begin{equation}\label{d4}
		\|\partial_t v\|_{\widetilde{L}^{\infty}([t_0,\infty);\dot{B}_{p, 1}^{-1+\frac{3}{p}})}+\|\partial_t v\|_{L^{1}([t_0,\infty);\dot{B}_{p,1}^{-1+\frac{3}{p}}\cap\dot{B}_{p,1}^{1+\frac{3}{p}})}\leq C\|u(t_0)\|_{\dot{B}_{p,1}^{s}}\leq C\|u_0\|_{\dot{B}_{p,1}^{-1+\frac{3}{p}}}.
		\end{equation} 
	\end{Proposition}
	
	The goal of this section is to prove that the following Proposition holds, which is the most important ingredient used in the proof of Theorem \ref{Theorem-1}, i.e.,
	
	\begin{Proposition}\label{Proposition-5.5}
		Let $M\stackrel{\mathrm{ def }}{=}\|\nabla\mu(\rho_0)\|_{L^{r}\cap L^3}$ with $r\in]3,+\infty[$, $n\in]2,\frac{6\delta_{-}}{1+\delta_{-}}[$ with $\delta\in]1/2,3/4[$ and $m\in(3,\min\{r,6\}).$ Then there exists some positive constant $\varepsilon,$ which depends only on $n,m,\bar{\mu}, \underline{\mu}$ and $M$ and $\|u_0\|_{\dot{H}^{-2\delta}}$ such that if $u$ and $w$ are the unique local strong solution of $(\ref{a1})$ and (\ref{dd2}) on $\mathbb{R}^{3}\times[t_0,T]$ respectively, and satisfying
		\begin{equation}\label{d39}
		\sup_{t\in[t_0,T]}\|\nabla\mu(\rho)\|_{L^{r}\cap L^3}\leq 4M\quad{\rm and}\quad
		\sup_{t\in[t_{0},T]}\|\nabla w\|^{2}_{L^{2}}\leq  4C\|u_{0}\|^{2}_{\dot{B}^{-1+\frac{3}{p}}_{p,1}},
		\end{equation}
		then the following estimates hold:
		\begin{equation}\label{d40}
		\sup_{t\in[t_0,T]}\|\nabla\mu(\rho)\|_{L^{r}\cap L^3}\leq 2M\quad{\rm and}\quad
		\sup_{t\in[t_{0},T]}\|\nabla w\|^{2}_{L^{2}}\leq 2C\|u_{0}\|^{2}_{\dot{B}^{-1+\frac{3}{p}}_{p,1}},
		\end{equation}
		provided that $\|u_0\|_{\dot{B}^{-1+\frac{3}{p}}_{p,1}}\leq \varepsilon.$
	\end{Proposition}
	
	\par Before proving Proposition \ref{Proposition-5.5}, we establish some necessary \emph{a priori} estimates, see Lemmas {\ref{Lemma-u1}-\ref{Lemma-5.6}}. 
 	\subsection{The {\em a priori} estimates} 
 		As a convention in the remaining of this section, we shall always  denote $s_1\in[\frac{3}{p},1+\frac{3}{p}]$. For simplicity, in what follows, we just present the \emph{a priori} estimates for smooth enough solutions of (\ref{dd2}) on $[0,T^{*}[.$
 	\subsubsection{$L^2$ estimate of $u$}
 	\begin{Lemma}\label{Lemma-u1}
 		Suppose $(\rho,u,\pi)$ is the unique local strong solution to $(\ref{a1})$ satisfying $(\ref{a7})$, then it holds that
 		\begin{equation}\label{dd38}
 		\sup_{t\in[t_0,T]}\int_{\mathbb{R}^3}\rho|u|^{2}dx+\int^{T}_{t_0}\int_{\mathbb{R}^3}|\nabla u|^{2}dxdt\leq C\|u(t_0)\|^{2}_{L^{2}},
 		\end{equation}
 		where $C$ depends on $\overline{\mu}$, $\underline{\mu}$, $M$, $\|\rho_0-1\|_{B^{\frac{3}{q}}_{q,1}}$ and $\|u_0\|_{\dot{H}^{-2\delta}}$.
 	\end{Lemma}
 	\begin{proof}
 		Taking the $L^{2}$ inner product of $(\ref{a1})_{2}$ with $u$ and using the fact $\operatorname{div} u=0,$ we obtain
 		\begin{equation}\label{d35}
 		\frac{1}{2}\frac{d}{dt}\|\sqrt{\rho}u\|^{2}_{L^{2}}+2\int_{\mathbb{R}^{3}}\mu(\rho)\mathbb{D}u:\mathbb{D}udx=0.
 		\end{equation}
 		Integrating in time over~$[t_{0}, t]$~yields
 		\begin{equation}
 		\|\sqrt{\rho}u\|^{2}_{L^{\infty}([t_{0},T];L^{2})}+\|\nabla u\|^{2}_{L^{2}([t_{0},T];L^{2})}\leq C\|u(t_0)\|^{2}_{L^{2}}.
 		\end{equation}
 		This completes the proof of Lemma \ref{Lemma-u1}.
 	\end{proof}
 	    \par Before estimating $u$, we first need the following large time-decay estimates, the proof process can be referred to \cite{NW2023NS}. 
 	    
 		\begin{Corollary}\label{c-4.2} $(\cite{NW2023NS})$
 		Under the assumptions of Theorem \ref{Theorem-1}, we have
 		\begin{equation}\label{delta-1}
 		\|u\|^{2}_{L^{2}}\leq C\mathcal{H}_{0}\langle t\rangle ^{-2\delta} \quad{\rm for~any}~ t\in[t_{0},T],
 		\end{equation}
 		and 
 		\begin{equation}\label{d44}
 		\begin{aligned}
 		\|t^{\delta}u\|^2_{L^{\infty}([t_{0},T];L^{2})}+\|t^{\delta_{-}}\nabla u\|^2_{L^{2}([t_{0},T];L^{2})}\leq C\mathcal{H}_{0},
 		\end{aligned}
 		\end{equation}
 		where 
 		\begin{equation}\label{h}
 		\mathcal{H}_{0}\stackrel{\mathrm{ def }}{=}1+\|u(t_{0})\|^{2}_{\dot{H}^{-2\delta}}+\|u(t_{0})\|^{2}_{H^1}(1+\|1-\rho_{0}\|^{2}_{L^{2}}+\|u(t_{0})\|^{2}_{L^{2}}).
 		\end{equation}
 	\end{Corollary}
 	\begin{proof} 
 		The large time decay estimate of $u$ in (\ref{delta-1}) is from the reference \cite{NW2023NS}, we have omitted the details here. Multiplying (\ref{d35}) by $t^{2\delta_{-}}$ and then integrating the resulting inequality over $[t_{0},T]$, we get that
 		\begin{equation}\label{d45}
 		\begin{aligned}
 		&\|t^{\delta_{-}}u\|^2_{L^{\infty}([t_{0},T];L^{2})}+\|t^{\delta_{-}}\nabla u\|^2_{L^{2}([t_{0},T];L^{2})}\\
 		\leq& C\|u(t_0)\|_{L^2}+\int^{T}_{t_{0}}t^{(-1)_{-}}\|t^{\delta}u(t)\|^{2}_{L^{2}}dt\leq C\mathcal{H}_{0},
 		\end{aligned}
 		\end{equation}
 		which completes the proof of this corollary.
 	\end{proof}
 
 	\subsubsection{$\dot{H}^1$ estimate of $u$} 
 	If we make the $\dot{H}^1$-normal estimate of $u$ directly, then the estimate cannot be closed, here we use the idea of decomposition, ie, $u=v+w$, so that we can not only get the smallness for the $L^\infty([t_0,\infty);H^1)$ estimate of $w$, but even the consistent bound for the $L^\infty([t_0,\infty);L^2)$ estimate of $\nabla u$. More deeply, in order to later estimate the smallness for the $L^{1}([t_0,\infty);L^\infty)$ of $\nabla u$, the smallness of the $H^1$-norm estimate of $w$ plays a crucial role.
 	\begin{Lemma}\label{Lemma-5.3}
 		Suppose $(v,\pi_{v})$ is the unique local strong solution to $(\ref{dd1})$ and $(\rho,w,\pi_{w})$ is the uniquely local solution to $(\ref{dd2})$ satisfying $(\ref{a7}).$ Then it holds that
 		\begin{equation}\label{d46}
 		\begin{aligned}
 		\int^{T}_{t_{0}}\int_{\mathbb{R}^{3}}\rho|w_{t}|^{2}dxdt+\sup_{t\in[t_{0},T]}\int_{\mathbb{R}^{3}}|\nabla w|^{2} dx\leq 2C\|u_{0}\|^{2}_{\dot{B}^{-1+\frac{3}{p}}_{p,1}},
 		\end{aligned}
 		\end{equation}
 		and
 		\begin{equation}\label{v}
 		\int^{T}_{t_{0}}\int_{\mathbb{R}^{3}}|v_{t}|^{2}dxdt+\sup_{t\in[t_{0},T]}\int_{\mathbb{R}^{3}}|\nabla v|^{2} dx\leq C\|u(t_0)\|^{2}_{\dot{H}^1},
 		\end{equation}
 		Assume that $u=v+w,$ one has
 		\begin{equation}\label{u}
 		\begin{aligned}
 		\int^{T}_{t_{0}}\int_{\mathbb{R}^{3}}|u_{t}|^{2}dxdt+\sup_{t\in[t_{0},T]}\int_{\mathbb{R}^{3}}|\nabla u|^{2} dx\leq C\left(1+\|u(t_0)\|^{2}_{\dot{H}^1}\right) ,
 		\end{aligned}
 		\end{equation}
 		where $C$ is independent of $t$.
 	\end{Lemma}
 	\begin{proof} First, using the fact that $\underline{\rho}<\rho_0<\overline{\rho}$, one deduces from the transport equation of (\ref{a1}) that $\underline{\rho}< \rho(x,t)<\overline{\rho}.$ We now decompose the proof of Lemma \ref{Lemma-5.3} into the following steps:\\
 		
 		\noindent\textbf{Step 1.} The estimates of $\|w\|_{L^{\infty}([t_0,T];L^2)}$ and $\|\nabla w\|_{{L}^{2}([t_0,T];L^2)}$.
 		
 		By using standard energy estimate to the $w$ equation of (\ref{dd2}) that
 		$$
 		\begin{aligned}
 		\frac{1}{2}\frac{d}{dt}\|\sqrt{\rho}w&\|^{2}_{L^2}+\|\nabla w\|^{2}_{L^2}
 		\leq \left|\int_{\mathbb{R}^{3}}(1-\rho)(\partial_{t}v+v\cdot\nabla v)|wdx \right|\\&+\left| \int_{\mathbb{R}^{3}}\rho w \cdot\nabla v|wdx\right|+\left| \int_{\mathbb{R}^{3}}2(\mu(\rho)-1)\mathbb{D}v:\mathbb{D}wdx\right|. 
 		\end{aligned}          
 		$$
 		Thus, we deduce that
 		\begin{equation}\label{d42}
 		\begin{aligned}
 	 \frac{1}{2}\frac{d}{dt}\|\sqrt{\rho}w\|^{2}_{L^2}+&\|\nabla w\|^{2}_{L^2}\leq \|1-\rho_0\|_{L^2} \|\partial_{t}v+v\cdot\nabla v\|_{L^\infty}\|w\|_{L^2}\\
 		&+\|\sqrt{\rho}w\|^2_{L^2}\|\nabla v\|_{L^\infty}+\|1-\mu(\rho_0)\|_{L^2}\|\nabla w\|_{L^2}\|\nabla v\|_{L^\infty},
 		\end{aligned}          
 		\end{equation}
 		where we use $\|\partial_{t} v\|_{L^{\infty}}\leq \|v\|_{L^\infty}\|\nabla v\|_{L^\infty}+\|\Delta v\|_{L^\infty}.$ Hence, one has
 		\begin{equation}\label{d43}
 		\begin{aligned}
 		\frac{d}{dt}\|\sqrt{\rho}w\|^{2}_{L^2}+\|\nabla w\|^{2}_{L^2}
 		\leq\frac{1}{2}\|\nabla w\|^2_{L^2}+&\|w\|_{L^2}\left(\|v\|_{L^\infty}\|\nabla v\|_{L^\infty}+\|\Delta v\|_{L^\infty}\right)\\
 		&+C\|\sqrt{\rho}w\|^2_{L^2}\|\nabla v\|_{L^\infty}+C\|\nabla v\|^{2}_{L^\infty}.
 		\end{aligned}
 		\end{equation}
 		Integrating in time over $[t_{0},T]$ yields
 		$$
 		\begin{aligned}
 		&\|w\|^{2}_{L^{\infty}([t_{0},T];L^{2})}+\|\nabla w\|^{2}_{L^{2}([t_{0},T];L^{2})}\\
 		\leq& C \int^{T}_{t_0}\|\nabla v\|^2_{L^\infty}dt\cdot\exp\left\{\int^{T}_{t_0}\|\nabla v\|_{L^\infty}dt+\int^{T}_{t_0}\|\Delta v\|_{L^\infty}dt\right\}.
 		\end{aligned}
 		$$
 		By virtue of Proposition \ref{Proposition-5.1},
 		\begin{equation}\label{w1}
 		\|w\|^{2}_{L^{\infty}([t_{0},T];L^{2})}+\|\nabla w\|^{2}_{L^{2}([t_{0},T];L^{2})}\leq C\|u_0\|^{2}_{\dot{B}^{-1+\frac{3}{p}}_{p,1}}.
 		\end{equation}
 		
 		\noindent\textbf{Step 2.} The estimates of $\| w_t\|_{{L}^{2}([t_0,T];L^2)}$ and $\|\nabla w\|_{L^{\infty}([t_0,T];L^2)}$.

 		Multiplying the momentum equations $(\ref{dd2})_2$ by $w_{t}$ and integrating over $\mathbb{R}^{3}$, we have
 		\begin{equation}\label{d47}
 		\begin{aligned}
 		&\int_{\mathbb{R}^{3}}\rho| w_{t}|^{2}dx+\frac{d}{dt}\int_{\mathbb{R}^{3}}\mu(\rho)\mathbb{D}w:\mathbb{D}w dx\\
 		\leq &\left|\int_{\mathbb{R}^{3}}(1-\rho)(\partial_{t}v+v\cdot\nabla v)|\partial_{t}w dx \right|+\left|\int_{\mathbb{R}^{3}}\rho (w\cdot\nabla v)|w_{t}dx \right|\\&+\left|\int_{\mathbb{R}^{3}}\operatorname{div}(2(\mu(\rho)-1)\mathbb{D}v)|\partial_{t}w dx\right|+\left|\int_{\mathbb{R}^{3}}\partial_{t}\mu(\rho)\mathbb{D}w:\mathbb{D}w dx\right|\\&+\left|\int_{\mathbb{R}^{3}}\rho (v+w)\cdot\nabla w|w_{t} dx\right|.
 		\end{aligned}
 		\end{equation}
 		\par Applying Gagliardo-Nirenberg inequality, it holds that
 		$$\begin{aligned}
 		\left|\int_{\mathbb{R}^{3}}(1-\rho)(\partial_{t}v+v\cdot\nabla v)|\partial_{t}w dx \right|&\leq
 		\|1-\rho_0\|_{L^2}\|\partial_{t}v+v\cdot\nabla v\|_{L^{\infty}}\|\sqrt{\rho}w_t\|_{L^2}\\
 		&\leq \frac{1}{8}\|\sqrt{\rho}w_t\|^{2}_{L^2}+C\|\Delta v\|^2_{L^\infty}+C\|v\|^4_{\dot{B}^{s_1}_{p,1}},
 		\end{aligned}$$
 		and
 		$$\begin{aligned}
 		\left|\int_{\mathbb{R}^{3}}\rho (w\cdot\nabla v)|w_{t}dx \right|
 		&\leq\|\sqrt{\rho}w_t\|_{L^{2}}\|w\|_{L^{2}}\|\nabla v\|_{L^{\infty}}\\
 		&\leq \frac{1}{8}\|\sqrt{\rho}w_t\|^{2}_{L^2}+C\| w\|^{2}_{L^2}\|v\|^{2}_{\dot{B}^{s_1}_{p,1}}.
 		\end{aligned}$$
 		Similarly,
 		$$\begin{aligned}
 		&\left|\int_{\mathbb{R}^{3}}\operatorname{div}(2(\mu(\rho)-1)\mathbb{D}v)|\partial_{t}w dx\right|\\
 		\leq &\|(\mu(\rho)-1)\Delta v+\nabla\mu(\rho)\cdot\mathbb{D}v\|_{L^2}\|\sqrt{\rho}w_t\|_{L^2}\\
 		\leq &\left\{\|\mu(\rho_0)-1\|_{L^2}\|\Delta v\|_{L^\infty}+4\|\nabla v\|_{L^6}\|\nabla\mu(\rho_0)\|_{L^3}\right\}\|\sqrt{\rho}w_t\|_{L^2}\\
 		\leq& \frac{1}{8}\|\sqrt{\rho}w_t\|^{2}_{L^2}+C\left(\|\mu(\rho_0-1)\|_{L^2},M\right)\left\{\|\Delta v\|^2_{L^\infty}+\|v\|^{2}_{\dot{B}^{s_1}_{p,1}}\right\}.
 		\end{aligned}$$
 		Along the same way,
 		$$\begin{aligned}
 		&\left|\int_{\mathbb{R}^{3}}\partial_{t}\mu(\rho)\mathbb{D}w:\mathbb{D}w dx\right|\leq\left|\int_{\mathbb{R}^{3}}(v+w)\cdot\nabla\mu(\rho)\mathbb{D}w:\mathbb{D}w dx\right|\\
 		\leq& \|\nabla \mu(\rho)\|_{L^3}\|\nabla w\|_{L^6}\|\nabla w\|_{L^2}\left( \|w\|_{L^\infty}+\|v\|_{L^\infty}\right)\\
 		\leq& 4\|\nabla \mu(\rho_0)\|_{L^3}\|\nabla^2 w\|_{L^2}\|\nabla w\|_{L^2}\left( \|\nabla w\|^{\frac{1}{2}}_{L^2}\|\nabla^{2}w\|^{\frac{1}{2}}_{L^2}+\|v\|_{L^\infty}\right)\\
 		\leq& C(M)\|\nabla w\|^{\frac{3}{2}}_{L^2}\|\nabla^{2}w\|^{\frac{3}{2}}_{L^2}+C(M)\|\nabla w\|_{L^2}\|\nabla^{2}w\|_{L^2}\|v\|_{\dot{B}^{s_1}_{p,1}}.
 		\end{aligned}$$
 		Here we have used the fact that
 		$$\partial_{t}\mu(\rho)+(v+w)\cdot\nabla \mu(\rho)=0,$$
 		which is a consequence of mass equation and the fact that $\operatorname{div} u=0.$ Notice that
 		$$\begin{aligned}
 		&\left|\int_{\mathbb{R}^{3}}\rho (v+w)\cdot\nabla w|w_{t} dx\right|
 		\\\leq &\|\sqrt{\rho}w_t\|_{L^2}\left\{ \|\nabla w\|_{L^3}\|w\|_{L^6}+\|\nabla w\|_{L^2}\|v\|_{L^\infty}\right\}\\
 		\leq &C\|\sqrt{\rho}w_t\|_{L^2}\|\nabla w\|^{\frac{3}{2}}_{L^2}\|\nabla^{2}w\|^{\frac{1}{2}}_{L^2}+\|\sqrt{\rho}w_t\|_{L^2}\|\nabla w\|_{L^2}\|v\|_{\dot{B}^{s_1}_{p,1}}.
 		\end{aligned}$$
 		Hence, by Young's inequality and Corollary \ref{c-2.1}, we obtain
 		\begin{equation}\label{d48}
 		\begin{aligned}
 		&\int_{\mathbb{R}^{3}}\rho| w_{t}|^{2}dx+\frac{d}{dt}\int_{\mathbb{R}^{3}}\mu(\rho)\mathbb{D}w:\mathbb{D}w dx\\
 		\leq& \frac{1}{2}\|\sqrt{\rho}w_t\|^{2}_{L^2}
 		+C\|\nabla w\|^{\frac{3}{2}}_{L^2}\left\{ \|\rho w_t\|_{L^{2}}+\|\nabla w\|_{L^{2}}+\|\nabla w\|^{3}_{L^{2}}+\|\Delta v\|_{L^\infty}+\|v\|_{\dot{B}^{s_1}_{p,1}}\right\}^{\frac{3}{2}}\\&+C\|\nabla w\|_{L^2}\| v\|_{\dot{B}^{s_1}_{p,1}}\left\{ \|\rho w_t\|_{L^{2}}+\|\nabla w\|_{L^{2}}+\|\nabla w\|^{3}_{L^{2}}+\|\Delta v\|_{L^\infty}+\|v\|_{\dot{B}^{s_1}_{p,1}}\right\}\\
 		&+C\|\sqrt{\rho}w_t\|_{L^2}\|\nabla w\|^{\frac{3}{2}}_{L^2}\left\{ \|\rho w_t\|_{L^{2}}+\|\nabla w\|_{L^{2}}+\|\nabla w\|^{3}_{L^{2}}+\|\Delta v\|_{L^\infty}+\|v\|_{\dot{B}^{s_1}_{p,1}}\right\}^{\frac{1}{2}}\\
 		&+C\|\Delta v\|^2_{L^\infty}+C\|v\|^{2}_{\dot{B}^{s_1}_{p,1}}+C\|v\|^{4}_{\dot{B}^{s_1}_{p,1}}+C\|w\|^{2}_{L^2}\|v\|^{2}_{\dot{B}^{s_1}_{p,1}}+C\|\nabla w\|^{2}_{L^2}\|v\|^{2}_{\dot{B}^{s_1}_{p,1}},
 		\end{aligned}
 		\end{equation}
 		which yields that
 		\begin{equation}\label{d49}
 		\begin{aligned}
 		&\int_{\mathbb{R}^{3}}\rho| w_{t}|^{2}dx+\frac{d}{dt}\int_{\mathbb{R}^{3}}\mu(\rho)\mathbb{D}w:\mathbb{D}w dx\\
 		\leq &\frac{7}{8}\|\sqrt{\rho}w_t\|^{2}_{L^2}+C\|\nabla w\|^{6}_{L^2}+C\|\nabla w\|^{4}_{L^2}+C\|\nabla w\|^{3}_{L^2}\\&+C\|\Delta v\|^2_{L^\infty}+C\|v\|^{2}_{\dot{B}^{s_1}_{p,1}}+C\|v\|^{4}_{\dot{B}^{s_1}_{p,1}}+C\|w\|^{2}_{L^2}\|v\|^{2}_{\dot{B}^{s_1}_{p,1}}.
 		\end{aligned}
 		\end{equation} 
 		With the help of Proposition \ref{Proposition-5.1}, we obatin for $t\in[t_0,T],$
 		\begin{equation}\label{dd50}
 		\begin{aligned}
 		&\int_{\mathbb{R}^{3}}\rho| w_{t}|^{2}dx+\frac{d}{dt}\int_{\mathbb{R}^{3}}\mu(\rho)\mathbb{D}w:\mathbb{D}w dx\\
 		\leq& C\|\nabla w\|^{6}_{L^2}+C\|\nabla w\|^{2}_{L^2}+C\|\Delta v\|^2_{L^\infty}+C\|v\|^{2}_{\dot{B}^{s_1}_{p,1}}+C\|w\|^{2}_{L^2}\|v\|^{2}_{\dot{B}^{s_1}_{p,1}}.
 		\end{aligned}
 		\end{equation}
 		Integrating with respect to time on $[t_{0}, T]$ and using (\ref{w1}) and gives
 		$$
 		\begin{aligned}
 		\int^{T}_{t_{0}}\int_{\mathbb{R}^{3}}&\rho| w_{t}|^{2}dxdt+\sup_{t\in[t_{0},T]}\int_{\mathbb{R}^{3}}|\nabla w|^{2} dx\leq C\|u_0\|^{2}_{\dot{B}^{-1+\frac{3}{p}}_{p,1}}\\&+C\int^{T}_{t_{0}}\|\nabla w\|^{6}_{L^{2}}dt+C\int^{T}_{t_{0}}\|\Delta v\|^{2}_{L^\infty}dt+C\int^{T}_{t_{0}}\|v\|^{2}_{\dot{B}^{s_1}_{p,1}}dt.
 		\end{aligned}
 		$$
 		Applying Proposition \ref{Proposition-5.1} and Gronwall's inequality, it implies that
 		$$
 		\begin{aligned}
 		&\int^{T}_{t_{0}}\int_{\mathbb{R}^{3}}\rho| w_{t}|^{2}dxdt+\sup_{t\in[t_{0},T]}\int_{\mathbb{R}^{3}}|\nabla w|^{2} dx
 		\leq C\|u_0\|^{2}_{\dot{B}^{-1+\frac{3}{p}}_{p,1}}\cdot\exp\{ C\int^{T}_{t_{0}}\|\nabla w\|^{4}_{L^{2}}dt\}.
 		\end{aligned}
 		$$
 	Under the assumptions that
 		\begin{equation}\label{d50}
 		\|u_0\|_{\dot{B}^{-1+\frac{3}{p}}_{p,1}}\leq \varepsilon_{1},
 		\end{equation}
 		and 
 		\begin{equation}
 		\sup_{t\in[t_{0},T]}\|\nabla w\|^{2}_{L^{2}}\leq  4C\|u_0\|^{2}_{\dot{B}^{-1+\frac{3}{p}}_{p,1}}\leq 1,
 		\end{equation}
 		then
 		\begin{equation}\label{d51}
 		\begin{aligned}
 		\int^{T}_{t_{0}}\|\nabla w\|^{4}_{L^{2}}dt&\leq \sup_{t\in[t_{0},T]}\|\nabla w\|^{2}_{L^{2}}\cdot\int^{T}_{t_{0}}\|\nabla w\|^{2}_{L^{2}}dt\\&
 		\leq 4C\|u_0\|^{4}_{\dot{B}^{-1+\frac{3}{p}}_{p,1}}\leq \|u_0\|^{2}_{\dot{B}^{-1+\frac{3}{p}}_{p,1}}.
 		\end{aligned}
 		\end{equation}
 		Hence, we arrive at
 		\begin{equation}\label{d52}
 		\begin{aligned}
 		\int^{T}_{t_{0}}\int_{\mathbb{R}^{3}}\rho| w_{t}|^{2}dxdt+\sup_{t\in[t_{0},T]}\int_{\mathbb{R}^{3}}|\nabla w|^{2} dx\leq 2C\|u_0\|^{2}_{\dot{B}^{-1+\frac{3}{p}}_{p,1}}.
 		\end{aligned}
 		\end{equation}
 		Choose some small positive constant $	\varepsilon_{1}=\min\{\sqrt{\frac{1}{4C}}, \sqrt{\frac{\ln2}{C}}\}$, which is clear that (\ref{d52}) holds, provided (\ref{d50}) holds.
 		\vskip 0.5cm
 		\noindent\textbf{Step 3.} The estimates of $\| v_t\|_{{L}^{2}([t_0,T];L^2)}$ and $\|\nabla v\|_{L^{\infty}([t_0,T];L^2)}$.
 		\vskip 0.5cm
 		Multiplying the classical Navier-Stokes equation $(\ref{a11})_1$ by $v_{t}$ and integrating over $\mathbb{R}^{3}$ yield
 		\begin{equation}\label{u1}
 		\begin{aligned}
 		\int_{\mathbb{R}^{3}}| v_{t}|^{2}dx+\frac{1}{2}\frac{d}{dt}\int_{\mathbb{R}^{3}}|\nabla v|^{2}dx
 		&=-\int_{\mathbb{R}^3} v\cdot\nabla v\cdot v_t dx\\
 		&\leq \frac{1}{2}\|v_t\|^2_{L^2}+C\|\nabla v\|^2_{L^2}\|v\|^2_{L^\infty},
 		\end{aligned}
 		\end{equation}
 		from which, we arrive at
 		\begin{equation}
 		\int^{T}_{t_{0}}\int_{\mathbb{R}^{3}}|v_{t}|^{2}dxdt+\sup_{t\in[t_{0},T]}\int_{\mathbb{R}^{3}}|\nabla v|^{2} dx\leq C\|u(t_0)\|^{2}_{\dot{H}^1}.
 		\end{equation}
 		By virtue of $u\stackrel{\text { def }}{=}v+w$ and (\ref{d46}), we completes the proof of Lemma.
 	\end{proof}
 	
 	\subsection{The \emph{a priori} time-decay estimates to (\ref{a1})}  
 	
 	\par In this subsection we note that the condition $u_0\in\dot{H}^{-2\delta}$ with $\delta\in ]1/2,3/4[$ here is necessary to give the equation a decay, since there is no Poincar\'{e}'s~inequality similar to that of the bounded domain in \cite{2015HW}.
 	
 	\begin{Lemma}\label{Lemma-5.4}
 		Suppose $(\rho,u,\pi)$ is the unique local strong solution to $(\ref{a1})$ satisfying $(\ref{a7}).$ Then under the assumptions of Proposition \ref{Proposition-5.5}, we have 
 		\begin{equation}\label{d53}
 		\begin{aligned}
 		\int^{T}_{t_{0}}\int_{\mathbb{R}^{3}}t|u_{t}|^{2}dxdt+\sup_{t\in[t_{0},T]}\int_{\mathbb{R}^{3}}t|\nabla u|^{2} dx\leq  \exp\{C\mathcal{H}_{0}\},
 		\end{aligned}
 		\end{equation}
 		where $\mathcal{H}_{0}$ is given by (\ref{h}) and $C$ depends on $\underline{\mu}$, $\overline{\mu}$, $M$, $\|\rho_0-1\|_{B^{\frac{3}{q}}_{q,1}}$ and $\|u_0\|_{\dot{H}^{-2\delta}}$.
 	\end{Lemma}
 	\begin{proof} 
 		Multiplying the momentum equations by $u_{t}$ and integrating over $\mathbb{R}^{3}$ yield
 		\begin{equation}\label{dd52}
 		\begin{aligned}
 		&\int_{\mathbb{R}^{3}}\rho| u_{t}|^{2}dx+\frac{d}{dt}\int_{\mathbb{R}^{3}}\mu(\rho)\mathbb{D}u:\mathbb{D}u dx\\
 		\leq&|\int_{\mathbb{R}^{3}}\rho u\cdot\nabla u\cdot u_{t}dx|+\int_{\mathbb{R}^{3}}|\nabla\mu(\rho)|\cdot|u|\cdot|\nabla u|^{2} dx.
 		\end{aligned}
 		\end{equation}
 		Along the same way as (\ref{d49}) and using Corollary \ref{Corollary-5.1} to estimate each term on the right hand of (\ref{dd52}), we obtain
 		\begin{equation}\label{dd53}
 		\begin{aligned}
 		&\int_{\mathbb{R}^{3}}\rho|u_{t}|^{2}dx+\frac{d}{dt}\int_{\mathbb{R}^{3}}\mu(\rho)\mathbb{D}u:\mathbb{D}u dx\\ \leq&\frac{7}{8}\|\sqrt{\rho}u_{t}\|^{2}_{L^{2}}+C\|\nabla u\|^{3}_{L^{2}}+C\|\nabla u\|^{4}_{L^{2}}+C\|\nabla u\|^{6}_{L^{2}},
 		\end{aligned}
 		\end{equation}
 		Multiplying (\ref{dd53}) by $t,$ as shown in the last proof, one has 
 		\begin{equation}\label{d54}
 		\begin{aligned}
 		&\int^{T}_{t_{0}}t\|u_{t}\|^{2}_{L^{2}}dt+\sup_{t\in[t_{0},T]}t\|\nabla u\|^{2}_{L^{2}}\leq C\|\nabla u(t_0)\|^2_{L^2}+C\int^{T}_{t_{0}}\|\nabla u\|^{2}_{L^{2}}dt
 		\\ &\quad+ C\int^{T}_{t_{0}}t\|\nabla u\|^{3}_{L^2}dt
 		+C\int^{T}_{t_{0}}t\|\nabla u\|^{4}_{L^{2}}dt+C\int^{T}_{t_{0}}t\|\nabla u\|^{6}_{L^{2}}dt.
 		\end{aligned}
 		\end{equation}
 		Applying Gronwall's inequality,
 		\begin{equation}\label{d56}
 		\begin{aligned}
 		&\int^{T}_{t_{0}}t\|u_{t}\|^{2}_{L^{2}}dt+\sup_{t\in[t_{0},T]}t\|\nabla u\|^{2}_{L^{2}}\\
 		\leq&  C\| u(t_{0})\|^{2}_{H^{1}}\exp\left\{\int^{T}_{t_{0}}\|\nabla u\|_{L^2}dt+\int^{T}_{t_{0}}\|\nabla u\|^{2}_{L^2}dt\int^{T}_{t_{0}}\|\nabla u\|^{4}_{L^2}dt\right\}.
 		\end{aligned}
 		\end{equation}
 	Here we have used the Corollary \ref{c-4.2}, so for $\delta_{-}>\frac{1}{2},$ we get that
 		\begin{equation}\label{d55}
 		\int^{T}_{t_{0}}\|\nabla u\|_{L^2}dt\leq\left(\int^{T}_{t_{0}}\|t^{\delta_{-}}\nabla u\|^{2}_{L^2}dt\right)^{\frac{1}{2}}\left(\int^{T}_{t_{0}}t^{-2\delta_{-}}dt\right)^{\frac{1}{2}}  \leq C\mathcal{H}_{0}.
 		\end{equation}   
 		This leads to (\ref{d53}).
 	\end{proof}
 	\begin{Lemma}\label{Lemma-5.5}
 		Suppose $(\rho,u,\pi)$ is the uniquely local  solution to $(\ref{a1})$ satisfying $(\ref{a7}).$ Then under the assumptions of Proposition \ref{Proposition-5.5}, we have 
 		\begin{equation}\label{d57}
 		\begin{aligned}
 		\sup_{t\in[t_{0},T]}t\|u_{t}\|^{2}_{L^{2}}+\int^{T}_{t_{0}}t\|\nabla u_{t}\|^{2}_{L^{2}}dt\leq \|u (t_0)\|^2_{H^2}\exp\{C\mathcal{H}_{0}\},
 		\end{aligned}
 		\end{equation}
 		and
 		\begin{equation}\label{d58}
 		\begin{aligned}
 		\sup_{t\in[t_{0},T]}t^2\|u_{t}\|^{2}_{L^{2}}+\int^{T}_{t_{0}}t^{2}\|\nabla u_{t}\|^{2}_{L^{2}}dt\leq \|u (t_0)\|^2_{H^2}\exp\{C\mathcal{H}_{0}\},
 		\end{aligned}
 		\end{equation}
 		where $\mathcal{H}_{0}$ is given by (\ref{h}) and $C$ depends on $\underline{\mu}$, $\overline{\mu}$, $M$, $\|\rho_0-1\|_{B^{\frac{3}{q}}_{q,1}}$ and $\|u_0\|_{\dot{H}^{-2\delta}}$.
 	\end{Lemma}
 	\begin{proof}
 		Applying $t$-derivative to the momentum equations, we easily know that
 		\begin{equation}\label{d59}
 		\begin{aligned}
 		\rho\partial_{tt}u+\rho u\cdot\nabla u_{t}-\operatorname{div}(2\mu(\rho)\mathbb{D}u_{t})+\nabla\pi_{t}=-\rho_t u_t- \left(\rho u\right)_t\cdot\nabla u+
 		\operatorname{div}(2\partial_{t}\mu(\rho)\mathbb{D}u).
 		\end{aligned}
 		\end{equation}
 		Multiplying (\ref{d59}) by $u_{t}$ and integrating over $\mathbb{R}^{3}$, we get after integration by parts that
 		\begin{equation}\label{d60}
 		\begin{aligned}
 		&\frac{1}{2}\frac{d}{dt}\int_{\mathbb{R}^{3}}\rho|u_{t}|^{2}dx+2\int_{\mathbb{R}^{3}}\mu(\rho)\mathbb{D}u_{t}:\mathbb{D}u_{t}dx\\=&-\int_{\mathbb{R}^{3}}\rho_t u_t|u_{t}dx-\int_{\mathbb{R}^{3}} \left(\rho u\right)_t\cdot\nabla u|u_{t}dx
 		-2\int_{\mathbb{R}^{3}}\partial_{t}\mu(\rho)\mathbb{D}u\cdot\nabla u_{t}dx\\\stackrel{\text { def }}{=}&\sum_{i=1}^{3} J_{i}.
 		\end{aligned}
 		\end{equation} 
 		\par Now, we will use the Gagliardo-Nirenberg inequality to estimate each term on the right hand of (\ref{d60}). First, with the help of the mass equation, one has
 		\begin{equation}\label{j1}
 		\begin{aligned}
 		J_1=-2\int_{\mathbb{R}^{3}}\rho u\cdot\nabla u_t\cdot u_{t}dx&\leq C\|u\|_{L^6}\|\nabla u_t\|_{L^2}\|u_t\|_{L^3}\\
 		&\leq C\|\nabla u\|_{L^2}\|u_t\|_{L^2}^{\frac{1}{2}}\|\nabla u_t\|^{\frac{3}{2}}_{L^2}\\
 		&\leq \frac{1}{8} \underline{\mu} \|\nabla u_t\|^{2}_{L^2}+C\|u_t\|^{2}_{L^2}\|\nabla u\|^{4}_{L^2}.
 		\end{aligned}
 		\end{equation} 
 		Taking into account the mass equation again, we arrive at
 		$$\begin{aligned}
 		J_2&=-\int_{\mathbb{R}^{3}}\rho_t u \cdot\nabla u|u_{t}dx-\int_{\mathbb{R}^{3}}\rho u_t\cdot\nabla u|u_{t}dx\\
 		& \leq  \int_{\mathbb{R}^{3}} \rho|u| \cdot|\nabla u|^2 \cdot|u_t| dx+C  \int_{\mathbb{R}^{3}} \rho|u|^2 \cdot|\nabla^2 u| \cdot|u_t| d x \\
 		& \quad+ \int_{\mathbb{R}^{3}}\rho|u|^2 \cdot|\nabla u| \cdot|\nabla u_t| d x+ \int_{\mathbb{R}^{3}} \rho|u_t|^2 \cdot|\nabla u| d x.
 		\end{aligned}$$
 		Hence, it follows from Sobolev embedding inequality and Gagliardo–Nirenberg inequality, that
 		$$
 		\begin{aligned}
 		\int_{\mathbb{R}^{3}} \rho|u| \cdot|\nabla u|^2 \cdot|u_t| d x & \leq C\|u_t\|_{L^6} \|u\|_{L^6} \|\nabla u\|_{L^3}^2 \\
 		& \leq C \|\nabla u_t\|_{L^2} \|\nabla u\|_{L^2}^2 \|\nabla^2 u\|_{L^2} \\
 		& \leq \frac{1}{8} \underline{\mu}\|\nabla u_t\|_{L^2}^2+C\|\nabla u\|_{L^2}^4\|\nabla^2 u\|_{L^2}^2,
 		\end{aligned}
 		$$
 		and, 
 		$$
 		\begin{aligned}
 		\int_{\mathbb{R}^{3}} \rho|u|^2 \cdot|\nabla^2 u| \cdot|u_t| d x & \leq C \|u_t\|_{L^6} \|\nabla^2 u\|_{L^2} \|u\|_{L^6}^2 \\
 		& \leq \frac{1}{8} \underline{\mu}\|\nabla u_t\|_{L^2}^2+C\|\nabla^2 u\|^2_{L^2} \|\nabla u\|_{L^2}^4.
 		\end{aligned}
 		$$
 		Similarly, it holds that
 		$$
 		\begin{aligned}
 		\int_{\mathbb{R}^{3}}\rho|u|^2 \cdot|\nabla u| \cdot|\nabla u_t| d x & \leq C \|\nabla u_t\|_{L^2} \|\nabla u\|_{L^6} \|u\|_{L^6}^2 \\
 		& \leq \frac{1}{8} \underline{\mu} \|\nabla u_t\|_{L^2}^2+ C\|\nabla^2 u\|_{L^2}^2 \|\nabla u\|_{L^2}^4,
 		\end{aligned}
 		$$
 		and,
 		$$\begin{aligned}
 		\int_{\mathbb{R}^{3}} \rho|u_t|^2 \cdot|\nabla u| d x&\leq C\|u_t\|^2_{L^4} \|\nabla u\|_{L^2} \\
 		&\leq  C\|u_t\|^{\frac{1}{2}}_{L^2} \|\nabla u_t\|^{\frac{3}{2}}_{L^2}\|\nabla u\|_{L^2}\\
 		&\leq \frac{1}{8} \underline{\mu} \|\nabla u_t\|_{L^2}^2+C\|u_t\|^2_{L^2}\|\nabla u\|^4_{L^2}.
 		\end{aligned}$$
 		By means of Corollary \ref{Corollary-5.1}, we deduce that
 		\begin{align}\label{j2}
 		\begin{aligned}
 		|J_2|&\leq \frac{1}{2} \underline{\mu}\|\nabla u_t\|_{L^2}^2+C\|\nabla u\|_{L^2}^4\|\nabla^2 u\|_{L^2}^2+C\|u_t\|^2_{L^2}\|\nabla u\|^4_{L^2}\\
 		&\leq \frac{1}{2} \underline{\mu}\|\nabla u_t\|_{L^2}^2 +C\|u_t\|^2_{L^2}\|\nabla u\|^4_{L^2}+C\|\nabla u\|_{L^2}^6+C\|\nabla u\|_{L^2}^{10}.
 		\end{aligned}
 		\end{align}  
 		Along the same way, we obtain
 		$$
 		\begin{aligned}
 		\left|J_3\right| &\leq C \int|u| \cdot|\nabla \mu(\rho)|\cdot |\mathbb{D}u| \cdot\left|\nabla u_t\right| d x \\
 		& \leq C\|\nabla \mu(\rho)\|_{L^3}\|u\|_{L^{\infty}}\|\nabla u\|_{L^{6}}\left\|\nabla u_t\right\|_{L^2} \\
 		& \leq 4C\|\nabla \mu(\rho_0)\|_{L^3}\|u\|_{L^6}^{1 / 2}\|\nabla u\|_{L^6}^{1 / 2}\left\|\nabla^2 u\right\|_{L^2}\left\|\nabla u_t\right\|_{L^2} \\
 		& \leq \frac{1}{8} \underline{\mu}\left\|\nabla u_t\right\|_{L^2}^2+C(M)\|\nabla u\|_{L^2}\left\|\nabla^2 u\right\|_{L^2}^3,
 		\end{aligned}
 		$$
 		thus, it follows from Corollary \ref{Corollary-5.1} that
 		\begin{equation}\label{j3}
 		\begin{aligned}
 		|J_3|\leq\frac{1}{8} \underline{\mu}\left\|\nabla u_t\right\|_{L^2}^2+C\|\nabla u\|_{L^2}\left\|u_t\right\|_{L^2}^3+C\|\nabla u\|_{L^2}^4+C\|\nabla u\|^{10}_{L^2}.
 		\end{aligned}
 		\end{equation}
 	 Substituting (\ref{j1})-(\ref{j3}) into (\ref{d60}) gives
 		\begin{equation}\label{d68}
 		\begin{aligned}
 		\frac{d}{dt}\int_{\mathbb{R}^{3}}\rho|u_{t}|^{2}dx+&\|\nabla u_{t}\|^{2}_{L^{2}} \leq C\|u_t\|^2_{L^2}\|\nabla u\|^4_{L^2}+C\|u_t\|^4_{L^2}\\&+C\|\nabla u\|_{L^2}^4+C\|\nabla u\|_{L^2}^6+C\|\nabla u\|^{10}_{L^2}.
 		\end{aligned}
 		\end{equation}
 		Multiplying (\ref{d68}) by $t$ gives that
 		$$
 		\begin{aligned}
 		\frac{d}{dt}\|\sqrt{t}&\sqrt{\rho}u_{t}\|^{2}_{L^{2}}+\|\sqrt{t}\nabla u_{t}\|^{2}_{L^{2}}\lesssim \|u_{t}\|^{2}_{L^{2}}+\|\sqrt{t}u_{t}\|^{2}_{L^{2}}\|\nabla u\|^{4}_{L^{2}}\\&+ \|\sqrt{t}u_{t}\|^{2}_{L^{2}}\|u_{t}\|^{2}_{L^{2}}+\|\sqrt{t}\nabla u\|^{2}_{L^{2}}\{\|\nabla u\|^{2}_{L^{2}}+\|\nabla u\|^{4}_{L^{2}}+\|\nabla u\|^{8}_{L^{2}}\}.
 		\end{aligned}
 		$$
 		Integrating with respect to time on $[t_{0},T]$ and using Gronwall's inequality, we obtain
 		\begin{equation}\label{d69}
 		\begin{aligned}
 		\sup_{t\in[t_{0},T]}t\|&u_{t}\|^{2}_{L^{2}}+\int^{T}_{t_{0}}t\|\nabla u_{t}\|^{2}_{L^{2}}dt\lesssim\Big\{\|u_{t}(t_0)\|^{2}_{L^{2}}+\int^{T}_{t_{0}}\|u_t\|^{2}_{L^{2}}dt\\&+\int^{T}_{t_{0}}\|\sqrt{t}\nabla u\|^{2}_{L^{2}}(\|\nabla u\|^{2}_{L^{2}}+\|\nabla u\|^{4}_{L^{2}}+\|\nabla u\|^{8}_{L^{2}})dt\Big\}\\&\quad\times\exp\left\{\int^{T}_{t_{0}}\|\nabla u\|^{4}_{L^{2}}dt+\int^{T}_{t_0}\|u_t\|^{2}_{L^{2}}dt\right\}.
 		\end{aligned}
 		\end{equation}
 		According to Lemma \ref{Lemma-5.3} and \ref{Lemma-5.4}, we derive
 		\begin{equation}\label{d70}
 		\begin{aligned}
 		&\int^{T}_{t_{0}}\|\sqrt{t}\nabla u\|^{2}_{L^{2}}\left(\|\nabla u\|^{2}_{L^{2}}+\|\nabla u\|^{4}_{L^{2}}+\|\nabla u\|^{8}_{L^{2}}\right)dt\\
 		\leq &\sup_{t\in[t_{0},T]}\|\sqrt{t}\nabla u\|^{2}_{L^{2}}\cdot\left(1+\sup_{t\in[t_{0},T]}\|\nabla u\|^{2}_{L^{2}}+\sup_{t\in[t_{0},T]}\|\nabla u\|^{6}_{L^{2}}\right)\\&\quad\times\int^{T}_{t_{0}}\|\nabla u\|^{2}_{L^{2}}dt\leq \exp\{C\mathcal{H}_{0}\}.
 		\end{aligned}
 		\end{equation}
 		Whereas taking $L^2$-norm of the $u_t$ at $t=t_0$ and using (\ref{d1}), it gives rise to
 		\begin{equation}\label{d71}
 		\begin{aligned}
 		\|u_t(t_0)\|_{L^2}
 		\lesssim &\|\rho u\cdot\nabla u (t_0)\|_{L^2}+\|\operatorname{div}(2\mu(\rho)\mathbb{D}u)(t_0)\|_{L^2}\\
 		\lesssim & \|u(t_0)\|_{L^6}\|\nabla u(t_0)\|_{L^3}+\|\nabla\mu(\rho)(t_0)\|_{L^3}\|\nabla u(t_0)\|_{L^{6}}+\|\Delta u(t_0)\|_{L^2}\\
 		\lesssim & \|\nabla u(t_0)\|^{\frac{3}{2}}_{L^2}\|\nabla^{2}u(t_0)\|^{\frac{1}{2}}_{L^2}+C(M)\|\nabla^2 u(t_0)\|_{L^2}\\
 		\lesssim & \|u (t_0)\|_{H^2}.
 		\end{aligned}
 		\end{equation}
 		Plugging (\ref{d70}) and (\ref{d71}) into (\ref{d69}), we have
 		\begin{equation}\label{d72}
 		\begin{aligned}
 		\sup_{t\in[t_{0},T]}t\|u_{t}\|^{2}_{L^{2}}+\int^{T}_{t_{0}}t\|\nabla u_{t}\|^{2}_{L^{2}}dt\leq\|u (t_0)\|^2_{H^2}\exp\{C\mathcal{H}_{0}\}.
 		\end{aligned}
 		\end{equation}
 		On the other hand, multiplying (\ref{d68}) by $t^2,$ one has
 		$$
 		\begin{aligned}
 		\frac{d}{dt}\|t\sqrt{\rho}u_{t}&\|^{2}_{L^{2}}+\|t\nabla u_{t}\|^{2}_{L^{2}}\lesssim \|\sqrt{t}u_{t}\|^{2}_{L^{2}}+ \|tu_{t}\|^{2}_{L^{2}}\|\nabla u\|^{4}_{L^{2}}\\&+\|tu_{t}\|^{2}_{L^{2}}\|u_t\|^{2}_{L^{2}}+\|\sqrt{t}\nabla u\|^{4}_{L^{2}}\{1+\|\nabla u\|^{2}_{L^{2}}+\|\nabla u\|^{6}_{L^{2}}\}.
 		\end{aligned}
 		$$
 		Owing to Corollary \ref{c-4.2}, we get for $\delta>\frac{1}{2},$ 
 		$$
 		\begin{aligned}
 		\int^{T}_{t_0}\|\sqrt{t}\nabla u\|^{4}_{L^{2}}dt\leq \sup_{t\in[t_0,T]}\|\sqrt{t}\nabla u\|^{2}_{L^{2}}\int^{T}_{t_0}t^{1-2\delta_{-}}\|t^{\delta_{-}}\nabla u\|^{2}_{L^{2}}dt\leq \exp\{C\mathcal{H}_{0}\}.
 		\end{aligned}$$ 
 		From which and Gronwall's inequality, we can deduce
 		\begin{equation}\label{d74}
 		\begin{aligned}
 		&\sup_{t\in[t_{0},T]}t^2\|u_{t}\|^{2}_{L^{2}}+\int^{T}_{t_{0}}t^2\|\nabla u_{t}\|^{2}_{L^{2}}dt\leq \|u (t_0)\|^2_{H^2} \exp\{C\mathcal{H}_{0}\}.
 		\end{aligned}
 		\end{equation}
 		This completes the proof of Lemma \ref{Lemma-5.5}.
 	\end{proof}
 	
 	\subsection{The $L^{1}([t_0,T];L^\infty(\mathbb{R}^{3}))$ estimate for $\nabla u$}
 	To close the $L^\infty([t_0,T];L^r(\mathbb{R}^{3}))$ estimate of $\nabla \mu(\rho)$ in this subsection, we need the $L^{1}([t_0,T];L^\infty(\mathbb{R}^{3})))$ estimate of $\nabla u$ to be small enough, i.e:
 	\begin{Lemma}\label{Lemma-5.6}
 		Let $n\in]2,\frac{6\delta_{-}}{1+\delta_{-}}[$ for $\delta\in]1/2,3/4[$ and $\alpha=\frac{(m-3)n}{3m-3n+nm}$. Assume that $(\rho,u,\pi)$ is the unique local strong solution to $(\ref{a1})$ satisfying $(\ref{a7}).$ Then under the assumptions of Proposition \ref{Proposition-5.5}, we have 
 		\begin{equation}\label{d75}
 		\begin{aligned}
 		\|\nabla u\|_{L^{1}([t_{0},T];L^\infty)}\leq \overline{C}\|u_{0}\|^{\frac{3\alpha(n-2)}{2n}}_{\dot{B}^{-1+\frac{3}{p}}_{p,1}}\exp\{C\mathcal{H}_{0}\}\quad{\rm and}\quad \overline{C}\overset{\text{def}}{=} 1+\|u(t_0)\|_{H^2},
 		\end{aligned}
 		\end{equation}
 		where $\mathcal{H}_{0}$ is given by (\ref{h}) and $C$ depends on $\underline{\mu}$, $\overline{\mu}$, $M$, $\|\rho_0-1\|_{B^{\frac{3}{q}}_{q,1}}$ and $\|u_0\|_{\dot{H}^{-2\delta}}$.
 	\end{Lemma}
 	\begin{proof} Applying the Gagliardo-Nirenberg inequality, we have
 		\begin{equation}\label{d80}
 		\begin{aligned}
 		\|\nabla u\|_{L^{1}([t_{0},T];L^\infty)}&\lesssim\|\nabla u\|^{\alpha}_{L^{1}([t_{0},T];L^{n})}\|\nabla^{2}u\|^{1-\alpha}_{L^{1}([t_{0},T];L^{m})},
 		\end{aligned}
 		\end{equation}
 		with $$0=\frac{\alpha}{n}+(1-\alpha)(\frac{1}{m}-\frac{1}{3})\quad{\rm or}\quad\alpha=\frac{(m-3)n}{3m-3n+nm},$$
 		where $n$ is a positive constant which will be determined later and $m>3$. We now decompose the proof of Lemma \ref{Lemma-5.6} into the following steps:\\
 		
 		\noindent\textbf{Step 1.} The estimate of $\|\nabla^{2}u\|_{L^{1}([t_{0},T];L^{m})}$.
 		
 		By virtue of Lemma \ref{Lemma-5.1}, one has for $m\in(2,\min\{r,6\})$ with $r\in]3,+\infty[$,
 		\begin{equation}\label{d76}
 		\|\nabla^2 u\|_{L^m}\lesssim \|\nabla u\|_{L^2}+\|F\|_{L^m}+\|(-\Delta)^{-1}\operatorname{div}F\|_{L^2}
 		\end{equation}
 		with $F=-\rho \partial_t u-\rho u\cdot\nabla u.$ Thus, for any $\eta>0$, we have
 		$$\begin{aligned}
 		\|F\|_{L^m}\lesssim &\|\rho u_t\|_{L^m}+\|u\|_{L^6}\|\nabla u\|_{L^{\frac{6m}{6-m}}}\\
 		\lesssim& \|u_t\|^{\frac{6-m}{2m}}_{L^2}\|\nabla u_t\|^{\frac{3m-6}{2m}}_{L^2}+\|\nabla u\|^{\frac{6(m-1)}{5m-6}}_{L^2}\|\nabla^2 u\|^{\frac{4m-6}{5m-6}}_{L^m}\\
 		\lesssim & \eta\|\nabla^2 u\|_{L^m}+\|u_t\|^{\frac{6-m}{2m}}_{L^2}\|\nabla u_t\|^{\frac{3m-6}{2m}}_{L^2}+\|\nabla u\|^{\frac{6(m-1)}{m}}_{L^2}.
 		\end{aligned}$$
 		Thanks to $\operatorname{div} u_t=0$, one has
 		$$\begin{aligned}
 		\|(-\Delta)^{-1}\operatorname{div}F\|_{L^2}\lesssim& \|(-\Delta)^{-1}\operatorname{div}((1-\rho)u_t)\|_{L^2}+\|\rho u\cdot\nabla u\|_{L^{\frac{6}{5}}}\\& \|(1-\rho)u_t\|_{L^{\frac{6}{5}}}+\|u\|_{3}\|\nabla u\|_{L^{2}}\\
 		\lesssim & \|1-\rho_0\|_{L^2}\|u_t\|^{\frac{1}{2}}_{L^2}\|\nabla u_t\|^{\frac{1}{2}}_{L^2}+\|u\|^{\frac{1}{2}}_{L^2}\|\nabla u\|^{\frac{3}{2}}_{L^2},
 		\end{aligned}$$
 		where we have used $L^{\frac{6}{5}}\hookrightarrow \dot{W}^{-1,2}$. Substituting the above inequality into (\ref{d76}) and taking $\eta>0$ sufficiently small, we arrive at
 		$$
 		\begin{aligned}
 		\|\nabla^2 u\|_{L^m}\lesssim \|\nabla u\|_{L^2}+\|u_t\|^{\frac{6-m}{2m}}_{L^2}\|\nabla u_t\|^{\frac{3m-6}{2m}}_{L^2}+\|\nabla u\|^{\frac{6(m-1)}{m}}_{L^2}+\|u_t\|^{\frac{1}{2}}_{L^2}\|\nabla u_t\|^{\frac{1}{2}}_{L^2}+\|u\|^{\frac{1}{2}}_{L^2}\|\nabla u\|^{\frac{3}{2}}_{L^2},
 		\end{aligned}
 		$$
 		and it is easy to observe that
 		\begin{equation}\label{d78}
 		\begin{aligned}
 		&\int^{T}_{t_{0}}\|\nabla^{2} u\|_{L^{m}}dt\lesssim \int^{T}_{t_{0}}\|\nabla u\|_{L^2}dt+\int^{T}_{t_{0}}\|u_t\|^{\frac{6-m}{2m}}_{L^2}\|\nabla u_t\|^{\frac{3m-6}{2m}}_{L^2}dt\\&+\int^{T}_{t_{0}}\|\nabla u\|^{\frac{6(m-1)}{m}}_{L^2}dt+\int^{T}_{t_{0}}\|u_t\|^{\frac{1}{2}}_{L^2}\|\nabla u_t\|^{\frac{1}{2}}_{L^2}dt+\int^{T}_{t_{0}}\|u\|^{\frac{1}{2}}_{L^2}\|\nabla u\|^{\frac{3}{2}}_{L^2}dt\stackrel{\text { def }}{=}\sum^{5}_{i=1} I_{i}.
 		\end{aligned}
 		\end{equation}
 		First, due to Corollary \ref{c-4.2} and $\delta_{-}>\frac{1}{2}$, we get
 		$$
 		\begin{aligned}
 		I_1\leq \|t^{\delta_{-}}\nabla u\|_{L^{2}([t_{0},T];L^{2})}\left(\int^{T}_{t_{0}}t^{-2\delta_{-}}dt\right)^{\frac{1}{2}}\leq C\mathcal{H}_{0}.
 		\end{aligned}
 		$$
 		Applying Lemma \ref{Lemma-5.5}, we arrive at
 		$$
 		\begin{aligned}
 		&I_2\leq \int^{T}_{t_{0}}\|t\partial_{t}u\|^{\frac{6-m}{2m}}_{L^{2}}\|t\nabla u_{t}\|^{\frac{3(m-2)}{2m}}_{L^{2}}\cdot t^{-1}dt
 		\leq \bigg(\sup_{t\in[t_{0},T]}\|t\partial_t u\|_{L^{2}}dt\bigg)^{\frac{6-m}{2m}}\\&\quad\times\bigg(\int^{T}_{t_{0}}\|t\nabla u_{t}\|^{2}_{L^{2}}dt\bigg)^{\frac{3(m-2)}{4m}}\cdot\bigg(\int^{T}_{t_{0}}t^{-\frac{4m}{m+6}}dt\bigg)^{\frac{6+m}{4m}}
 		\leq \|u(t_0)\|_{H^2}\exp\{C\mathcal{H}_{0}\}.
 		\end{aligned}
 		$$
 		Similarly, 
 		$$ 
 		I_3\leq \sup_{t\in[t_{0},T]}\|\nabla u\|^{\frac{4m-6}{m}}_{L^{2}}\cdot\int^{T}_{t_{0}}\|\nabla u\|^{2}_{L^{2}}dt\leq \exp\{C\mathcal{H}_{0}\}.
 		$$
 		The same estimate holds for $I_4$. Notice that
 		$$
 		\begin{aligned}
 		I_5\leq C\sup_{t\in[t_0,T]}\|t^\delta u\|^{\frac{1}{2}}_{L^{2}}\left(\int^{T}_{t_{0}}\|t^{\delta_{-}}\nabla u\|^{2}_{L^{2}}dt\right)^{\frac{3}{4}} \left(\int^{T}_{t_{0}}t^{-8\delta_{-}}dt\right)^{\frac{1}{4}}\leq C\mathcal{H}_{0}.
 		\end{aligned}
 		$$
 		Substituting the estimates of $I_1-I_5$ into the (\ref{d78}), we can deduce 
 		\begin{equation}\label{d79}
 		\begin{aligned}
 		\int^{T}_{t_{0}}\|\nabla^{2} u\|_{L^{m}}dt\leq (1+\|u(t_0)\|_{H^2})\exp\{C\mathcal{H}_{0}\}.
 		\end{aligned}
 		\end{equation}
 		
 		\noindent\textbf{Step 2.} The estimate of $\|\nabla u\|_{L^{1}([t_{0},T];L^n)}.$
 		
 		 By virtue of the Gagliardo-Nirenberg inequality, 
 		 \begin{equation}\label{dd64}
 		 \|t^{(1-\beta)\delta_{-}}\nabla u\|_{L^{2}([t_{0},T];L^{n})}\leq C\|\nabla u\|^{\beta}_{L^{2}([t_{0},T];L^{6})}\|t^{\delta_{-}}\nabla u\|^{1-\beta}_{L^{2}([t_{0},T];L^{2})}
 		 \end{equation}
 		 with
 		 $$
 		 \begin{aligned}\frac{1}{n}=\frac{\beta}{6}+\frac{1-\beta}{2}\quad{\rm or}\quad \beta=\frac{3(n-2)}{2n}.
 		 \end{aligned}
 		 $$
 		 Thanks to Corollary \ref{c-4.2}, we obtain
 		 \begin{equation}\label{dd63}
 		 \|t^{\delta_{-}}\nabla u\|^2_{L^{2}([t_{0},T];L^{2})}\leq C\mathcal{H}_{0},
 		 \end{equation}
 		 Thanks to $u=v+w$ and Proposition \ref{Proposition-5.1}, one has
 		 $$
 		 \int^{T}_{t_{0}}\|\nabla v\|^{2}_{L^{6}}dt\leq \int^{T}_{t_{0}} \|v\|^2_{\dot{B}^{\frac{1}{2}+\frac{3}{p}}_{p,1}}dt\lesssim \|u_0\|^2_{\dot{B}^{-1+\frac{3}{p}}_{p,1}},
 		 $$
 		 Combining (\ref{w1}) and (\ref{d52}), we infer that
 		\begin{equation}
 		\begin{aligned}
 		&\int^{T}_{t_{0}}\|\nabla w\|^{2}_{L^{6}}dt 
 		\lesssim \int^{T}_{t_{0}}\|\nabla^2 w\|^{2}_{L^{2}}dt \lesssim\int^{T}_{t_{0}}\|\partial_{t}w\|^{2}_{L^{2}}dt+\int^{T}_{t_{0}}\|\nabla w\|^{2}_{L^{2}}dt\\&+\sup_{t\in[t_{0},T]}\|\nabla w\|^{4}_{L^{2}}\int^{T}_{t_{0}}\|\nabla w\|^{2}_{L^{2}}dt+\int^{T}_{t_{0}}\|\Delta v\|^2_{L^\infty}dt+\int^{T}_{t_{0}}\|v\|^2_{\dot{B}^{s_1}_{p,1}}dt\lesssim \|u_0\|^2_{\dot{B}^{-1+\frac{3}{p}}_{p,1}},
 		\end{aligned}
 		\end{equation}
 		where we used the embedding inequality $\dot{H}^{1}\hookrightarrow L^{6}$ in the first inequality. Thus, we deduce that
 		\begin{equation}\label{dd65}
 		\int^{T}_{t_{0}}\|\nabla u\|^{2}_{L^{6}}dt\leq \int^{T}_{t_{0}}\|\nabla v\|^{2}_{L^{6}}dt+\int^{T}_{t_{0}}\|\nabla w\|^{2}_{L^{6}}dt\leq C\|u_0\|^2_{\dot{B}^{-1+\frac{3}{p}}_{p,1}}.
 		\end{equation}
 		Substituting (\ref{dd63}) and (\ref{dd65}) into (\ref{dd64}) and taking $n\in]2,\frac{6\delta_{-}}{1+\delta_{-}}[$ with $\delta\in]1/2,3/4[$, we get, by using H\"{o}lder's inequality, that
 		$$
 		\begin{aligned}
 		\|\nabla u\|_{L^{1}(L^{n})}&\leq C	\|t^{\frac{6-n}{2n}\delta_{-}}\nabla u\|_{L^{2}([t_{0},T];L^{n})}\left(\int^{T}_{t_{0}}t^{\frac{n-6}{n}\delta_{-}}dt\right)^{\frac{1}{2}}\\&\leq \|u_0\|^{\frac{3(n-2)}{2n}}_{\dot{B}^{-1+\frac{3}{p}}_{p,1}}\exp\{C\mathcal{H}_{0}\},
 		\end{aligned}
 		$$
 		which along with (\ref{d79}) yields
 		$$
 		\begin{aligned}
 		\|\nabla u\|_{L^{1}([t_{0},T];L^\infty)}&\leq\|\nabla u\|^{\alpha}_{L^{1}([t_{0},T];L^{n})}\|\nabla^{2} u\|^{1-\alpha}_{L^{1}([t_{0},T];L^{m})}\\&\leq \|u_{0}\|^{\frac{3\alpha(n-2)}{2n}}_{\dot{B}^{-1+\frac{3}{p}}_{p,1}}\left(1+\|u(t_0)\|_{H^2}\right)\exp\{C\mathcal{H}_{0}\}.
 		\end{aligned}
 		$$
 		This completes the proof of Lemma \ref{Lemma-5.6}.
 	\end{proof}
 	
 	 With Lemma \ref{Lemma-u1}-\ref{Lemma-5.6} at hand, we are in a position to prove the following Proposition.
 
 \begin{proof}[Proof of  Proposition\ref{Proposition-5.5}]
 	Since $\mu(\rho)$ satisfies
 	\begin{equation}\label{d83}
 	\partial_{t}(\mu(\rho))+u\cdot\nabla \mu(\rho)=0,
 	\end{equation}
 	standard caclulations show that
 	\begin{equation}\label{d84}
 	\frac{d}{dt}\|\nabla\mu(\rho)\|_{L^{r}}\leq r \|\nabla u\|_{L^{\infty}}\|\nabla\mu(\rho)\|_{L^{r}},
 	\end{equation}
 	which together with Gronwall's inequality and (\ref{d75}) yields
 	\begin{equation}\label{d85}
 	\begin{aligned}
 	\sup_{t\in[0, T]}\|\nabla\mu(\rho)\|_{L^{r}}
 	\leq&\|\nabla \mu(\rho_0)\|_{L^{r}}\exp{\{r\int^{t_0}_{0}\|\nabla u\|_{L^{\infty}}dt+r\int^{T}_{t_0}\|\nabla u\|_{L^{\infty}}dt\}}\\
 	\leq& \|\nabla\mu(\rho_0)\|_{L^{r}}\exp\left\{C\|u_0\|_{\dot{B}^{-1+\frac{3}{p}}_{p,1}}+\overline{C}\|u_{0}\|^{\frac{3\alpha(n-2)}{2n}}_{\dot{B}^{-1+\frac{3}{p}}_{p,1}}\exp\{C\mathcal{H}_{0}\}\right\}.
 	\end{aligned}
 	\end{equation}
 	Hence, one has
 	$$\sup_{t\in[t_{0}, T]}\|\nabla\mu(\rho)\|_{L^{r}}\leq 2\|\nabla\mu(\rho_0)\|_{L^{r}},$$
 	provided that
 	\begin{equation}\label{d86}
 	\|u_0\|_{\dot{B}^{-1+\frac{3}{p}}_{p,1}}\leq\varepsilon_{2}\quad{\rm and}\quad\varepsilon_{2}+\overline{C}\varepsilon^{\frac{3\alpha(n-2)}{2n}}_{2}\exp\{\mathcal{H}_0\}\overset{\text{def}}{=}C^{-1}\ln 2.
 	\end{equation}
 	Similarly, we have
 		$$\sup_{t\in[t_{0}, T]}\|\nabla\mu(\rho)\|_{L^{3}}\leq 2\|\nabla\mu(\rho_0)\|_{L^{3}}\leq 2\|\mu^{\prime}(\cdot)\|_{L^\infty}\|\nabla\rho_0\|_{L^3}\leq C\|\rho_0-1\|_{\dot{B}^{\frac{3}{q}}_{q,1}}.$$
 	\par Choosing $\varepsilon=\min\{1,\varepsilon_{0},\varepsilon_{1},\varepsilon_{2}\},$ we directly obtain (\ref{d40}) from (\ref{d84})-(\ref{d86}). The proof of Proposition \ref{Proposition-5.5} is finished.
 \end{proof}
 \subsection{Proof of Theorem \ref{Theorem-1}}
 	
 We first rewrite the momentum equation in $(\ref{a1})_{2}$ as 
\begin{equation}\label{d87}
\partial_{t}u+u\cdot\nabla u-\mu(\rho)\Delta u+\nabla\pi=(1-\rho)(\partial_t u+u\cdot\nabla u)+2\nabla\mu(\rho)\mathbb{D}u.
\end{equation}
Applying the operator $\dot{\Delta}_{j}\mathbb{P}$ to (\ref{d87}) gives
\begin{equation}\label{d88}
\begin{aligned}
&\partial_{t}\dot{\Delta}_{j}u+u\cdot\nabla \dot{\Delta}_{j}u-\operatorname{div}\{\mu(\rho)\dot{\Delta}_{j}\nabla u\}\\=&\dot{\Delta}_j\mathbb{P}\{(1-\rho)(\partial_t u+u\cdot\nabla u)\} +[u\cdot\nabla;\dot{\Delta}_{j}\mathbb{P}]u\\&-[\mu(\rho);\dot{\Delta}_{j}\mathbb{P}]\Delta u-\nabla\mu(\rho)\cdot\dot{\Delta}_{j}\nabla u+\dot{\Delta}_{j}\mathbb{P}(2\nabla\mu(\rho)\mathbb{D}u).
\end{aligned}
\end{equation}
Due to $\operatorname{div}u=0$ and $\mu(\rho)\geq \underline{\mu}$ by multiplying (\ref{d88}) by $|\dot{\Delta}_{j}u|^{p-2}\dot{\Delta}_{j}u$ and then integrating the resulting equality over $\mathbb{R}^{3}$, we obtain
\begin{equation}\label{d89}
\begin{aligned}
\frac{d}{dt}\|\dot{\Delta}_{j} u\|_{L^{p}}+&2^{2j}\|\dot{\Delta}_{j} u\|_{L^{p}}\lesssim \|\dot{\Delta}_j\{(1-\rho)(\partial_t u+u\cdot\nabla u)\}\|_{L^p}\\&+\|[u\cdot\nabla;\dot{\Delta}_{j}\mathbb{P}]u\|_{L^{p}}+\|[\mu(\rho);\dot{\Delta}_{j}\mathbb{P}]\Delta u\|_{L^{p}}\\&+\|\nabla\mu(\rho)\cdot\dot{\Delta}_{j}\nabla u\|_{L^p}+\|\dot{\Delta}_{j}\mathbb{P}(2\nabla\mu(\rho)\mathbb{D}u)\|_{L^p}.
\end{aligned}
\end{equation} 
Intergrating the above inequality over $[t_0,T]$ and multiplying (\ref{d89}) by $2^{j(-1+\frac{3}{p})}$, then summing up the resulting inequality over $j\in\mathbb{Z},$ we achieve 
\begin{equation}\label{d90}
\begin{aligned}
&\|u\|_{\widetilde{L}^{\infty}([t_{0},T];\dot{B}^{-1+\frac{3}{p}}_{p,1})}+\|u\|_{L^{1}([t_{0},T];\dot{B}^{1+\frac{3}{p}}_{p,1})}\\
\lesssim &\|u(t_{0})\|_{\dot{B}^{-1+\frac{3}{p}}_{p,1}}+\|(1-\rho)(\partial_t u+u\cdot\nabla u)\|_{L^{1}([t_{0},T];\dot{B}^{-1+\frac{3}{p}}_{p,1})}\\&+ \sum_{j\in\mathbb{Z}}2^{j(-1+\frac{3}{p})}\|[u\cdot\nabla;\dot{\Delta}_{j}\mathbb{P}]u\|_{L^{1}([t_{0},T];L^{p})}+\sum_{j\in\mathbb{Z}}2^{j(-1+\frac{3}{p})}\|[\mu(\rho);\dot{\Delta}_{j}\mathbb{P}]\Delta u\|_{L^{1}([t_{0},T];L^{p})}\\&+\sum_{j\in\mathbb{Z}}2^{j(-1+\frac{3}{p})}\|\nabla\mu(\rho)\cdot\dot{\Delta}_{j}\nabla u\|_{L^{1}([t_{0},T];L^{p})}+\|\nabla\mu(\rho)\mathbb{D}u\|_{L^{1}([t_{0},T];\dot{B}^{-1+\frac{3}{p}}_{p,1})}.
\end{aligned}
\end{equation} 
\par In what follows,~we shall deal with the right-hand side of~(\ref{d90}). Applying Lemma \ref{Lemma-2.2} gives
$$
\begin{aligned}
&\|(1-\rho)(\partial_t u+u\cdot\nabla u)\|_{L^{1}([t_{0},T];\dot{B}^{-1+\frac{3}{p}}_{p,1})}\\
\lesssim& \|1-\rho\|_{L^\infty([t_0,T];\dot{B}^{\frac{3}{p}}_{p,1})}\|\partial_t u+u\cdot\nabla u\|_{L^{1}([t_{0},T];\dot{B}^{-1+\frac{3}{p}}_{p,1})}.
\end{aligned}$$
Yet thanks to Lemma \ref{Lemma-5.5}, one has
$$
\begin{aligned}
\|\partial_t u\|_{L^{1}([t_{0},T];\dot{B}^{-1+\frac{3}{p}}_{p,1})}
&\leq  \sup_{t\in[t_0,T]}\|t u_t\|^{\frac{1}{2}}_{L^2}(\int^{T}_{t_0}\|t\nabla u_t\|^{2}_{L^2}dt)^{\frac{1}{4}}(\int^{T}_{t_0}t^{-\frac{4}{3}}dt)^{\frac{3}{4}}\\
&\leq\|u(t_0)\|_{H^2} \exp\{C\mathcal{H}_{0}\},
\end{aligned}
$$
and
$$
\begin{aligned}
\|u\cdot\nabla u\|_{L^{1}([t_{0},T];\dot{B}^{-1+\frac{3}{p}}_{p,1})}\lesssim \int^{T}_{t_0}(\|\nabla w\|_{L^2}\|\Delta w\|_{L^2}+\|v\|^{2}_{\dot{B}^{\frac{3}{p}}_{p,1}})dt\lesssim \|u_0\|^2_{\dot{B}^{-1+\frac{3}{p}}_{p,1}}.
\end{aligned}
$$
By virtue of Lemma \ref{Lemma-2.4} and (\ref{d79}) with $m=3$, one has
\begin{equation}\label{d96}
\begin{aligned}
\|1-\rho\|_{\widetilde{L}_{T}^{\infty}(\dot{B}^{\frac{3}{2}}_{2,1})}&\leq\|1-\rho(t_0)\|_{\dot{B}^{\frac{3}{2}}_{2,1}}\exp\left\{C\|\nabla u\|_{\dot{B}^{1}_{3,\infty}\cap L^\infty}\right\}\\
&\leq \|1-\rho(t_0)\|_{\dot{B}^{\frac{3}{2}}_{2,1}}\exp\left\{\overline{C}\exp\{C\mathcal{H}_{0}\}\right\}.
\end{aligned}
\end{equation}
From which, we can deduce that
\begin{equation}\label{d92}
\begin{aligned}
&\|(1-\rho)(\partial_t u+u\cdot\nabla u)\|_{L^{1}([t_{0},T];\dot{B}^{-1+\frac{3}{p}}_{p,1})}\\
\leq& \|1-\rho(t_0)\|_{\dot{B}^{\frac{3}{2}}_{2,1}}\exp\left\{\overline{C}\exp\{C\mathcal{H}_{0}\}\right\}.
\end{aligned}
\end{equation}
Similarly, by virtue of Lemma \ref{Lemma-2.3}, 
\begin{equation}\label{d93}
\sum_{j\in\mathbb{Z}}2^{j(-1+\frac{3}{p})}\|[u\cdot\nabla;\dot{\Delta}_{j}\mathbb{P}]u\|_{L^{1}([t_{0},T];L^{p})}\lesssim\int^{T}_{t_{0}} \|\nabla u\|_{L^{\infty}}\|u\|_{\dot{B}^{-1+\frac{3}{p}}_{p,1}}dt.
\end{equation}
For $[\mu(\rho);\dot{\Delta}_{j}\mathbb{P}]\Delta u,$ homogenous Bony's decomposition implies
$$
[\mu(\rho);\dot{\Delta}_{j}\mathbb{P}]\Delta u=[T_{\mu(\rho)};\dot{\Delta}_{j}\mathbb{P}]\Delta u+T^{'}_{\dot{\Delta}_{j}\Delta u}\mu(\rho)-\dot{\Delta}_{j}\mathbb{P}(T_{\Delta u}\mu(\rho))-\dot{\Delta}_{j}\mathbb{P}\mathcal{R}(\Delta u,\mu(\rho)).$$
It follows again from the above estimate,~which implies that
$$
\begin{aligned}
&\sum_{j\in\mathbb{Z}}2^{j(-1+\frac{3}{p})}\|[T_{\mu(\rho)};\dot{\Delta}_{j}\mathbb{P}]\Delta u\|_{L^{p}}\\
\lesssim&\sum_{j\in\mathbb{Z}}\sum_{|j-k|\leq 4}2^{j(-2+\frac{3}{p})}\|\nabla\dot{S}_{k-1}\mu(\rho)\|_{L^{\infty}}\|\dot{\Delta}_{k}\Delta u\|_{L^{p}}\\
 \lesssim&\sum_{j\in\mathbb{Z}}\sum_{|j-k|\leq 4}2^{(j-k)(-2+\frac{3}{p})}\|\nabla \mu(\rho)\|_{L^{r}}2^{k(-2+\frac{3}{p}+\frac{3}{r})}\|\dot{\Delta}_{k}\Delta u\|_{L^{p}}\\
 \lesssim& \|\nabla \mu(\rho)\|_{L^{r}}\|u\|_{\dot {B}^{\frac{3}{p}+\frac{3}{r}}_{p,1}}.
\end{aligned}
$$
Similarly, one has
$$
\begin{aligned}
&\sum_{j\in\mathbb{Z}}2^{j(-1+\frac{3}{p})}\|T^{'}_{\dot{\Delta}_{j}\Delta u} \mu(\rho)\|_{L^{p}}\\
\lesssim&\sum_{j\in\mathbb{Z}}\sum_{k\geq j-2}2^{j(-1+\frac{3}{p})}\|\dot{S}_{k+2}\dot{\Delta}_{j}\Delta u\|_{L^{p}}\|\dot{\Delta}_{k} \mu(\rho)\|_{L^{\infty}}\\
\lesssim&\sum_{j\in\mathbb{Z}}2^{j(-2+\frac{3}{p}+\frac{3}{r})}\|\dot{\Delta}_{j}\Delta u\|_{L^{p}}\sum_{k\geq j-2}2^{(j-k)(1-\frac{3}{r})}\|\dot{\Delta}_{k}\nabla \mu(\rho)\|_{L^{r}}\\
\lesssim &\|u\|_{\dot {B}^{\frac{3}{p}+\frac{3}{r}}_{p,1}}\|\nabla \mu(\rho)\|_{L^{r}}.
\end{aligned}
$$
For the term $\dot{\Delta}_{j}\mathbb{P}(T_{\Delta u} \mu(\rho))$, we split the estimate into two cases. For the first case that $p<r$, we get
$$
\begin{aligned}
&\sum_{j\in\mathbb{Z}}2^{j(-1+\frac{3}{p})}\|\dot{\Delta}_{j}\mathbb{P}(T_{\Delta u} \mu(\rho))\|_{L^{p}}\\
\lesssim&\sum_{j\in\mathbb{Z}}\sum_{|k-j|\leq 4}2^{(j-k)(-1+\frac{3}{p})}2^{k(-2+\frac{3}{p})}\|\dot{S}_{k-1}\Delta u\|_{L^{p^{*}}}2^{k}\|\dot{\Delta}_{k}\mu(\rho)\|_{L^{r}}\\
\lesssim &\|u\|_{\dot {B}^{\frac{3}{p}+\frac{3}{r}}_{p,1}}\|\nabla \mu(\rho)\|_{L^{r}},
\end{aligned}
$$
where $\frac{1}{p}=\frac{1}{p^{*}}+\frac{1}{r}$ and we have used the embedding inequality $\dot {B}^{\frac{3}{p}+\frac{3}{r}}_{p,1}\hookrightarrow\dot{B}^{\frac{3}{p}}_{p^{*},1}$ in the second inequality. Along the same line, for the case that $p\geq r>3,$ we have
$$
\begin{aligned}
&\sum_{j\in\mathbb{Z}}2^{j(-1+\frac{3}{p})}\|\dot{\Delta}_{j}\mathbb{P}(T_{\Delta u} \mu(\rho))\|_{L^{p}}\\
\lesssim&\sum_{j\in\mathbb{Z}}\sum_{|k-j|\leq 4}2^{(j-k)(-1+\frac{3}{p})}2^{k(-1+\frac{3}{p})}\|\dot{S}_{k-1}\Delta u\|_{L^{p}}\|\dot{\Delta}_{k}\mu(\rho)\|_{L^{\infty}}\\
\lesssim&\sum_{j\in\mathbb{Z}}\sum_{|k-j|\leq 4}2^{(j-k)(-1+\frac{3}{p})}2^{k(-2+\frac{3}{p}+\frac{3}{r})}\|\dot{S}_{k-1}\Delta u\|_{L^{p}}\|\dot{\Delta}_{k}\nabla\mu(\rho)\|_{L^{r}}\\
\lesssim &\|u\|_{\dot {B}^{\frac{3}{p}+\frac{3}{r}}_{p,1}}\|\nabla \mu(\rho)\|_{L^{r}}.
\end{aligned}
$$
The same estimate holds for $\dot{\Delta}_{j}\mathbb{P}\mathcal{R}(\Delta u,\mu(\rho)),$ note that $r<9,$
$$
\begin{aligned}
\sum_{j\in\mathbb{Z}}2^{j(-1+\frac{3}{p})}\|\dot{\Delta}_{j}\mathbb{P}\mathcal{R}(\Delta u,\mu(\rho))\|_{L^{p}}
\lesssim&\sum_{j\in\mathbb{Z}}2^{j(-1+\frac{3}{p}+\frac{3}{r})}\|\dot{\Delta}_{j}\mathcal{R}(\Delta u,\mu(\rho))\|_{L^{r^{*}}}\\
\lesssim&\sum_{j\in\mathbb{Z}}\sum_{k\geq j-3}2^{j(-1+\frac{3}{p}+\frac{3}{r})}\|\dot{\Delta}_{k}\Delta u\|_{L^{p}}\|\dot{\Delta}_{k}\mu(\rho)\|_{L^{r}}\\
\lesssim &\|u\|_{\dot {B}^{\frac{3}{p}+\frac{3}{r}}_{p,1}}\|\nabla \mu(\rho)\|_{L^{r}},
\end{aligned}
$$
where $\frac{1}{r^{*}}=\frac{1}{p}+\frac{1}{r}.$ Hence, we deduce that
\begin{equation}\label{d94}
\begin{aligned}
&\sum_{j\in\mathbb{Z}}2^{j(-1+\frac{3}{p})}\|[\mu(\rho);\dot{\Delta}_{j}\mathbb{P}]\Delta u\|_{L^{1}([t_{0},T];L^{p})} \lesssim\int^{T}_{t_{0}}\|u\|_{\dot {B}^{\frac{3}{p}+\frac{3}{r}}_{p,1}}\|\nabla \mu(\rho)\|_{L^{r}}dt.
\end{aligned}
\end{equation}
Along the same line, one has
$$
\begin{aligned}
&\sum_{j\in\mathbb{Z}}2^{j(-1+\frac{3}{p})}\|\nabla\mu(\rho)\cdot\dot{\Delta}_{j}\nabla u\|_{L^{1}([t_{0},T];L^{p})}+\|\nabla\mu(\rho)\mathbb{D}u\|_{L^{1}([t_{0},T];\dot{B}^{-1+\frac{3}{p}}_{p,1})}\\
\lesssim&\int^{T}_{t_{0}}\|u\|_{\dot {B}^{\frac{3}{p}+\frac{3}{r}}_{p,1}}\|\nabla \mu(\rho)\|_{L^{r}}dt+\int^{T}_{t_{0}}\|\nabla u\|_{L^\infty}\|1-\rho\|_{\dot {B}^{\frac{3}{2}}_{2,1}}dt.
\end{aligned}
$$
Applying interpolation inequality $\|u\|_{\dot {B}^{\frac{3}{p}+\frac{3}{r}}_{p,1}}\lesssim \|u\|^{\frac{r-3}{2r}}_{\dot{B}^{-1+\frac{3}{p}}_{p,1}}\|u\|^{\frac{r+3}{2r}}_{\dot{B}^{1+\frac{3}{p}}_{p,1}}$ and Lemma \ref{Lemma-5.6}, we arrive at
\begin{equation}\label{d95}
\begin{aligned}
&\sum_{j\in\mathbb{Z}}2^{j(-1+\frac{3}{p})}\|\nabla\mu(\rho)\cdot\dot{\Delta}_{j}\nabla u\|_{L^{1}([t_{0},T];L^{p})}+\|\nabla\mu(\rho)\mathbb{D}u\|_{L^{1}([t_{0},T];\dot{B}^{-1+\frac{3}{p}}_{p,1})}\\
\lesssim&\int^{T}_{t_{0}}\|u\|^{\frac{r-3}{2r}}_{\dot{B}^{-1+\frac{3}{p}}_{p,1}}\|u\|^{\frac{r+3}{2r}}_{\dot{B}^{1+\frac{3}{p}}_{p,1}}\|\nabla \mu(\rho)\|_{L^{r}}dt+\|1-\rho(t_0)\|_{\dot{B}^{\frac{3}{2}}_{2,1}}\exp\left\{\overline{C}\exp\{C\mathcal{H}_{0}\}\right\}.
\end{aligned}
\end{equation}
\par Plugging (\ref{d92}) (\ref{d93}) (\ref{d94}) and (\ref{d95}) into (\ref{d90}) and using Young's inequality, we obtain
\begin{equation}
\begin{aligned}
\|u&\|_{\widetilde{L}^{\infty}([t_{0},T];\dot{B}^{-1+\frac{3}{p}}_{p,1})}+\|u\|_{L^{1}([t_{0},T];\dot{B}^{1+\frac{3}{p}}_{p,1})}\leq\|u(t_0)\|_{\dot{B}^{-1+\frac{3}{p}}_{p,1}}+\int^{T}_{t_{0}} \|\nabla u\|_{L^{\infty}}\|u\|_{\dot{B}^{-1+\frac{3}{p}}_{p,1}}dt\\
&\quad+\int^{T}_{t_{0}}\|u\|_{\dot{B}^{-1+\frac{3}{p}}_{p,1}}\|\nabla \mu(\rho)\|^{\frac{2r}{r-3}}_{L^{r}}dt+\|1-\rho(t_0)\|_{\dot{B}^{\frac{3}{2}}_{2,1}}\exp\left\{\overline{C}\exp\{C\mathcal{H}_{0}\}\right\},
\end{aligned}
\end{equation}
which together with Gronwall's inequality yields 
\begin{equation}\label{d97}
\begin{aligned}
&\|u\|_{\widetilde{L}^{\infty}([t_{0},T];\dot{B}^{-1+\frac{3}{p}}_{p,1})}+\|u\|_{L^{1}([t_{0},T];\dot{B}^{1+\frac{3}{p}}_{p,1})}\\
\leq&\left(\|u(t_0)\|_{\dot{B}^{-1+\frac{3}{p}}_{p,1}}+\|1-\rho(t_0)\|_{\dot{B}^{\frac{3}{2}}_{2,1}}\right)\exp\left\{\overline{C}\{1+T\|\nabla \mu(\rho_0)\|^{\frac{2r}{r-3}}_{L^{r}}\}\exp\left\{C\mathcal{H}_{0}\right\}\right\}.
\end{aligned}
\end{equation} 
The similar estimates hold for $\nabla\pi$ and $u_t.$ Hence, we completes the proof of Theorem \ref{Theorem-1}.


\begin{thebibliography}{00}
		\bibitem{2007A} H. Abidi, Navier-Stokes equations with variable density and viscosity in the critical space. Rev. Mat. Iberoam. 23 (2007) 537-586.
		\bibitem{2011AGZ} H. Abidi, G. L. Gui, P. Zhang, On the decay and stability to global solutions of the 3D inhomogeneous Navier-Stokes equations. Commun. Pure Appl. Math. 64 (2011) 832-881.
		\bibitem{2012AGZ} H. Abidi, G. L. Gui, P. Zhang, On the wellposedness of 3D inhomogeneous Navier-Stokes equations in the critical spaces. Arch. Ration. Mech. Anal. 204 (2012) 189-230.
		\bibitem{2013AGZ}  H. Abidi, G. L. Gui, P. Zhang, Well-posedness of 3D inhomogeneous Navier-Stokes equations with highly oscillatory initial velocity field. J. Math. Pures Appl. 100 (2013) 166-203.
		\bibitem{2023AGZ}  H. Abidi, G. L. Gui, P. Zhang, On the global existence and uniqueness of solution to 2D inhomogeneous incompressible Navier-Stokes equations in critical spaces. https://doi.org/10.48550/arXiv.2312.03990.
	    \bibitem{2015az} H. Abidi, P. Zhang, On the global well-posedness of 2D density-dependent Navier-Stokes system with variable viscosity. J. Differ. Equ. 259 (2015) 3755-3802.
		\bibitem{2015AZ} H. Abidi, P. Zhang, Global well-posedness of 3D density-dependent Navier-Stokes system with variable viscosity. Sci. China Math. 58 (6) (2015) 1129-1150.
		\bibitem{2011BCD} H. Bahouri, J. Y. Chemin, R. Danchin, Fourier Analysis and Nonlinear Partial Differential Equations. Grundlehren Math. Wiss., vol. 343, Springer-Verlag, Berlin, Heidelberg (2011).
		\bibitem{2017B} C. Burtea, Optimal well-posedness for the inhomogeneous incompressible Navier-Stokes system with
		general viscosity, Analysis and PDE 10 (2) (2017) 439-479.

		\bibitem{2004CK} Y. Cho, H. Kim, Unique solvability for the density-dependent Navier-Stokes equations. Nonlinear Anal. 59 (4) (2004) 465-489.
		\bibitem{2013CHW} W. Craig, X. D. Huang, Y. Wang, Global strong solutions for 3D nonhomogeneous incompressible Navier-Stokes equations. J. Math. Fluid Mech. 15 (2013) 747-758.
		\bibitem{2003D} R. Danchin, Density-dependent incompressible viscous fluids in critical spaces, Proc. Roy. Soc. Edinburgh Sect. A 133 (2003) 1311-1334.
		\bibitem{2004D} R. Danchin, Local and global well-posedness results for flows of inhomogeneous viscous fluids. Adv. Differ. Equ. 9 (2004) 353-386.
		\bibitem{2010D} R. Danchin, On the well-posedness of the incompressible density-dependent euler equations in the
		$L^{p}$ framework, J. Differ. Equ. 248 (8) (2010) 2130-2170.
		\bibitem{2012DM} R. Danchin, P.B. Mucha, A Lagrangian approach for the incompressible Navier-Stokes equations with variable
		density, Comm. Pure Appl. Math. 65 (2012) 1458-1480.
		\bibitem{2022a} R. Danchin, S. Wang, Global unique solutions for the inhomogeneous Navier-Stokes equation with only bounded density in critical regularity spaces, preprint, 2022. 
		\bibitem{1997D} B. Desjardins, Regularity results for two-dimensional flows of multiphase viscous fluids. Arch. Ration. Mech. Anal. 137 (1997) 135-158.
		\bibitem{2009GZ} G. Gui, P. Zhang, Global smooth solutions to the 2-D inhomogeneous Navier-Stokes Equations with variable viscosity. Chin Ann Math Ser B (2009) 30: 607–630.
		\bibitem{2021HLL} C. He, J. Li, B. Q. L${\rm \ddot{u}}$, Global well-posedness and exponential stability of 3D Navier-Stokes equations with
		density-dependent viscosity and vacuum in unbounded domains.
		Arch. Rational Mech. Anal. 239 (2021) 1809-1835.
	    \bibitem{2013HPZ} J. Huang, M. Paicu, P. Zhang, Global solutions to 2-D inhomogeneous Navier-Stokes system with general velocity. J. Math Pure Appl (2013) 100: 806-831.
	    \bibitem{2014HW} X. D. Huang, Y. Wang, Global strong solution with vacuum to the two-dimensional
		density-dependent Navier-Stokes system. SIAM J. Math. Appl. 46 (2014) 1771-1788.
		\bibitem{2015HW} X. D. Huang, Y. Wang, Global strong solution of 3D inhomogeneous Navier-Stokes equations with density-dependent viscosity. J. Differ. Equ. 259 (2015) 1606-1627.
		\bibitem{1974K} A.V. Kazhikov: Solvability of the initial-boundary value problem for the equations of
		the motion of an inhomogeneous viscous incompressible fluid. (Russian). Dokl. Akad. Nauk SSSR 216 (1974) 1008-1010.
	    \bibitem{1996L} P. L. Lions, Mathematical Topics in Fluid Mechanics, Vol. I: Incompressible Models. Oxford University Press, Oxford (1996).
		\bibitem{2019LS} B. Q. L${\rm \ddot{u}}$, S.S. Song, On local strong solutions to the three-dimensional nonhomogeneous incompressible Navier-Stokes equations with density-dependent viscosity and	vacuum. Nonlinear Anal. Real World Appl. 46 (2019) 58-81.
		\bibitem{NW2023NS} D. J. Niu, L. Wang, On the global well-posedness of 3D inhomogeneous incompressible Navier-Stokes system with density-dependent viscosity, (2023) Preprint.
		\bibitem{2022q} C. Y. Qian, Y. Qu, Global well-posedness for 3D incompressible inhomogeneous asymmetric fluids with
		density-dependent viscosity, J. Differ. Equ. 306 (2022) 333-402.
		\bibitem{2017j} X. Zhai, Z. Yin, Global well-posedness for the 3d incompressible inhomogeneous Navier-Stokes equations and MHD equations, J. Differ. Equ. 262 (3) (2017) 1359-1412.
		\bibitem{2015Z} J. W. Zhang, Global well-posedness for the incompressible Navier-Stokes equations
		with density-dependent viscosity coefficient. J. Differ. Equ. 259 (2015) 1722-1742.
		\bibitem{2020z} P. Zhang, Global Fujita-Kato solution of 3D inhomogeneous incompressible Navier-Stokes system, adv. math. 363 (2020) 107007.
	\end{thebibliography}
\end{document}